\documentclass{amsart}

\usepackage[foot]{amsaddr}
\makeatletter
\def\BState{\State\hskip-\ALG@thistlm}
\makeatother

\usepackage{color}

\usepackage{enumitem}
\usepackage{textcomp}

\usepackage{stmaryrd}

\setlist[itemize]{leftmargin=*}

\usepackage{amsthm}
\usepackage{amsmath}
\usepackage{mathtools}
\usepackage{amssymb}
\usepackage{commath}
\usepackage{bigints}
\usepackage{romanbar}
\usepackage{algorithm2e}
\usepackage{pgfplotstable}
\setlist[itemize]{leftmargin=*}

\newtheorem{thm}{Theorem}
\newtheorem{prop}{Proposition}
\newtheorem{cor}{Corollary}
\newtheorem{remark}{Remark}

\newtheorem{lemma}{Lemma}
\newcommand{\vc}{\mathbf{c}}
\newcommand{\ve}{\mathbf{e}}

\newcommand{\vx}{\mathbf{x}}
\newcommand{\vy}{\mathbf{y}}

\newcommand{\vu}{\mathbf{u}}

\newcommand{\vR}{\mathbf{R}}
\newcommand{\vv}{\mathbf{v}}

\newcommand{\vt}{\mathbf{t}}
\newcommand{\vn}{\mathbf{n}}
\newcommand{\vm}{\mathbf{m}}

\newcommand{\vnu}{\boldsymbol{\nu}}

\newcommand{\vF}{\mathbf{F}}
\newcommand{\vL}{\mathbf{L}}

\newcommand{\vC}{\mathbf{C}}

\newcommand{\vb}{\mathbf{b}}

\newcommand{\vxl}{\vx_\lambda}

\newcommand{\jump}[1]{[ #1 ]} % jump
\newcommand{\Th}{\mathcal{T}_h} % mesh
\newcommand{\Eh}{\mathcal{E}_h} % set of edges
\newcommand{\E}{\mathcal{E}} % set of edges
\newcommand{\V}{\mathbb{V}} % one element of \Th
\newcommand{\R}{\mathbb{R}} % real numbers
\newcommand{\bz}{\boldsymbol{0}}
\newcommand{\calO}{\mathcal{O}}

\newcommand{\Id}{\mathbf{I}} % {\normalfont{\Romanbar{1}}}
\newcommand{\I}{\normalfont{\Romanbar{1}}}
\newcommand{\II}{\normalfont{\Romanbar{2}}} % second fundamental form
\newcommand{\tr}{{\rm tr}} % trace
\newcommand{\curl}{{\rm curl}} 
\newcommand{\argmin}{{\rm argmin}} 
\newcommand{\diam}{{\rm diam}} 
\newcommand{\wt}{\widetilde}
\newcommand{\tol}{\textrm{tol}}

\def\restriction#1#2{\mathchoice
	{\setbox1\hbox{${\displaystyle #1}_{\scriptstyle #2}$}
		\restrictionaux{#1}{#2}}
	{\setbox1\hbox{${\textstyle #1}_{\scriptstyle #2}$}
		\restrictionaux{#1}{#2}}
	{\setbox1\hbox{${\scriptstyle #1}_{\scriptscriptstyle #2}$}
		\restrictionaux{#1}{#2}}
	{\setbox1\hbox{${\scriptscriptstyle #1}_{\scriptscriptstyle #2}$}
		\restrictionaux{#1}{#2}}}
\def\restrictionaux#1#2{{#1\,\smash{\vrule height .8\ht1 depth .85\dp1}}_{\,#2}} % restriction

\def\lucas{\textcolor{black}}
\def\shuo{\textcolor{black}}

\title[Membrane model for liquid crystal polymer networks]{Reduced Membrane Model for Liquid Crystal Polymer Networks: Asymptotics and Computation}
\author{Lucas Bouck}
\address[Lucas Bouck]{\lucas{Department of Mathematics, Carnegie Mellon University,
		Pittsburgh, Pennsylvania 15213, USA.}}
\email{\lucas{lbouck@andrew.cmu.edu}}
\author{Ricardo H. Nochetto}
\address[Ricardo H. Nochetto]{Department of Mathematics and Institute for Physical Science
		and Technology \\ University of Maryland,
		College Park, Maryland 20742, USA.}
\email{rhn@umd.edu}
\author{Shuo Yang}
\address[Shuo Yang]{\shuo{Beijing Institute of Mathematical Sciences and Applications}, Beijing 101408, China.}
\email{shuoyang@bimsa.cn}

\date{\today}

\begin{document}
\maketitle
\begin{abstract}
We examine a reduced membrane model of liquid crystal polymer networks (LCNs) via asymptotics and computation. This model requires solving a minimization problem for a non-convex stretching energy. We show a formal asymptotic derivation of the $2D$ membrane model from $3D$ rubber elasticity. We construct approximate solutions with point defects. We design a finite element method with regularization, and propose a nonlinear gradient flow with Newton inner iteration to solve the non-convex discrete minimization problem. We present numerical simulations of practical interests to illustrate the ability of the model and our method to capture rich physical phenomena.
\end{abstract}

%%%%%%%%%%%%%%%%%%%%%%%%%%%%%%%%%%%%%%%%%%%%%%%%%%%%%%%%%%%%%%%%%%%%%%%%%%%%%%%%%%
\section{Introduction}
%%%%%%%%%%%%%%%%%%%%%%%%%%%%%%%%%%%%%%%%%%%%%%%%%%%%%%%%%%%%%%%%%%%%%%%%%%%%%%%%%%
Liquid crystal polymer networks (LCNs) are materials that combine elastomeric polymer networks with mesogens (compounds that display liquid crystal properties). The long rod-like molecules of liquid crystals are \emph{densely} crosslinked with the elastomeric polymer network. This contrasts with liquid crystal elastomers (LCEs), whose crosslinks are less dense. The orientation of the liquid crystal (LC) molecules can be represented by a director. The orientation of the director influences deformation of materials when actuated. Common modes of actuation are heating \cite{aharoni2018universal, ware2015voxelated} and light \cite{camacho2004fast,mcconney2013topography}. In this study, we focus on such actuated deformations of LCNs.

LCNs are one of many possible materials that enable spontaneous mechanical motion under a stimulus. This has been referred to as ``the material is the machine'' \cite{bhattacharya2005material}. Due to this feature, engineers create soft robots using LCNs/LCEs materials, such as thermo-responsive micro-robots with autonomous locomotion in unstructured environments \cite{zhao2022twisting}, soft materials that ``swim'' away from light \cite{camacho2004fast}, and LCN actuators that can lift an object tens of times its weight \cite{ware2015voxelated}. They offer abundant application prospects, for instance in the design of biomedical devices \cite{li2006artificial,hebert1997dynamics}.

Since the deformation of LCNs depend on the orientation of the nematic director, the director can be blueprinted or programmed so that the materials achieve desired shapes \cite{aharoni2018universal, white2015programmable, modes2011blueprinting, warner2020topographic}. Some methods to program the orientation of the liquid crystals include mechanical alignment \cite{camacho2004fast}, photoalignment \cite{ware2015voxelated}, and additive manufacturing \cite{kotikian20183d}, which is a subset of 4D printing \cite{kuang2019advances}. Even if the director is constant throughout the material, interesting shapes may occur due to nonuniform actuation. An example of nonuniform actuation from light is the LCE swimmer in \cite{camacho2004fast}. Two reviews of experimental work on LCEs/LCNs can be found in \cite{white2015programmable, mccracken2020materials}.

For 3D bodies, one of the most accepted elastic energies for modeling the interaction of the material deformation with the LCs is known as the {\it trace formula} \cite{bladon1994deformation, warner2007liquid,warner2003thermal}, although other types of elastic energies have been proposed \cite{desimone2009elastic}. Depending on the density of crosslinks, the director field may be either totally free \cite{cesana2015effective, desimone2002macroscopic} or subject to a Frank elasticity term \cite{barchiesi2015frank, luo2012numerical, plucinsky2018actuation, bartels2022nonlinear}. This kind of blueprinted configuration describes the situation where the LCs are unconstrained or constrained only on a low-level by rubbery polymers.
On the other hand, the LCs may also be frozen into the material via a direct algebraic constraint \cite{ozenda2020blend, cirak2014computational}, \lucas{an approach that we also follow; see our Eq.\ \eqref{eq:kinematic-cons}.
%such as \eqref{eq:kinematic-cons}, or 
The} director's deviation from \eqref{eq:kinematic-cons} may be penalized with a nonideal energy contribution \cite{plucinsky2018actuation,verwey1997compositional}. 
Although the LCs may also be subject to a Frank energy term, a key modeling difference between LCE/LCNs and nematic LCs is that the former are constrained by the rubber. It is known that higher degree defects are unstable \cite{brezis1986harmonic} for the one constant Frank model of nematic LCs. \lucas{However in LCNs, higher degree defects do not split apart due to the constraining nature of the polymer network. We refer to \cite{mcconney2013topography} for blueprinted defects with degrees up to order $10$ in LCNs.}

%and our simulations in Sections \ref{sec:dir_defect} and \ref{sec:eff-reg}.

The energy scaling with respect to thickness in models of thin $3D$ elastic bodies dictates 2D models of LCNs/LCEs. If the energy is scaled linearly with the thickness, the resulting model is a membrane model: the energy is a function of the first fundamental form of the deformed surface and encodes stretching. Works that studied membrane models include \cite{cesana2015effective, ozenda2020blend, conti2018adaptive}. For LCNs, the first fundamental form of zero stretching energy states satisfy a pointwise metric condition. Extensive work dedicated to examining configurations that satisfy this metric condition include \cite{modes2011gaussian, mostajeran2015curvature, plucinsky2018patterning, plucinsky2016programming, plucinsky2017deformations, modes2011blueprinting, warner2018nematic, aharoni2014geometry, mostajeran2016encoding}. For a review of these techniques, we refer to \cite{warner2020topographic}. The second common scaling is a cubic scaling in the thickness, and results in a plate model driven by bending. 
The metric condition giving zero stretching energy becomes a constraint in the bending model. Some existing bending models include theory derived via formal asymptotics \cite{ozenda2020blend}, a von Karman plate model derived in \cite{mihai2020plate} using asymptotics, a rigorous Gamma convergence theory for a model of bilayer materials composed of LCEs and a classical isotropic elastic plate \cite{bartels2022nonlinear}, or a plate model where the LC dramatically changes its orientation through the thickness \cite{agostiniani2020rigorous}. Moreover, reduced $1D$ models for LCNs/LCEs have been explored as well; we refer to \cite{bartels2022rods} for a rod model and to \cite{agostiniani2017shape, singh2022ribbon} for ribbon models.

The computation of LCEs/LCNs has received considerable attention in recent years. Publications include computations of various membrane models \cite{plucinsky2017deformations}, a membrane model with regularization \cite{cirak2014computational}, a bending model of LCE bilayer structure \cite{bartels2022nonlinear}, a relevant 2D model for LCEs \cite{luo2012numerical}, 3D models \cite{conti2002soft,chung2017finite}, and LCE rods \cite{bartels2022rods}. Paper \cite{luo2012numerical} proves well-posedness of a mixed method for a 2D model with Frank-Oseen regularization. 

The goal of this paper is to predict actuated equilibrium shapes of thin LCN membranes using a finite element method (FEM). We discretize a membrane energy of LCNs using piecewise linear finite elements and add a numerical regularization that mimics a higher order bending energy. To solve the discrete minimization problem, we design a nonlinear gradient flow with an embedded Newton method.
%We then apply this method to predict shapes of LCN membranes.
Our FEM is able to predict configurations of LCNs of practical significance, whose solutions are hard or impossible to derive by hand. We present salient examples in Section \ref{sec:incompatible-s} of LCNs with preferred discontinuous metric. We complement the numerical study with the derivation of the LCN model via Kirchhoff-Love asymptotics, and the development of a new formal asymptotic method to approximate shapes of membranes that arise from higher order defects, which we discuss first.

Our companion paper \cite{bouck2022NA} provides a numerical analysis of the finite element method (FEM) described in this article. We refer to Section \ref{sec:contributions} for a list of our main contributions and outline.

%---------------------------------------------------------------------------------
\subsection{3D elastic energy: Neo-classical energy}
%---------------------------------------------------------------------------------
We are concerned with thin films of LCNs. Slender materials are usually modeled as $3D$ hyper-elastic bodies $\mathcal{B}:=\Omega\times(-t/2,t/2)$, with $\Omega\subset \mathbb{R}^2$ being a bounded Lipschitz domain and $t>0$ being a small thickness parameter. We denote by $\vu:\mathcal{B}\to\mathbb{R}^3$ the 3D deformation and by $\vF:=\nabla\vu\in\mathbb{R}^{3\times3}$ the deformation gradient of the LCNs material. 

We denote by $\vm:\mathcal{B}\to \mathbb{S}^2$ the \textit{blueprinted} nematic director field on the reference configuration and by $\vn:\mathcal{B}\to\mathbb{S}^2$ the director field on the deformed configuration. The former is dictated by construction of the LCNs material, whereas the latter obeys an equation that depends on the density of crosslinks between the mesogens and polymer network. LCNs (also called liquid crystal glasses) have moderate to dense crosslinks whereas liquid crystals elastomers (LCEs) have low crosslinks \cite{white2015programmable}. In this paper, we focus on LCNs and leave a numerical study of LCEs for future research. Mathematically, the strong coupling in LCNs is expressed in terms of the following {\it kinematic constraint} between $\vm$ and $\vn$ \cite{ozenda2020blend}: 
\begin{equation}\label{eq:kinematic-cons}
\vn:=\frac{\vF\vm}{|\vF\vm|}.
\end{equation}
In contrast to LCEs \cite{warner2007liquid,bartels2022nonlinear}, $\vn$ is not a free variable but rather a frozen director field for LCNs \cite{cirak2014computational}. For LCEs the energy density may be minimized over $\vn$ first and next over $\vF$, like in \cite{cesana2015effective, conti2002soft}, or a Frank elastic energy for $\vn$ may be introduced (c.f.\ \cite{barchiesi2015frank,bartels2022nonlinear,luo2012numerical}). Moreover, we note that a director field description may not be the only choice for modeling LC components. One can also formulate a model with $Q$-tensor descriptions like in \cite{calderer2015landau}.

The LC effect on the material is governed by the so-called step-length tensors
\begin{equation}\label{eq:def-Lm}
\vL_{\vm}:=(s_0+1)^{-1/3}(\Id_3+s_0\vm\otimes\vm) 
\end{equation}
in the {\it reference} configuration, and
\begin{equation}\label{eq:def-Ln}
\vL_{\vn}:=(s+1)^{-1/3}(\Id_3+s\vn\otimes\vn) 
\end{equation}
in the {\it deformed} configuration. Both of these are uniaxial tensor fields that exhibit the typical head-to-tail symmetry commonly observed in LCs. Our definition of these step length tensors follows the notation of \cite{ozenda2020blend, nguyen2017theory}. The tensor $\vL_{\vm}$ can be related to the more commonly encountered step length tensor $\ell_\perp^0\Id_3 +(\ell_\parallel^0-\ell_\perp^0)\vm\otimes\vm$ \cite[Eq. (13)]{corbett2008polarization} by setting $\ell_\perp^0 = (s_0+1)^{-1/3}$ and $\ell_\parallel^0 = (s_0+1)^{2/3}$. These step length tensors measure the anisotropy contributed by nematogenic molecular units to nematic elastomers/networks, which are isotropic solids with a fluid-like anisotropic ordering.
Moreover, $s_0,s\in L^\infty(\Omega)$ are nematic order parameters that refer to the reference configuration and deformed configuration respectively. They are typically constant and depend on temperature, but may also vary in $\Omega$ if the liquid crystal polymers are actuated non-uniformly. These parameters have a physical range
\[
-1 < s_0,s \le C < \infty.
\]
Consequently, both $\vL_{\vm}$ and $\vL_{\vn}$ are SPD tensor fields, which reduce to the identity matrix, i.e. $\vL_{\vm} = \vL_{\vn} = \Id_3\in\R^{3\times3}$, provided $s=s_0=0$ (no actuation).

The neo-classical energy density for incompressible nematic elastomers/networks has been proposed by Bladon, Warner and Terentjev in \cite{bladon1994deformation, warner2007liquid,warner2003thermal} and reads
\begin{equation}\label{eq:neo-classical}
W_{3D}(\vx,\vF):= \tr \big(\vF^T\vL_{\vn}^{-1}\vF\vL_{\vm} \big) - 3;
\end{equation}
this energy depends explicitly on the space variable $\vx:=(\vx',x_3):=(x_1,x_2,x_3)\in\mathcal{B}$ due to the dependence of $\vm,\vn,s,s_0$ on $\vx$.
The energy \eqref{eq:neo-classical} can be rewritten as the neo-Hookean energy density
\begin{equation}\label{eq:neo-hookian}
W_{3D}(\vx,\vF)= \big|\vL_{\vn}^{-1/2}\vF\vL_{\vm}^{1/2}\big|^2- 3,
\end{equation}
and reduces to the classical neo-Hookean energy density for rubber-like materials $W_{3D}(\vF) = |\vF|^2- 3$ provided $s=s_0=0$. Moreover, the material is assumed to be incompressible, i.e,
\begin{equation}\label{eq:incompressible}
\det \vF =1.
\end{equation}
The 3D energy density $W_{3D}$ is {\it non-degenerate}, namely 
\begin{equation}\label{eq:nondegeneracy}
W_{3D}(\vx, \vF)\geq \mathrm{dist} \big(\vL_{\vn}^{-1/2}\vF\vL_{\vm}^{1/2},SO(3) \big)^2\ge0 
\end{equation}
for all $\vF\in \mathbb{R}^{3\times3}$ such that $\det \vF = 1$.
We refer to \cite[Appendix A]{plucinsky2018actuation} and \cite{bouck2022NA} for a proof of this fundamental property. 

The 3D elastic energy is given in terms of the energy densities \eqref{eq:neo-classical} or \eqref{eq:neo-hookian} by
\begin{equation}\label{eq:3D-energy}
E_{3D}[\vu]=\int_{-t/2}^{t/2}\int_{\Omega}W_{3D}(\vx,\nabla\vu) \, d\vx'dx_3,
\end{equation}
where $\vx'\in\Omega$, $x_3\in(-t/2,t/2)$, and $ \det \nabla\vu = 1$.

%-------------------------------------------------------------------------------------
\subsection{Model reduction}
%-------------------------------------------------------------------------------------
We assume the $3D$ blueprinted director field $\vm=(\wt\vm,0):\mathcal{B}\to \mathbb{S}^2$ is planar and it depends only on $\vx'$; hence, with a slight abuse of notation, we identify $\vm$ and $\widetilde\vm$ and let $\vm:\Omega\to\mathbb{S}^1$ be the $2D$ \emph{blueprinted} director field. We denote by $\vm_\perp:\Omega\to\mathbb{S}^1$ a director field perpendicular to $\vm$ everywhere in $\Omega$, and by $\vy:\Omega\to \mathbb{R}^3$ the deformation of 2D midplane $\Omega$.

The $2D$ membrane model requires solving the following minimization problem: find $\vy^{\ast}\in H^1(\Omega; \mathbb{R}^3)$ such that
\begin{equation}\label{eq:stretching-energy}
\vy^{\ast} = \argmin_{\vy\in H^1(\Omega; \mathbb{R}^3)}E_{str}[\vy],\quad E_{str}[\vy]:=\int_\Omega W_{str}(\vx',\nabla\vy)d\vx',
\end{equation}
where $W_{str}$ is a stretching energy density that is only a function of $\vx'\in\Omega$ and of the {\it first fundamental form} $\I[\vy] := \nabla\vy^T\nabla\vy$. It is given by
\begin{equation}\label{eq:stretching-energy-density}
W_{str}(\vx',\nabla\vy) := \lambda \left[\frac{1}{J[\vy]}+\frac{1}{s+1}\left(\text{tr} \, \I[\vy] +  s_0 C_\vm[\vy] + s \frac{J[\vy]}{C_\vm[\vy]}\right)\right]  - 3,
\end{equation}
and, since $s,s_0>-1$, the {\it actuation parameter} $\lambda :\Omega\to \R^+$ is well-defined by
\begin{equation}\label{eq:lambda}
\lambda = \lambda_{s,s_0}  := \sqrt[3]{\frac{s+1}{s_0+1}}.
\end{equation}
If the material is heated, then $\lambda <1$, whereas if it is cooled, then $\lambda >1$.
Moreover, $J[\vy] ,C_\vm[\vy]$ are among the following abbreviations:
\begin{equation}\label{eq:shortenings}
J[\vy] = \det \I[\vy], \quad C_\vm[\vy] = \vm \cdot \I[\vy]\vm,\quad C_{\vm_\perp}[\vy] = \vm_\perp \cdot \I[\vy]\vm_\perp.
\end{equation}
When the second argument of $W_{str}$ is $\vF\in\R^{3\times 2}$, we then use the following notational abbreviations:
\begin{equation}\label{eq:shortenings-F} 
\begin{aligned}
\I(\vF) := \vF^T\vF,& \quad J(\vF) := \det \I(\vF),\\
C_\vm(\vF) := \vm \cdot\I(\vF)\vm,& \quad C_{\vm_\perp}(\vF) := \vm_\perp \cdot \I(\vF)\vm_\perp.
\end{aligned}
\end{equation}
%
%We also define the first order variational derivative of $\I[\vy]$ along $\vv$ as  
%%L
%\begin{align}\label{eq:shortenings-linear-I}
%K[\vy,\vv] := \delta \I[\vy](\vv)=\nabla\vy^T\nabla\vv+\nabla\vv^T\nabla\vy,
%\end{align}
%%
%for later use. 
We emphasize that since $s,s_0,\vm$ depend on $\vx'\in\Omega$ then $W_{str}$ also has an explicit dependence on $\vx'$.

We note that \eqref{eq:stretching-energy-density} is consistent with the stretching energy in \cite{ozenda2020blend} after additionally assuming an inextensibility constraint $J[\vy] = 1$ and up to the multiplicative parameter $\lambda$ and the constant $-3$. 

The energy $E_{str}$ in \eqref{eq:stretching-energy} is not weakly lower semicontinuous in $H^1(\Omega;\mathbb{R}^3)$, which we show in \cite{bouck2022NA}. As a result, $E_{str}$ may not have minimizers in $H^1(\Omega;\mathbb{R}^3)$, but may admit minimizing sequences \cite{bethuel1999variational}. In fact, one can adapt \cite[Example 2.8]{bouck2022NA} to show that the energy density \eqref{eq:stretching-energy-density} is not quasiconvex  in the sense of \cite[Def. 1.5]{dacorogna2007direct}, the correct notion of convexity for a vector-valued problem. 
The lack of weak lower semi-continuity is responsible for the main difficulties to prove convergence of our discretization as well as to design efficient iterative solvers for the discrete minimization problem. We discuss convergence of discrete minimizers in \cite{bouck2022NA}, and present a nonlinear iterative scheme with inner Newton solver in Section \ref{sec:newton-scheme}.

Throughout this work, we do not impose any boundary condition so that the material under consideration has free boundaries. If necessary, one can take Dirichlet boundary conditions into account with a simple modification on the method.

An important property of the stretching energy is that $W_{str}(\vx',\nabla\vy)=0$ if and only if $\I[\vy]=g$ pointwise, where $g\in\R^{2\times2}$ is the {\it target metric}
\begin{equation}\label{eq:target-metric}
g = \lambda^2\vm\otimes \vm +\lambda^{-1}\vm_\perp\otimes \vm_\perp;
\end{equation}
we show this property in Section \ref{sec:target-metric}. In the physics literature, maps $\vy$ that satisfy the metric constraint $\I[\vy]=g$ are known as spontaneous distortions \cite{warner2007liquid, modes2011gaussian}. The physics community has developed techniques to find such deformations in special situations. Some examples are radially symmetric director fields \cite{warner2018nematic}, cylindrical shapes \cite{aharoni2014geometry}, and nonisometric origami \cite{modes2011blueprinting, plucinsky2018actuation, plucinsky2018patterning, plucinsky2016programming}. We refer to \cite{warner2020topographic} for a review of the techniques to predict shapes based on the metric. The purpose of this work is to provide a different approach via energy minimization and approximation. Rather than constructing $\vy$ analytically such that $\I[\vy] = g$, we numerically approximate minimizers to the stretching energy. 
We will validate our numerical method in Section \ref{sec:simulations} by successfully reproducing the intricate shapes resulting from higher order defects observed in experimental studies \cite{mcconney2013topography}, as well as exact nonisometric origami solutions \cite{plucinsky2018patterning}. An advantage of employing energy minimization and numerical approximation is the ability to tackle more general scenarios that lack exact analytical solutions. We extensively explore incompatible metrics in Section \ref{sec:incompatible-s}, which present significant challenges when attempting to solve $\I[\vy] = g$ exactly or study solutions analytically. Our computations \lucas{inspired} lab experiments \cite{bauman2023private}.
%We first show in Section \ref{sec:simulations} that the effectiveness of numerical method proposed in this paper can be validated by reproducing the shapes arising from higher order defects observed in experiments \cite{mcconney2013topography} as well as exact origami solutions \cite{plucinsky2018patterning}. One advantage of energy minimization and numerical approximation, is that we can tackle more general situations that may not have exact solutions. We numerically explore cases with incompatible metrics, which may be difficult to address by trying to analytically solve $\I[\vy] = g$.

%-----------------------------------------------------------------------------
\subsection{Discretizations}
%-----------------------------------------------------------------------------
In this work, we propose a FEM discretization to \eqref{eq:stretching-energy}. We consider the space $\V_h$ of continuous piecewise linear finite elements over a shape regular mesh $\Th$, and approximate the deformation $\vy$ by $\vy_h\in\V_h$. To define a discrete energy, we replace $\vy$ in \eqref{eq:stretching-energy} by $\vy_h$ and then add a \emph{regularization} term 
\lucas{
\begin{equation}\label{eq:regularization}
%R_h[\vy_h]:=c_rh^2\sum_{e\in\mathcal{E}_h}\frac{1}{h}\int_e \big|\jump{\nabla \vy_h}\big|^2
R_h[\vy_h] := \sum_{e\in\E_h}\int_{e} c_r h_e |\jump{\nabla\vy_h}|^2,
\end{equation}
where $c_r:\E_h\to \mathbb{R}$ is a nonnegative function of the skeleton $\E_h$. If $c_r$ is uniformly positive, then $R_h$} mimics a higher order bending energy. This bending regularization term is a scaled $L^2$ norm of jumps $\jump{\nabla \vy_h}$ across all the interior edges $e\in\mathcal{E}_h$ of $\Th$, it represents a scaled discrete Hessian of $\vy_h$, and provides a numerical selection mechanism to remove oscillations from equilibrium configurations; see Section \ref{sec:eff-reg}. \lucas{In cases where $c_r = 0$ on a collection of edges, this allows for folding to occur on those edges. If $c_r$ is uniformly positive across the whole skeleton $\E_h$, we abuse notation and regard $c_r$ as a positive parameter.} We prove convergence of discrete minimizers in our accompanying paper \cite{bouck2022NA}. The regularization $R_h[\vy_h]$ is closely related to the Nelson-Seung discrete bending energy \cite{seung1988defects}, which computes the jump of the normal vector across edges. \shuo{We also note that this energy is quite similar to the ridge energy introduced in \cite{pedrini2021ridge}, which \lucas{accounts for} the angle between normal vectors in the case $s=0$.} We further justify this choice of $R_h$ in Section \ref{sec:method}.
Moreover, we design a nonlinear gradient flow scheme with an embedded Newton sub-iteration, to solve the discrete minimization problem. We refer to Section \ref{sec:method} for details.

The discrete energy, defined as $E_h := E_{str}+R_h$, serves as a discrete counterpart to the blended energy given by
\begin{equation}\label{eq:blend-energy}
E_{str}[\vy] + t^2\int_\Omega|D^2\vy|^2d\vx'.
\end{equation}
The mesh size $h$ in $R_h$ plays a role analogous to the material thickness $t$ in \eqref{eq:blend-energy}. The second term in \eqref{eq:blend-energy} represents a simplified version of the bending energy, as discussed in \cite[Eq. (1.11)]{plucinsky2018actuation}. According to \cite{plucinsky2018actuation}, this energy term exhibits asymptotic behavior similar to that of a fully nonlinear elastic model as $t\to0$. The inclusion of the higher-order term provides additional compactness, motivated by the geometric rigidity described in \cite{friesecke2002theorem}. Similarly, the regularization $R_h$ in \eqref{eq:regularization} provides compactness if $E_h$ scales quadratically at minimizers, namely $E_h[\vy_h]\le\Lambda h^2$. This is stated in Theorem \ref{T:convergence} and proved in our accompanying paper \cite{bouck2022NA}.

%We view this term taking the role of $\int_\Omega h^2|D^2\vy|^2d\vx'$, so that the discrete energy $E+R_h$ is a discrete analog to the energy
%\[
%E_{str}[\vy] + t^2\int_\Omega|D^2\vy|^2d\vx',
%\]
%where the mesh size $h$ serves an analogous role to the thickness $t$ of the material. This kind of simplified model has been considered before in the context of LCEs/LCNs \cite[Eq. (1.11)]{plucinsky2018actuation}. As noted in \cite{plucinsky2018actuation}, this energy behaves asymptotically similarly to a fully nonlinear elastic model as $t\to0$. The higher order term provides additional compactness in the bending scaling regime similar to the geometric rigidity of \cite{friesecke2002theorem}. For more on this simplified bending and stretching model, we point to the discussion in \cite{plucinsky2018actuation} and references therein. 

%------------------------------------------------------------------------------
\subsection{Our contributions and outline of the paper}\label{sec:contributions}
%------------------------------------------------------------------------------
We summarize our main contributions in this work as follows. 

\begin{itemize}
\item \textbf{Asymptotic derivation of the $2D$ membrane model.} In Section \ref{sec:asymptotics} we derive a $2D$ membrane model from $3D$ rubber elasticity via an asymptotic analysis, motivated by \cite{ozenda2020blend}. However, we remove the inextensibility assumption $\det\I [\vy]=1$ of \cite{ozenda2020blend}. 
%and show that our reduced model induces a strictly smaller $3D$ energy than that in \cite{ozenda2020blend}.
%\lucas{We also show that the new ansatz from dropping the inextensibility assumption with zero stretching energy has strictly lower 3D energy density than the ansatz found in \cite{ozenda2020blend}.}
\shuo{We additionally show that this modified model exhibits a lower 3D energy density compared to the ansatz in \cite{ozenda2020blend}.}
Moreover, we provide a concise and novel proof that global minimizers of \eqref{eq:stretching-energy} satisfy the target metric \eqref{eq:target-metric} in Section \ref{sec:target-metric}.\looseness=-1

\smallskip  
\item \textbf{Asymptotic profiles of defects.} In Section \ref{sec:special-solution}, we present a new formal construction of solutions for rotationally symmetric blueprinted director fields $\vm$ with a defect of degree $n>1$. Our technique hinges on the ideas of lifted surfaces (inspired by \cite{plucinsky2018actuation}),  composition of defects and Taylor expansion.

\smallskip   
\item \textbf{Finite element discretization and iterative solver.} In Section \ref{sec:method}, we present our new finite element discretization of \eqref{eq:stretching-energy} that includes the regularization term \eqref{eq:regularization}, as well as a nonlinear gradient flow scheme together with a Newton sub-iteration to solve the resulting discrete nonconvex minimization problem.
  
\smallskip 
\item \textbf{Numerical simulations.} In Section \ref{sec:simulations}, we present numerous simulations of practical interest. They illustrate the ability of our discrete membrane model to capture intriguing physical phenomena such as origami-like structures and deformations due to defects with varying degree. Furthermore, our discrete method enables us to investigate computationally novel incompatible origami structures, where the metric $g$ is discontinuous. Our computational findings might provide valuable insights for future endeavors in modeling, analysis, laboratory experiments, and potential applications of LCN materials.
%We conclude with computations of origami structures where the metric $g$ is discontinuous, which we term incompatible nonisometric origami.
\end{itemize}  

%In summary, in this work, we focus on the asymptotic and computational aspects of the $2D$ reduced membrane model for LCNs.
%In contrast, we discuss further analytical properties of the membrane model (coercivity, non-degeneracy, and lack of rank-one convexity) in our companion paper \cite{bouck2022NA}, and we present a full convergence analysis of the proposed finite element discretization in the spirit of $\Gamma$-convergence. The latter includes the presence of creases.

%%%%%%%%%%%%%%%%%%%%%%%%%%%%%%%%%%%%%%%%%%%%%%%%%%%%%%%%%%%%%%%%%%%%%%%%%%%%%%%%%%%%%
\section{Membrane model of liquid crystals polymer networks}
%%%%%%%%%%%%%%%%%%%%%%%%%%%%%%%%%%%%%%%%%%%%%%%%%%%%%%%%%%%%%%%%%%%%%%%%%%%%%%%%%%%%%
In this section, we introduce a formal asymptotic derivation of a membrane model of LCNs and discuss properties of the model related to its global minimizers. We also describe a modified asymptotics that leads to a bending model.

%------------------------------------------------------------------------------------
\subsection{Derivation of stretching energy from asymptotics}\label{sec:asymptotics}
%------------------------------------------------------------------------------------
This section is dedicated to deriving a 2D stretching or membrane energy from \eqref{eq:3D-energy} via formal asymptotics as the thickness $t$ goes to zero. In particular, we shall derive the formal limit $\lim_{t\to0}\frac{1}{t}E_{3D}[\vu]$. This procedure will follow closely the derivation of \cite{ozenda2020blend}, but we will relax the simplifying assumption $\det \I[\vy] = 1$ made in \cite{ozenda2020blend}. We also contrast the asymptotic method presented here with the more analytical method presented in \cite{cirak2014computational}. 

%------------------------------------------------------------------------------------
\subsubsection{Kirchhoff-Love assumption and overview of strategy}\label{sec:asymptotic-1}
%------------------------------------------------------------------------------------
We assume that the 3D deformation $\vu: \mathcal{B}:=\Omega \times (-t/2,t/2)\to \mathbb{R}^3$ takes the form
\begin{equation}\label{eq:kirchhhoff-love-assumption}
\vu(\vx',x_3) = \vy(\vx') +\phi(\vx',x_3) \, \vnu(\vx')
\end{equation}
where $\vy:\Omega\to \mathbb{R}^3$ is the reduced deformation and $\vnu:\Omega\to \mathbb{R}^3$ is the normal to the deformed midplane $\vy(\Omega)$. We posit that $\phi$ takes the form:
\begin{equation}\label{eq:phi-expansion}
\phi(\vx',x_3) = \alpha(\vx')x_3 + \mathcal{O}(x_3^2),
\end{equation}
which is a modified {\it Kirchhoff-Love assumption}. The higher order terms would be useful for deriving the bending energy, but we do not need them for the stretching energy. Note that $\alpha$ is undetermined for the moment. Later, $\alpha$ will be chosen so that $\det\nabla \vu$ is nearly 1, i.e.\ $\det\nabla\vu(\vx',x_3) = 1 +\mathcal{O}(x_3)$. The approach we adopt here follows \cite{ozenda2020blend}. We assume that the integral formula \eqref{eq:3D-energy} is valid and finite if the deformation $\vu$ satisfies $\det\nabla\vu(\vx',x_3) = 1 +\mathcal{O}(x_3)$. To justify this assumption heuristically, if $\vu$ is smooth and satisfies $\det\nabla\vu(\vx',x_3) = 1 +\mathcal{O}(x_3)$, then we expect that for sufficiently small $t$ one can find an incompressible $\vv$ such that 
\begin{equation}\label{eq:3D-energy-approx-uv}
|E_{3D}[\vv] - E_{3D}[\vu]| \leq Ct^2.
\end{equation} 
The non-degeneracy \eqref{eq:nondegeneracy} still holds for $\vF=\nabla\vu$ as $t\to0$ in view of \eqref{eq:3D-energy-approx-uv}.

To prove \eqref{eq:3D-energy-approx-uv}, if $\vu(\vx',0), \partial_3\vu(\vx',0) \in C^\infty(\overline{\Omega};\mathbb{R}^3)$ are such that  $\det\nabla\vu(\vx',0)= 1$, then \cite[Proposition 5.1]{conti2006derivation} states that there is a $t_0>0$ and  $\vv\in C^\infty(\overline{\Omega}\times(-t_0/2,t_0/2);\mathbb{R}^3)$ that satisfies $\det\nabla\vv(\vx',x_3) = 1$ everywhere in $\overline{\Omega}\times(-t_0/2,t_0/2)$, $\vv(\vx',0) = \vu(\vx',0)$, and $|\nabla\vv(\vx',x_3) - \nabla \vu(\vx',0)|\leq Cx_3$. The uniform pointwise estimate and integrating in the $x_3$ direction would show $|E_{3D}[\vv] - E_{3D}[\vu]| \leq Ct^2$.

The goal of asymptotics is to write the energy $W_{3D}(\vx,\nabla \vu)$ given in \eqref{eq:neo-classical} for the deformation $\vu$ in terms of powers of $t$ and the reduced stretching energy $W_{str}(\vx',\nabla'\vy)$
\begin{equation}\label{eq:expansion}
\int_{-t/2}^{t/2}\int_\Omega W_{3D}(\vx,\nabla \vu) \, d\vx' dx_3 = t \int_\Omega W_{str}(\vx',\nabla'\vy) \, d\vx' + \mathcal{O}(t^3),
\end{equation}
where $\nabla':=(\partial_1,\partial_2)$ denotes the gradient with respect to $\vx'$.
The stretching energy $W_{str}$ in \eqref{eq:expansion} gives the leading order effects of the energy as the body thickness $t$ vanishes in the sense that formally
\[
\lim_{t\to0}\frac{1}{t} \int_{-t/2}^{t/2}\int_\Omega {W}_{3D} (\vx,\nabla \vu) \, d\vx' dx_3 = \int_\Omega W_{str}(\vx',\nabla'\vy) \, d\vx'.
\]
The third order term in \eqref{eq:expansion} corresponds to the bending energy $W_{ben}$ and is not the focus of the current work.
Combined with the modified Kirchhoff-Love assumption \eqref{eq:kirchhhoff-love-assumption}, the process to derive the stretching energy is as follows:
\begin{enumerate}
\item Write the right Cauchy-Green tensor $\vC = \vF^T\vF$ in terms of leading order terms.
\item Write $W_{3D}$ in terms of $\vC$ and powers of $x_3$.
\item Collect $\mathcal{O}(1)$ terms of $W_{3D}$ which contribute to the stretching energy.
\item Determine $\alpha$ so that $\vu$ satisfies incompressibility in an asymptotic sense.
\end{enumerate}

%---------------------------------------------------------------------------------
\subsubsection{Right Cauchy-Green tensor}
%---------------------------------------------------------------------------------
Substituting \eqref{eq:kirchhhoff-love-assumption} into $\vC:=\nabla\vu^T\nabla\vu$ yields
\begin{equation}
\vC = \nabla \vu^T\nabla \vu = \begin{bmatrix}\vC_\phi & \vC_{1\times 2}^T\\ \vC_{1\times 2} & \vC_{1\times 1}\end{bmatrix},
\end{equation}
where
\begin{align}
\vC_\phi &= \nabla'\vy^T\nabla'\vy + \phi (\nabla'{\vnu}^T\nabla'\vy+ \nabla'\vy^T\nabla'{\vnu})+ \phi^2 \nabla'{\vnu}^T\nabla'{\vnu} +\nabla'\phi\otimes \nabla'\phi\\
\vC_{1\times 2} &= (\vnu\otimes \nabla' \phi)^T\partial_3 \phi \vnu = \partial_3\phi \nabla' \phi\\
\vC_{1\times 1} &= (\partial_3 \phi)^2.
\end{align}
Here, we have used the facts that $\nabla'\vy^T\vnu=0$, $\nabla'{\vnu}^T\vnu=0$ and $|\vnu|=1$. 
Since $\nabla'\phi(\vx',x_3) = \nabla'\alpha x_3+ \mathcal{O}(x_3^2)$ and $\partial_3 \phi(\vx',x_3) = \alpha +\mathcal{O}(x_3)$, we have $C_{1\times2}~=~\mathcal{O}(x_3)$, and hence we may drop $C_{1\times2}$ as higher order terms.

Also ignoring any terms higher than constant order, we have 
\begin{equation*}
\vC_\phi =  \I[\vy] +\mathcal{O}(x_3),
\end{equation*}
where $\I[\vy]=\nabla'\vy^T\nabla'\vy$ is the first fundamental form of $\vy$. Since 
\begin{equation*}
\big(\partial_3\phi(\vx',x_3) \big)^2 = \alpha(\vx')^2 +\mathcal{O}(x_3),
\end{equation*}
we find
\begin{equation}\label{eq:C-asymp}
\vC = \begin{bmatrix}\I[\vy] & 0\\ 0 & \alpha(\vx')^2\end{bmatrix} +\mathcal{O}(x_3).
\end{equation}

%----------------------------------------------------------------------------------
\subsubsection{Expanding $W_{3D}$}
%----------------------------------------------------------------------------------
Recall that we assume that the $3D$ blueprinted director field ${\bf m}$ lies in the plane i.e. 
\begin{equation}\label{eq:m-in-plane}
{\bf m}(\vx) = ({\bf \wt{m}}(\vx'), 0 )^T.
\end{equation} 
First, substituting the kinematic constraint \eqref{eq:kinematic-cons} into \eqref{eq:neo-classical} with $\vF=\nabla\vu$, we obtain
\begin{equation}\label{eq:3D-energy-intermediate}
W_{3D}(\vx',\nabla\vu)=\lambda\Big(\tr \, \vC+\frac{s_0}{s+1}\vm\cdot \vC\vm-\frac{s}{s+1}\frac{\vm\cdot \vC^2\vm}{\vm\cdot \vC\vm}\Big)-3,
\end{equation}
where $\lambda$ is defined in \eqref{eq:lambda}, and we notice that from now on $W_{3D}$ depends on $\vx'$ instead of $\vx$, due to the assumption \eqref{eq:m-in-plane}. 
Then plugging the asymptotic form \eqref{eq:C-asymp} of $\vC$ into \eqref{eq:3D-energy-intermediate} and using \eqref{eq:m-in-plane}, the energy density $W_{3D}(\vx',\nabla\vu)$ becomes
\begin{equation*}
\begin{aligned}
W_{3D} &(\vx',\nabla\vu) \\
&  
= \lambda\left[\text{tr} \,\I[\vy] +\alpha(\vx')^2+  \frac{s_0}{s+1} {\bf \wt{m}}\cdot \I[\vy]{\bf \wt{m}} - \frac{s}{s+1} \frac{{\bf \wt{m}}\cdot \I[\vy]^2{\bf \wt{m}}}{{\bf \wt{m}}\cdot \I[\vy]{\bf \wt{m}}}\right]-3 +\mathcal{O}(x_3).
\end{aligned}
\end{equation*}
Since $\I[\vy]$ is a $2\times 2$ matrix, the Cayley-Hamilton Theorem gives 
$$\I[\vy]^2 = \big(\text{tr} \, \I[\vy]\big) \, \I[\vy] - \text{det}\, \I[\vy] \, \Id_2,$$
so that the energy now reads
\begin{equation*}
\begin{aligned} 
  W_{3D} & (\vx',\nabla\vu)
  \\
  & = \lambda\left[\alpha(\vx')^2+\frac{1}{s+1}\left(\text{tr} \,\I[\vy] +  s_0 {\bf \wt{m}}\cdot \I[\vy]{\bf \wt{m}} + s \frac{\text{det} \I[\vy]}{{\bf \wt{m}}\cdot \I[\vy]{\bf \wt{m}}}\right)\right] - 3 +\mathcal{O}(x_3).
\end{aligned}  
\end{equation*}
We now have all the constant order terms of $W_{3D}(\vx',\nabla\vu)$. The only remaining task is to determine $\alpha(\vx')$. We do this next.

%----------------------------------------------------------------------------------
\subsubsection{Incompressibility}\label{sec:asymptotic-4}
%----------------------------------------------------------------------------------
Since we would like $\vu$ to satisfy incompressibility $\det \nabla\vu = 1 + \mathcal{O}(x_3)$, we impose $\text{det} \, \vC = 1+ \mathcal{O}(x_3)$. By \eqref{eq:C-asymp}, we see that
\begin{equation}
\text{det} \,\vC = \text{det} \,\I[\vy] \,\alpha(\vx')^2 + \mathcal{O}(x_3),
\end{equation}
whence
\begin{equation}\label{eq:alpha}
\alpha(\vx') = \frac{1}{\sqrt{\text{det} \, \I[\vy] }}
\end{equation}
gives us the desired equality $\text{det} \,\vC = 1+ \mathcal{O}(x_3)$. This further implies
\begin{equation}\label{eq:gradu-grady}
\nabla \vu = [\nabla\vy, (\det \, \I[\vy])^{-1/2} \vnu] + \mathcal{O}(x_3).
\end{equation}
\begin{remark}[comparison with \cite{ozenda2020blend}]\label{rmk:comparison-with-virga}
%The key difference between the derivation presented here and the one in \cite{ozenda2020blend} can be found in equation \eqref{eq:alpha}. 
\lucas{The key difference between our derivation and the one in \cite{ozenda2020blend} is that $\det\I[\vy]=1$, or equivalently $\alpha=1$, was assumed in \cite{ozenda2020blend}. Instead, we impose $\alpha = (\det \, \I[\vy])^{-1/2}$ in \eqref{eq:alpha}}. A physical interpretation is that the LCN film adjusts its thickness to accommodate changes in area within the midplane and maintain 3D incompressibility. Further elaboration on this difference, relating it to the stretching energy, can be found in Section \ref{sec:inextensibility-incompressibility}. Notably, removing the assumption of $\det\I[\vy]=1$ is also more convenient for numerical computations.
\end{remark}

%%In \eqref{eq:alpha} lies the key difference between the derivation here vs.\ that of \cite{ozenda2020blend}, which assumes $\det\I[\vy]=1$ or equivalently $\alpha=1$. Physically, $\alpha =( \det \, \I[\vy])^{-1/2} $ means the LCN film accommodates changes with area in the midplane by changing the thickness of the film to satisfy 3D incompressibility. This statement is further expanded in terms of energy in Section \ref{sec:inextensibility-incompressibility}. Also, the removal of the assumption $\det\I[\vy]=1$ is more convenient for numerics.
%----------------------------------------------------------------------------------
\subsubsection{Stretching energy}
%----------------------------------------------------------------------------------
For convenience of presentation, from now on we slightly abuse notation and drop the prime for $\nabla'$ and tilde for $\wt\vm$. Therefore, we will denote
\begin{equation}\label{eq:m-tildem}
  \vm = \wt\vm\in\R^2,
  \quad
  \vm = (\wt\vm,0)^T \in \R^3
\end{equation}
depending on whether we regard $\vm$ as a vector in $\R^2$ or $\R^3$.

We conclude that, with the steps built in Sections \ref{sec:asymptotic-1}-\ref{sec:asymptotic-4}, we finally derive the stretching energy in \eqref{eq:stretching-energy-density}, namely
\begin{equation}\label{eq:stretching-energy-density-2}
W_{str}(\vx',\nabla\vy) := \lambda \left[\frac{1}{J[\vy]}+\frac{1}{s+1}\left(\text{tr} \, \I[\vy] +  s_0 C_\vm[\vy] + s \frac{J[\vy]}{C_\vm[\vy]}\right)\right]  - 3,
\end{equation}

\subsection{Global minimizers and target metrics.}\label{sec:target-metric}
%---------------------------------------------------------------------------------------
In this section, we characterize global minimizers of \eqref{eq:stretching-energy}.
We show that minimizing the stretching energy density $W_{str}$ is equivalent to satisfying the target metric constraint pointwise. We point to \cite[Appendix A]{plucinsky2018actuation} for a similar result but for a more complicated 3 dimensional model. The authors in \cite{plucinsky2018actuation} also show that this metric condition arises from a compactness result in the bending scaling regime \cite[Theorem 1.13]{plucinsky2018actuation}.

We first note that the algebra in Section \ref{sec:asymptotics} can be used to show the following algebraic relation between $W_{str}$ and $W_{3D}$. 
\begin{lemma}[stretching vs.\ 3D energy]\label{lem:minimal-energy-extension}
Let $\vx'\in \Omega$ and let $\vF = [\vF_1, \vF_2] \in \mathbb{R}^{3\times2}$ be such that $\text{rank} \, (\vF) = 2$. Then the following algebraic relation holds:
\begin{equation}\label{eq:minimal-energy-extension}
W_{str}(\vx',\vF) = W_{3D}((\vx',0),[\vF,\vb(\vF)]) \, , \quad  \vb(\vF) = \frac{\vF_1\times\vF_2}{|\vF_1\times\vF_2|^2}.
\end{equation}
\end{lemma}
\begin{proof}
Let $\vb = \vb[\vF]$ and notice that $\vb^T\vF = 0$ and $|\vb|^2 = \frac{1}{|\vF_1\times\vF_2|^2} = \frac{1}{J(\vF)}$. Therefore, $\vC := [\vF,\vb]^T[\vF,\vb]$ satisfies
\[
\vC = \begin{bmatrix} \I(\vF) & 0 \\ 0 & J(\vF)^{-1}\end{bmatrix}.
\]
Inserting $\vC$ into the right hand side of \eqref{eq:3D-energy-intermediate} and using that $\vm$ is planar yields
\begin{equation*}
W_{3D}\big(\vx,[\vF,\vb]\big) =\lambda\left(J(\vF)^{-1}+ \tr \, \I(\vF)+\frac{s_0}{s+1}C_\vm(\vF)-\frac{s}{s+1}\frac{\vm\cdot \I(\vF)^2\vm}{C_\vm(\vF)}\right)-3,
\end{equation*}
for $\vx:=(\vx',0)$. We next apply the Cayley-Hamilton Theorem to get $\I(\vF)^2 = \big(\text{tr} \, \I(\vF)\big) \, \I[\vy] - \text{det}\, \I(\vF) \, \Id_2$ and 
\begin{align*}
W_{3D}\big(\vx,[\vF, \vb]\big) &=\lambda\left[J(\vF)^{-1}+\frac{1}{s+1}\left( \tr \, \I(\vF)+s_0C_\vm(\vF)-s\frac{J(\vF)}{C_\vm(\vF)}\right)\right]-3,
\end{align*}
which is the desired expression of $W_{str}(\vx',\vF)$ according to \eqref{eq:stretching-energy-density-2}.
\end{proof}

In addition to \eqref{eq:minimal-energy-extension}, $\vb(\vF)$ satisfies the minimality property
\begin{equation}\label{eq:optimal-choice}
W_{3D}\big((\vx',0),[\vF,\vb(\vF)]\big) = \min_{\vc\in\R^3: \det [\vF,\vc]=1} W_{3D}\big((\vx',0),[\vF,\vc]\big);
\end{equation}
we refer to \cite[Proposition 1]{bouck2022reduced} for a proof.
This algebraic relation is the main tool in the derivation of the membrane model with higher order energies in \cite{cirak2014computational} and is used in the derivation of a different membrane model in \cite[Lemma 5.3]{cesana2015effective}.  For us, the algebraic relation \eqref{eq:minimal-energy-extension} shows that $W_{str}$ is nondegenerate, and that deformations with zero stretching energy satisfy a target metric condition. We discuss this below.

The lower bound of the next proposition is an easy consequence of 
Lemma \ref{lem:minimal-energy-extension} (stretching vs.\ 3D energy)
and \eqref{eq:nondegeneracy}, whereas the upper bound is proved in \cite{bouck2022NA}.
\begin{prop}[nondegeneracy]\label{cor:stretching-nondegeneracy}
The stretching energy density $W_{str}$ satisfies
\begin{equation*}
 \mathrm{dist}\big(\vL_{\vn}^{-1/2}[\vF, \vb]\vL_{\vm}^{1/2},SO(3)\big)^2\leq W_{str}\big(\vx',\vF\big)\leq 3\,\mathrm{dist}\big(\vL_{\vn}^{-1/2}[\vF, \vb]\vL_{\vm}^{1/2},SO(3)\big)^2
\end{equation*}
for all $\vF\in \mathbb{R}^{3\times2}$ such that $\mathrm{rank} \, (\vF)=2$ and $\vb = \frac{\vF_1\times\vF_2}{|\vF_1\times\vF_2|^2}$.
\end{prop}
%The next result states that $W_{str}(\vx',\vF) = 0$ if and only if $\vF$ satisfies a target metric condition. This fact is well-known in the physics literature, and is commonly used to predict the shapes of LCNs/LCEs under actuation \cite{warner2018nematic, aharoni2014geometry, modes2011blueprinting, plucinsky2018patterning, plucinsky2016programming, warner2020topographic}. 
%In fact, the metric condition arises from a 3D model of LCEs in the bending scaling regime \cite[Theorem 1.13]{plucinsky2018actuation}. We point to \cite[Proposition A.4]{plucinsky2018actuation} for a result showing that the metric condition is equivalent to vanishing energy states from that more complicated 3D model. The proof we present here is for the simpler $W_{str}$. The resulting proof is simpler as well.
The following result establishes that $W_{str}(\vx',\vF) = 0$ if and only if $\vF$ satisfies a target metric condition. This well-known fact is extensively documented in the physics literature, where it is commonly utilized to predict the shapes of LCNs/LCEs during actuation \cite{warner2018nematic, aharoni2014geometry, modes2011blueprinting, plucinsky2018patterning, plucinsky2016programming, warner2020topographic}. We next present a concise and novel proof of it.

\begin{prop}[target metric]\label{prop:target-metric} 
The stretching energy density $W_{str}(\vx',\vF) = 0$ vanishes at $\vF\in\R^{3\times2}$ if and only if $\I(\vF) = g$
where $g\in\mathbb{R}^{2\times2}$ is given by
\begin{equation}\label{eq:target-metric}
g = \lambda^2 {\vm}\otimes{\vm} +\lambda^{-1}{\vm}_\perp\otimes{\vm}_\perp,
\end{equation}
$\lambda$ is defined in \eqref{eq:lambda} and $\vm_\perp:\Omega\to\mathbb{S}^1$ is perpendicular to $\vm$.
\end{prop}
\begin{proof}
We first note that due to Proposition \ref{cor:stretching-nondegeneracy}, $W_{str}(\vx',\vF)=0$ if and only if
$\vR:=\vL_\vn^{-1/2}[\vF,\, \vb]\vL_\vm^{1/2}~\in~ SO(3)$, where $\vb = \frac{\vF_1\times\vF_2}{|\vF_1\times\vF_2|^2}$.
Then
\[
\vR^T\vR=\Id_3 \quad\Rightarrow\quad
[\vF,\, \vb]^T\vL_\vn^{-1}[\vF,\, \vb] = \vL_\vm^{-1}.
\]
To show the equality $\I(\vF) = g$, we equate the upper $2\times2$ blocks of the preceding matrix equality and use that $\vn = \vF\vm / |\vF\vm|$, according to \eqref{eq:kinematic-cons}, to obtain
\[
(s+1)^{\frac13}\left(\I(\vF) - \frac{s}{s+1}\frac{\I(\vF)\vm \otimes (\I(\vF)\vm)}{\vm\cdot \I(\vF)\vm}\right) = \wt{\vL_\vm^{-1}},
\]
where $\wt{\vL_\vm^{-1}}$ denotes the upper $2\times2$ block of $\vL_\vm^{-1}=(s_0+1)^{\frac13} \big(\Id_3-\frac{s_0}{s_0+1} \vm\otimes\vm\big)$.
Multiplying both sides by $\vm$ shows that $\vm$ is an eigenvector of $\I(\vF)$ with corresponding eigenvalue $\lambda^2$. Since $\I(\vF)$ is symmetric, then $\vm_\perp$ is an eigenvector, and we can similarly show that the corresponding eigenvalue for $\vm_\perp$ is $\lambda^{-1}$, whence $\I(\vF) = g$.
Conversely, if $\I(\vF) = g$ we resort to \eqref{eq:stretching-energy-density-2} to write
\[
W_{str} (\vx',\vF)= \lambda \left( \frac{2}{\lambda} + \frac{s_0+1}{s+1} \lambda^2 \right) - 3 = 0,
\]
in light of \eqref{eq:lambda}. This concludes the proof.
\end{proof}

A direct consequence of the characterization of the target metric is that $H^1$ isometric immersions of $g$ are minimizers to the stretching energy.
\begin{cor}[immersions of $g$ are minimizers with vanishing energy]\label{cor:minimizers}
A deformation $\vy\in H^1(\Omega;\mathbb{R}^3)$ satisfies 
\begin{equation}\label{eq:metric-cons}
\I[\vy]=g\text{ a.e. in }\Omega, 
\end{equation}
i.e., $\vy$ is an \emph{isometric immersion} of the metric $g$ defined in \eqref{eq:target-metric}, if and only if $\vy$ is a global minimizer to \eqref{eq:stretching-energy} with $E_{str}[\vy]=0$. 
\end{cor}

Therefore, the solvability of \eqref{eq:stretching-energy} is related to the long standing open problem in differential geometry of existence of isometric immersions in $\mathbb{R}^3$ for a general metric $g:\Omega\to\mathbb{R}^{2\times2}$. Smooth isometric immersions in $\mathbb{R}^3$ are known to exist for certain metrics with positive or negative curvatures, while there are also examples of metrics that have no $C^2$ isometric immersions; we refer to the book \cite{han2006isometric} for discussions and further references. Corollary \ref{cor:minimizers} requires the minimal regularity $\vy\in H^1(\Omega;\mathbb{R}^3)$, but we further assume the existence of an $H^2$ isometric immersion to prove convergence of our FEM with regularization \eqref{eq:regularization} in \cite{bouck2022NA}. The existence of either $H^1$ or $H^2$ isometric immersions seems to be an open question, to the best of our knowledge. Finally, it is conceivable that $g$ is not immersible and yet there is a global minimizer $\vy$ of \eqref{eq:stretching-energy} with $E_{str}[\vy]>0$; this justifies the requirement $E_{str}[\vy]=0$ in Corollary \ref{cor:minimizers}. We explore this matter computationally in Section \ref{sec:incompatible-s}.

%%%%%%%%%%%%%%%%%%%%%%%%%%%%%%%%%%%%%%%%%%%%%%%%%%%%%%%%%%%%%%%%%%%%%%%%%%%%%%%%%%%%%%%%
\subsection{Inextensibility vs Incompressibility}\label{sec:inextensibility-incompressibility}
This section is dedicated to comparing the choice of $\alpha= (\det \I[\vy])^{-1/2}$ in \eqref{eq:alpha} vs.\ $\alpha=1$ as in \cite{ozenda2020blend}. We complement Remark \ref{rmk:comparison-with-virga} (\lucas{comparison} with \cite{ozenda2020blend}) with a \lucas{pointwise statement about energy densities}. \lucas{We first observe that for any $\vF\in \R^{3\times2}$ that satisfies the metric condition $\I(\vF) = \vF^T\vF = g(\vx')$ at $\vx'\in \Omega$, we have
$$
W_{str}(\vx',\vF) = W_{3D}((\vx',0) ,[\vF, \vb(\vF)]) \textcolor{red}{= \min_{\vc\in\R^3: \det [\vF,\vc]=1} W_{3D}\big((\vx',0),[\vF,\vc]\big)}= 0
$$
according to Lemma \ref{lem:minimal-energy-extension} (stretching vs 3D energy) and Proposition \ref{prop:target-metric} (target metric). Moreover, since $g(\vx')$ is symmetric positive definite, one such matrix exists: let 
$$
 \vF= \begin{pmatrix}\sqrt{g(\vx')} \\ \bz \end{pmatrix}\in \R^{3\times2}
$$
and note that $\I(\vF) = g(\vx')$ and $\det\I(\vF) = \lambda$. The next lemma states that our ansatz yields strictly lower 3D energy density than that of \cite{ozenda2020blend}.
} 
\begin{lemma}[energy \lucas{density} gap]\label{lem:energy-gap}
%For $\vF\in \mathbb{R}^{3\times2}$ with $\textrm{rank }(\vF) = 2$, we define $\vb(\vF) = \frac{\vF_1\times\vF_2}{|\vF_1\times\vF_2|^2}$. 
\lucas{For any $\vx'\in\Omega$ and $\lambda\neq1$, we have
\[
\inf_{\vF\in \mathbb{R}^{3\times2}, \det \normalfont{\textrm{I}}(\vF)=1} W_{3D}\big((\vx',0),[\vF,\, \vb(\vF)]\big) > 0
\]
where $\vb(\vF)$ is defined in \eqref{eq:minimal-energy-extension}. 
%Moreover, the right-hand side is the smallest possible energy $W_{3D}\big((\vx',0),[\vF,\,\vc]\big)$ with $\vc\in\R^3$ so that $\det[\vF,\,\vc] = 1$.
}
\end{lemma}
\begin{proof}
\lucas{In view of Lemma \ref{lem:minimal-energy-extension} (stretching vs.\ 3D energy) and Proposition \ref{prop:target-metric} (target metric), we just have to prove}
\[
\inf_{\vF\in \mathbb{R}^{3\times2}, \det \normalfont{\textrm{I}}(\vF)=1}W_{str}(\vx',\vF) > 0.
\]
For any $\vF\in \mathbb{R}^{3\times2}$, we may write $\I(\vF) = a\vm\otimes \vm +b \vm_\perp\otimes \vm_\perp + {c}(\vm_\perp\otimes \vm+ \vm\otimes \vm_\perp)$, and note that $\det\I((\vF) = 1$ is equivalent to $ab - c^2 = 1$. The resulting stretching energy density \eqref{eq:stretching-energy-density-2} written in terms of $\vF$ reads
\[
W_{str}(\vx',\vF) = \lambda\left[1+ \frac{1}{s+1}\left((s_0+1)a +b +s\frac{1}{a}\right)\right] - 3.
\]
We append the constraint $ab-c^2=1$ to $W_{str}(\vx',\vF)$ via a Lagrange multiplier and differentiate the resulting expression to obtain the optimality conditions: $c=0, a=b^{-1}=\big(\frac{s+1}{s_0+1}\big)^{1/2} = \lambda^{3/2}$. This implies
\[
\inf_{\vF\in \mathbb{R}^{3\times2}, \det \normalfont{\textrm{I}}(\vF)=1} W_{str}(\vx',\vF) = 2\lambda^{-1/2} +\lambda - 3,
\]
which is strictly positive if $\lambda\neq1$. 
%Finally, applying \eqref{eq:optimal-choice} concludes the proof.
\end{proof}

%%%%%%%%%%%%%%%%%%%%%%%%%%%%%%%%%%%%%%%%%%%%%%%%%%%%%%%%%%%%%%%%%%%%%%%%%%%%%%%
\subsection{Bending energy}\label{sec:bending-energy}
Although we are concerned with a membrane model, we now briefly address the bending energy for LCNs in order to provide motivation for the regularization $R_h$ in \eqref{eq:regularization}, later discussed in Section \ref{sec:method}. We refer to \cite[Section 3.4]{bouck:2023thesis} for a complete derivation.

The bending energy, denoted as $E_{bend}[\vy]:= \lucas{\frac{1}{12}}\int_\Omega W_{bend}(\vx',\nabla\vy, D^2\vy) d\vx'$, is the second term in the asymptotic expansion $t^{-1}E_{3D} = E_{str}+t^2E_{bend} +\mathcal{O}(t^4)$. While building upon the work of \cite{ozenda2020blend}, our approach again differs in that it relaxes the inextensibility assumption $\det\I[\vy]=1$. In fact, our new ansatz is given by
\begin{equation}\label{eq:bending-ansatz}
\begin{aligned}
\vu(\vx',x_3) &= \vy(\vx')+\left(\alpha(\vx')x_3+\beta(\vx')x_3^2+\gamma(\vx')x_3^3\right)\vnu(\vx'),
\\
\beta(\vx') &= -H \alpha(\vx')^2, \quad \gamma(\vx') = \frac{\alpha(\vx')^2}{3}(6H^2-K),
\end{aligned}
\end{equation}
where $K = \det(-\II[\vy]\I[\vy]^{-1})$ and $H = \frac12\tr(-\II[\vy]\I[\vy]^{-1})$ represent the \lucas{Gaussian} curvature and mean curvature of the surface $\vy(\Omega)$, respectively, and are written in terms of the {\it second fundamental form} $\II[\vy] = (\partial_{ij}\vy \cdot \vnu)_{ij=1}^2\in\R^{2\times2}$ of $\vy(\Omega)$. It is worth noting that this ansatz satisfies $\det\nabla\vu(\vx',x_3) = 1+\mathcal{O}(x_3^3)$, and that it reduces to that in \cite[Eq.\ (59)]{ozenda2020blend} under the inextensibility assumption $\alpha(\vx') = 1$.

We further assume that $s,s_0$ are constants with respect to $\vx'$. We insert the ansatz \eqref{eq:bending-ansatz} into the expression for $W_{3D}$ and collect terms up to $\mathcal{O}(x_3^2)$. If $\I[\vy]=g$, then the resulting bending energy density is given by
\begin{equation*}\label{eq:bouck-bending-energy}
W_{bend}(\vx',\nabla\vy,D^2\vy) =  \frac{16 H^2}{J[\vy]} +\frac{s}{s+1}\frac{J[\vy]}{C_m[\vy]}
\left(8H \frac{C_{\II}[\vy]}{C_\vm[\vy]}+4\frac{C_{\II}[\vy]^2}{C_\vm[\vy]^2}\right)+C(g(\vx')).
\end{equation*}
Here, $C_{\II}[\vy] := \vm\cdot \II[\vy]\vm$ and $C(g(\vx'))$ depends only on the target metric $g(\vx')$. It should be noted that $J[\vy]$ and $C_\vm[\vy]$ are constants when $\I[\vy]=g$.
Moreover, this bending energy $W_{bend}$ agrees with \cite[Eq.\ (74)]{ozenda2020blend}, under the condition that $\I[\vy]$ satisfies the target metric specified in \cite[Eq.\ (17)]{ozenda2020blend}. If, in addition $s=0$, then the middle term in $W_{bend}$ vanishes and thus $W_{bend}$ simplifies to
\begin{equation}\label{eq:bouck-bending-energy-simple}
W_{bend}(\vx',\nabla\vy,D^2\vy) =  C\big(\tr(g^{-1/2}\II[\vy]g^{-1/2})\big)^2+C(g(\vx')),
\end{equation}
because $H$ reads equivalently $H = \frac{1}{2}\tr(-\I[\vy]^{-1/2}\II[\vy]\I[\vy]^{-1/2})$ and $\I[\vy]=g$. We stress that the constant $C$ depends on $J[\vy]=\det g$ and so only on $\lambda$. The bending energy \eqref{eq:bouck-bending-energy-simple} corresponds to a term in the energy for prestrained plates \cite{efrati2009,lewicka2016} and will be useful for the discussion of Section \ref{sec:method}.

%----------------------------------------------------------------------------------------
\section{Asymptotic profiles of defects}\label{sec:special-solution}
%----------------------------------------------------------------------------------------

%In contrast to free LCs, the rubber of an LCN constrains the LC, and higher order defects are stable for LCNs. 
We now develop a new formal asymptotic method to construct asymptotic profiles for blueprinted director fields $\vm$ with point defects degree greater than 1, which builds on known solutions for defects of degree $1$ and $1/2$. 
Our asymptotic method relies on several key ingredients, including the concept of lifted surface \cite{plucinsky2018actuation}, composition of director fields, and a formal Taylor expansion. Together, these components constitute a novel procedure for deriving asymptotic profiles of solutions when dealing with general point defect of degrees greater than 1.
We are not aware of studies of shapes beyond the Gauss curvature obtained in \cite{modes2011blueprinting} for higher degree defects. Our approximate solutions provide insight on the complicated shapes that can be programmed upon actuation. We reproduce these profiles computationally later in Section \ref{sec:dir_defect}.

%----------------------------------------------------------------------------------------
\subsection{Lifted surfaces}\label{sec:lifted-surface}
%----------------------------------------------------------------------------------------
Lifted surfaces for LCNs/LCEs are originally introduced in \cite{plucinsky2018actuation}. We adapt the idea to the reduced model \eqref{eq:stretching-energy} in this subsection. To this end, we consider the following parameterization of lifted surfaces 
\begin{equation}\label{eq:lifted-surface-1}
\vy^{l}(\vx') =  \big(\alpha \vx',\phi(\alpha\vx') \big)^T \quad \forall\vx'\in\Omega, 
\end{equation}
where $\alpha\in\mathbb{R}$ will be determined later. Here, $\phi:\alpha\Omega\to \mathbb{R}$ represents the graph of the lifted surfaces. Our goal is to match the metric $g$ in \eqref{eq:target-metric} with $\I[\vy^l]$, i.e,
\begin{equation}\label{eq:metric-lifted-1}
\I[\vy^l] = g = \lambda^2\vm\otimes\vm+ \lambda^{-1}\vm_{\perp}\otimes\vm_{\perp} = (\lambda^2 - \lambda^{-1})\vm\otimes\vm+\lambda^{-1}\Id_2.
\end{equation}
Since \eqref{eq:lifted-surface-1} yields
\begin{equation}\label{eq:metric-lifted-2}
\I[\vy^l] = \alpha^2\nabla\phi(\alpha\vx')\otimes \nabla\phi(\alpha\vx') +\alpha^2\Id_2,
\end{equation}
\eqref{eq:metric-lifted-1} is valid if $\phi$ satisfies $|\nabla\phi|=\sqrt{\lambda^3-1}$ a.e. in $\Omega$, and $\alpha=\lambda^{-1/2}$, with the properties that $\lambda>1$ and $\lambda$ is constant over $\Omega$. Substituting them into \eqref{eq:lifted-surface-1} gives
\begin{equation}\label{eq:lifted-surface-2}
  \vy^{l}(\vx') =  \Big(\vxl',\phi(\vxl')\Big)^T,
  \quad
  \vxl':=\lambda^{-1/2}\vx'.
\end{equation}
Since this deformation is an isometric immersion of the metric \eqref{eq:target-metric}, it is also an equilibrium configuration provided $\vm(\vx')=\pm (\lambda^3-1)^{-1/2}\nabla\phi(\vxl')$ according to Corollary \ref{cor:minimizers}. We observe that the discussion so far has restricted $\lambda>1$, which means the LCN is being cooled. If $\lambda<1$ and $\phi$ satisfies $\pm\sqrt{\lambda^{-3} - 1}\nabla \phi (\lambda \vx') = \vm_\perp(\vx')$, then a lifted surface of the form 
\begin{equation}\label{eq:lifted-surface-heat}
\vy^{l}(\vx') =  \Big(\lambda \vx',\phi(\lambda \vx')\Big)^T,
\end{equation}
satisfies $\I[\vy^{l}] = g$.  Since a lifted surface may be constructed for $\lambda<1$ in a similar fashion as for $\lambda>1$, we restrict the remaining discussion of this section to $\lambda>1$. However, we note that the computations in Section \ref{sec:dir_defect} typically set $\lambda<1$.

%
%--------------------------------------------------------------------------------------
\subsection{Surfaces for defects of degree $1$ and $1/2$}\label{sec:metric-1-defects}
%--------------------------------------------------------------------------------------

To set the stage, we first go over known lifted surfaces that arise from degree 1 and degree 1/2 defects. These solutions will match the metric $g$ exactly, and will help us later in constructing approximate solutions for higher order defects in Sections \ref{sec:metric-two-defects} and \ref{sec:metric-3half-defects}.

A director field $\vm_1$ with a defect of degree $1$ reads
\begin{equation}\label{eq:degree1}
\vm_1(\vx') = \frac{\vx'}{|\vx'|}.
\end{equation}
If $\vR_1$ is a rotation of $\pm\pi/2$, the corresponding exact solution $\vy_1$ for $\vR_1\vm_1$ reads
\begin{equation}\label{eq:cone}
\vy_1(\vx') = \Big(\vxl', \phi_1(\vxl')\Big)^T
\end{equation}
where 
\begin{equation}\label{eq:cone-phi}
\phi_1(\vx') = \sqrt{\lambda^3-1}\;(1 - |\vx'|);
\end{equation}
$\vy_1$ is a {\it cone} with vertex at the origin as long as $\lambda>1$  \cite{modes2011gaussian}. If $\lambda<1$, then the cone solution in \eqref{eq:cone-phi} is no longer well defined. In fact, the director field $\vm_1$ in \eqref{eq:degree1} will produce what is known as an {\it anticone} configuration  \cite{modes2011gaussian}. The solution for a degree 1 defect will be a cone or anti-cone depending on the angle $\alpha_r$ between $\vm_1$ and $\vx'$ as well as $\lambda$ \cite{mostajeran2016encoding}. 
See Fig.~\ref{fig:cone-anticone} for the cone and anti-cone shapes computed by our algorithm.

\begin{figure}[htbp]
\includegraphics[width=.7\textwidth]{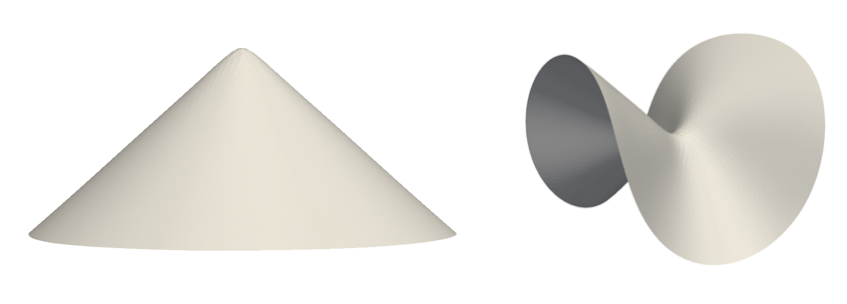}
\caption{Computed solution with the blueprinted director field $\vm_1$ that has degree $1$ defect, $\lambda<1$ and $\alpha_r=0\text{ (right)},\pi/2\text{ (left)}$. We refer to Section \ref{sec:dir_defect} for details of these numerical simulations.}
\label{fig:cone-anticone}
\end{figure}

Next, we introduce a solution induced by a director field with a degree $1/2$ defect,
which will help us construct an approximate solution for a degree $3/2$ defect in Section \ref{sec:metric-3half-defects}. Motivated by \cite{modes2011blueprinting}, we consider the director field
\begin{equation}\label{eq:degree-1half}
\vm_{1/2}(\vx') = 
\begin{cases}
	\text{sign}(x_2)\ve_2, & x_1\geq 0 \\
	\frac{\vx'}{|\vx'|}, & x_1<0 \, ,
\end{cases}
\end{equation}
and note that $\vm_\perp\otimes\vm_\perp$ is the typical line field for a defect of degree $1/2$ at the origin. Since $\vm_{1/2}(\vx') = \vm_1(\vx')$ when $x_1<0$, we expect a cone configuration forming in the left half-plane. For $\lambda>1$, an exact solution is given by the lifted surface configuration
\begin{equation}\label{eq:soln-degree-1half}
\vy_{1/2}(\vx') = \Big(\vxl', \phi_{1/2}(\vxl') \Big)^T,
\end{equation}
where
\begin{equation}\label{eq:lifted-degree-1half}
\phi_{1/2}(\vx') = \begin{cases} 
	\sqrt{\lambda^3-1}\; (1 - |x_2|),& x_1\geq 0 \\
	\sqrt{\lambda^3-1}\; (1 - |\vx'|), & x_1<0
\end{cases}.
\end{equation}
This entails stretching in the direction $\vm_{1/2}$ and shrinking in the perpendicular direction $\vm_{1/2}^\perp$, which in turn explains the shape of the membrane in Fig.~\ref{fig:deg-half} for $x_1>0$.
We see that when $x_1<0$, the map $\vy_{1/2}$ coincides with the cone in \eqref{eq:cone}. We plot $\vm_{1/2}$ (left), $\vy_{1/2}$ (middle) and our computed solution (right) in Fig.\ref{fig:deg-half}.

\begin{figure}[htbp]
\includegraphics[width=0.95\textwidth]{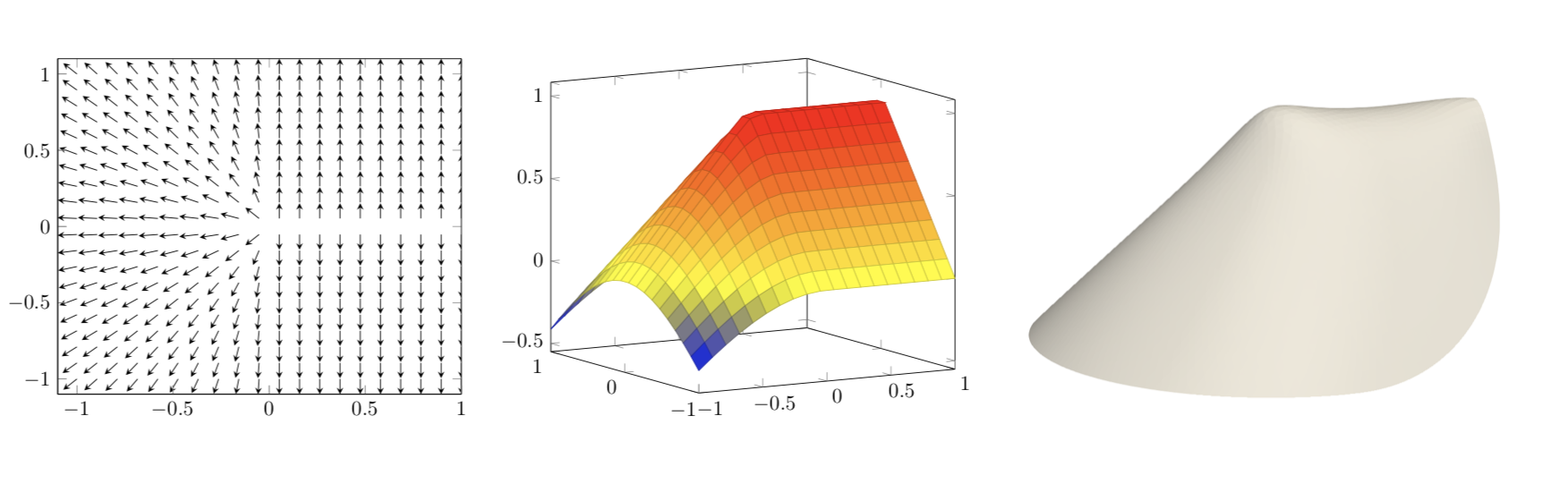}
\caption{Director field $\vm_{1/2}$ from \eqref{eq:degree-1half} (left), lifted surface $\vy_{1/2}$ from \eqref{eq:soln-degree-1half}-\eqref{eq:lifted-degree-1half} for $\lambda =2^{1/3}$ (middle), and computed solution in a unit disc domain with $\vm=\vm_{1/2}$ and a Dirichlet boundary condition that is compatible with \eqref{eq:soln-degree-1half}-\eqref{eq:lifted-degree-1half} (right). Note that the gradient of $\phi_{1/2}$ is parallel to $\vm_{1/2}$ whereas $\vm_{1/2}^\perp$ is the typical director field for a $1/2$ defect.}
\label{fig:deg-half}
\end{figure}

%--------------------------------------------------------------------------------------
\subsection{Higher degree defects: main idea and idealized construction }\label{sec:higher-degree-idea}
%--------------------------------------------------------------------------------------
We now consider rotationally symmetric blueprinted director fields $\vm_n$ with defects of integer degree $n>1$. Such director fields are given in polar coordinates by
\begin{equation}\label{eq:m_rot_sym_1}
\vm_{n}(r,\theta) = \big(\cos(n\theta),\sin(n\theta) \big).
\end{equation}
We observe that the line field $\vm_n\otimes\vm_n$ given by $\vm_n$ in \eqref{eq:m_rot_sym_1} exhibits a discontinuity at the origin. We denote by $g_n$ the metric generated by $\vm_{n}$ and arbitrary $\lambda$ via \eqref{eq:target-metric}. Ideally, the goal is to build on the solution \eqref{eq:cone} for $n=1$ and composition of defects to obtain a solution $\vy_n$ with degree $n$ defect. The main idea is as follows.

We exploit the relation to the complex-valued function $f_n(z) = e^{in\arg(z)}$ to write
\[
\vm_n({\bf x}') = p^{-1} \big(f_n(p(\vx'))\big),
\]
where $z=|z| e^{i\arg(z)}$ for any $z\in\mathbb{C}$ and $p:\mathbb{R}^2\to \mathbb{C}$ is the map $p(\vx') = x_1+i x_2$.
From this perspective, we can write a director field with degree $n$ defect as the multiplication or composition of two director fields with degree $1$ and $n-1$ defects
\[
e^{in\text{arg}(z)} = e^{i\text{arg}(z)}e^{i(n-1)\text{arg}(z)}.
\]
If $\vm_1:=(\mu_1,\mu_2)$ and $\vm_{n-1} := (\xi_1,\xi_2)$, then $\mu_1 + i \mu_2 = e^{i\arg(p(\vm_1))}$ and $\xi_1 + i \xi_2 = e^{i\arg(p(\vm_{n-1}))}$ imply
\[
e^{i \arg(p(\vm_n))} = (\mu_1 + i \mu_2) (\xi_1 + i \xi_2) = (\mu_1\xi_1-\mu_2\xi_2) + i (\mu_1\xi_2 + \mu_2\xi_1).
\]
Applying $p^{-1}$ to both sides yields
\begin{equation}\label{eq:recursion}
  \vm_n =  \begin{pmatrix} \mu_1 & -\mu_2 \\ \mu_2 & \mu_1\end{pmatrix}
  \begin{pmatrix} \xi_1 \\ \xi_2\end{pmatrix} = \vR_1\vm_{n-1},
\end{equation}
where $\vR_1:= \big( \vm_{1},\vm_{1}^\perp\big)$ is a rotation matrix that depends on $\vx'$. In view of \eqref{eq:metric-lifted-1} we may write the metric $g_n$ at $\vx'$ as
\begin{align}\label{eq:defect-comp}
g_n = (\lambda^2-\lambda^{-1})\vR_1 \big(\vm_{n-1}\otimes \vm_{n-1}\big)\vR_1^T + \lambda^{-1}\Id_2. 
\end{align}
Assuming $\lambda>1$, we compare \eqref{eq:defect-comp} with the metric that arises from function composition of two defects of degree $1$ and $n-1$. With $\vxl'=\lambda^{-1/2}\vx'$ already defined in \eqref{eq:lifted-surface-2}, we consider the following modified lifted surface
\begin{equation}\label{eq:lifted-surface-deg2}
  \vy_n(\vx') := \Big(\vxl', \phi_n \big(\vv(\vxl')\big)\Big)^T.
\end{equation}
Compared to either \eqref{eq:cone} or \eqref{eq:soln-degree-1half} of Section \ref{sec:lifted-surface}, we now compose $\phi_n$ with an unknown function $\vv:\lambda^{-1/2}\Omega\to\lambda^{-1/2}\Omega$. We then apply the chain rule to determine
\begin{align}\label{eq:func-comp}
&\I[\vy_n(\vx')] = \lambda^{-1}\Id_2 +\lambda^{-1}\nabla \vv(\vxl')^T \big(\nabla\phi_n(\vv(\vxl'))\otimes\nabla\phi_n(\vv(\vxl'))\big)\nabla \vv(\vxl').
\end{align}
To match \eqref{eq:defect-comp} an ideal construction would be to find $\phi_n$ and $\vv$ so that $\nabla\phi_n(\vv(\vxl')) = \sqrt{\lambda^3-1}\;\vm_{n-1}(\vx')$ and $\nabla \vv(\vxl') = \vR_1(\vx')^T$. We will find $\phi_n$ in terms of $\phi_{n-1}$, but before we do so we need to argue with $\vv$. An ideal $\vv$ should have a gradient whose rows are $\vm_1$ and $\vm_1^\perp$. Since $\vm_1$ points radially outward and $\vm_1^\perp$ is tangent to concentric circles, the choice of $\vv$ in polar coordinates should be $\vv(r,\theta) = (v_1(r), v_2(\theta))$ in order for the rows of $\nabla\vv$ to be parallel to $\vm_1$ and $\vm_1^\perp$. One such choice of $\vv$ is
\begin{equation}\label{eq:def-v-polar}
\vv(r,\theta) = \begin{pmatrix} a\log r \\ a \, \theta \end{pmatrix} 
\end{equation}
for $a>0$, whose gradient in Euclidean coordinates is formally
\begin{equation}\label{eq:grad-v}
\nabla \vv(\vx') = \frac{a}{|\vx'|}\vR_1(\vx')^T.
\end{equation}
The choice of $v_1(r) = a\log r$ is so that the scaling of $\frac{1}{r}$ matches the gradient of $v_2(\theta) = a\theta$. Here, $\nabla \vv$ matches $\vR_1(\vx)^T$ up to the scaling $\frac{a}{|\vx'|}$, and we nearly recover the ideal $\vv$. Finding a vector field $\vv$ whose gradient equals a space-dependent rotation $\vR_1(\vx')^T$ is questionable. In fact, in order for $\curl\big(\psi(r)\vm_1^\perp(\vx)\big) = 0$ and potentially have an antiderivative, the only choice of $\psi$ is $\psi(r) = \frac{a}{r}$. Therefore, $\psi(|\vx'|) = \frac{a}{|\vx'|}$ is the only scaling for which one may hope to find an antiderivative of $\psi(|\vx'|)\vR_1(\vx')^T$. The choice of $\phi_n$ is designed to compensate for this scaling. If
\begin{equation}\label{eq:def-phi-n}
  \phi_n(\vx') := \frac{\sqrt{\lambda^3-1}}{n \, a^{n-1}} \, |\vx'|^n
  \quad\Rightarrow\quad
  \nabla\phi_n(\vx') = \frac{|\vx'|}{a} \nabla\phi_{n-1}(\vx'),
\end{equation}
which is consistent with \eqref{eq:cone-phi} for $n=1$. Combining the inductive hypothesis
\[
\nabla\phi_{n-1}\big(\vv(\vxl')\big) = \sqrt{\lambda^3-1} \, \vm_{n-1}(\vx'),
\]
with the recursion relation \eqref{eq:def-phi-n} yields
\begin{align*}
\nabla\big[\phi_n\big(\vv(\vxl')\big)\big] &= \nabla \vv(\vxl')^T \nabla \phi_n\big(\vv(\vxl')\big)
= \frac{|\vv(\vxl')|}{|\vxl'|} \vR_1(\vxl') \nabla\phi_{n-1}\big(\vv(\vxl')\big)
\\
& = \frac{|\vv(\vxl')|}{|\vxl'|} \sqrt{\lambda^3-1} \vR_1(\vx') \vm_{n-1}(\vx')
= \frac{|\vv(\vxl')|}{|\vxl'|} \sqrt{\lambda^3-1} \, \vm_n(\vx').
\end{align*}
This shows that we need $|\vv(\vx')|=|\vx'|$ to close the argument, which may not be possible unless $\vv(\vx')=\vR(\vx')^T\vx'$ with $\vR(\vx')$ a rotation. This in turn would not lead to \eqref{eq:grad-v}. Finally, the cone solution for $n=1$ satisfies $\nabla\phi_1(\vxl') = \pm\sqrt{\lambda^3-1}\;\vm_1(\vx')$, whereas the ideal construction requires $\nabla\phi_1\big(\vv(\vxl')\big) = \pm\sqrt{\lambda^3-1}\;\vm_1(\vx')$. The sign does not matter because $g_1$ is invariant under $\vm_1\mapsto -\vm_1$, but there is a mismatch in the argument of $\nabla\phi_1$ since $\vv(\vxl')$ may not be equal to $\vxl'$ everywhere. We next discuss how to circumvent these obstructions to the idealized construction via approximation. 

%--------------------------------------------------------------------------------------
\subsection{Formal approximation of idealized construction}\label{sec:formal-approximation}
%--------------------------------------------------------------------------------------
We now build an approximate deformation $\vy_n$ such that $\I[\vy_n] \approx g_n$. To this end, we modify $\vv$ from \eqref{eq:def-v-polar}, so that $\vv(\vxl') \approx \vxl'$ near the point $\vx^* = (a,0)^T$ for $a>0$; this avoids a singularity at $\bf 0$. To guarantee that $\vv(\vx^*)=\vx^*$ and $\nabla\vv(\vx^*)=\Id_2$, we choose
\begin{equation}\label{eq:def-v}
\vv(\vx') = \begin{pmatrix}\frac{a}{2}\text{log}(x_1^2+x_2^2)+C_a \\ a\,\text{arctan}(x_2/x_1)\end{pmatrix}, 
\end{equation}
where  $C_a = a~ -~ a\log(a).$  Hence, $\vv$ satisfies \eqref{eq:grad-v} and the formal Taylor expansion
\begin{align*}
 \vv(\vx') = \vx^*+ (\vx'-\vx^*)+ \mathcal{O}\big(|\vx'-\vx^*|^2\big) = \vx'+\mathcal{O}\big(|\vx'-\vx^*|^2\big),
\end{align*}
or equivalently the following expression in the rescaled coordinates $\vxl'$
\begin{equation}\label{eq:v-taylor-scaled}
 \vv(\vxl') = \vxl'+\mathcal{O}\big(|\vx'-\vx^*|^2\big),
\end{equation}
because $\lambda=\mathcal{O}(1)$.
Using \eqref{eq:v-taylor-scaled}, we approximately satisfy the three crucial requirements
\begin{align*}
|\vv(\vxl')|^2 &= |\vxl'|^2 +\mathcal O(|\vx'-\vx^*|^2),
\\
\vm_1(\vv(\vxl')) &= \vm_1(\vxl')+\mathcal O(|\vx'-\vx^*|^2),
\\
\nabla\big[\phi_n\big(\vv(\vxl')\big)\big] &= \sqrt{\lambda^3-1} \, \vm_n(\vx') + \mathcal O(|\vx'-\vx^*|^2).
\end{align*}
Inserting these formal approximations into \eqref{eq:func-comp} yields a map $\vy_n$ defined by \eqref{eq:lifted-surface-deg2} for $n\ge2$ that approximately satisfies the metric constraint in a vicinity of $\vx^*$
\begin{equation}\label{eq:approx-metric-g2}
\I[\vy_n(\vx')] = g_n(\vx')+\mathcal O(|\vx'-\vx^*|^2).
\end{equation}
%

%--------------------------------------------------------------------------------------
\subsection{Approximate surfaces for defects of degree two}\label{sec:metric-two-defects}
%--------------------------------------------------------------------------------------
We now specialize the above construction for $n=2$. In view of \eqref{eq:def-phi-n}, we realize that
\[
\phi_2(\vx') := \frac{\sqrt{\lambda^3-1}}{2a}|\vx'|^2
\quad\Rightarrow\quad
\nabla \phi_2(\vx') =  \frac{\sqrt{\lambda^3-1}}{a}\vx' = \frac{\sqrt{\lambda^3-1}}{a}|\vx'|\vm_1(\vx'),
\]
Hence, \eqref{eq:lifted-surface-deg2} gives an approximate map $\vy_2$ with
\begin{equation*}
\phi_2\big(\vv(\vx')\big) = \frac{\sqrt{\lambda^3-1}}{2a}\left(\left(\frac{a}{2}\log\big(x_1^2+x_2^2\big)+a - a\log a\right)^2 + a^2\text{arctan}^2\Big(\frac{x_2}{x_1}\Big)\right),
\end{equation*}
for $x_1>0$ and any $a>0$ not be too large so that $\vy_2$ captures the correct defect configuration. 
We display $\phi\circ \vv$ for $a=.75,\lambda = 1.1$ in Fig.~\ref{fig:deg2-taylor}, reflected for $x_1<1$ to account for symmetry, along with the computed solution from Section \ref{sec:dir_defect}.

\begin{figure}[htbp]
\includegraphics[width=.45\textwidth]{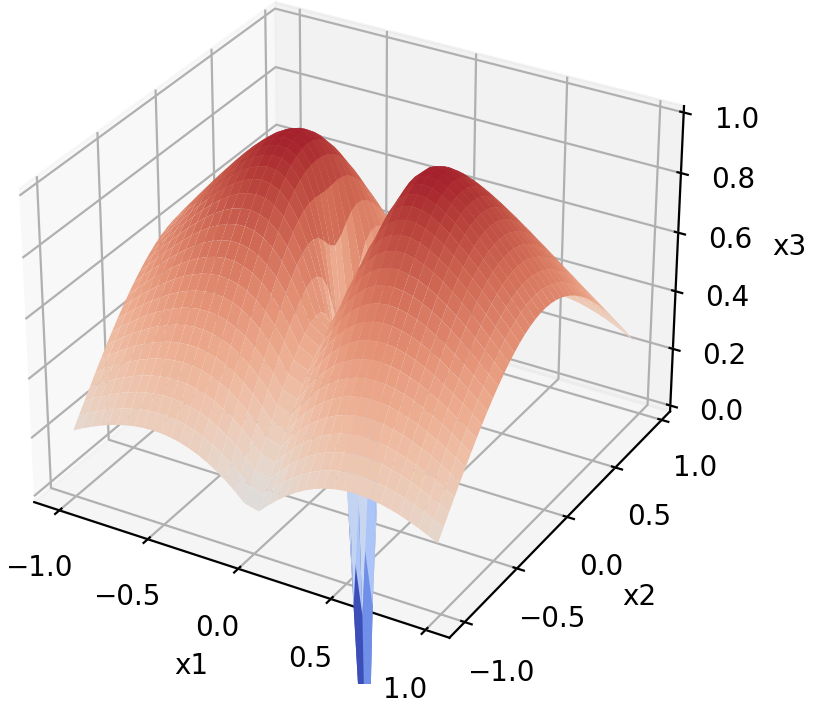}
\includegraphics[width=.4\textwidth]{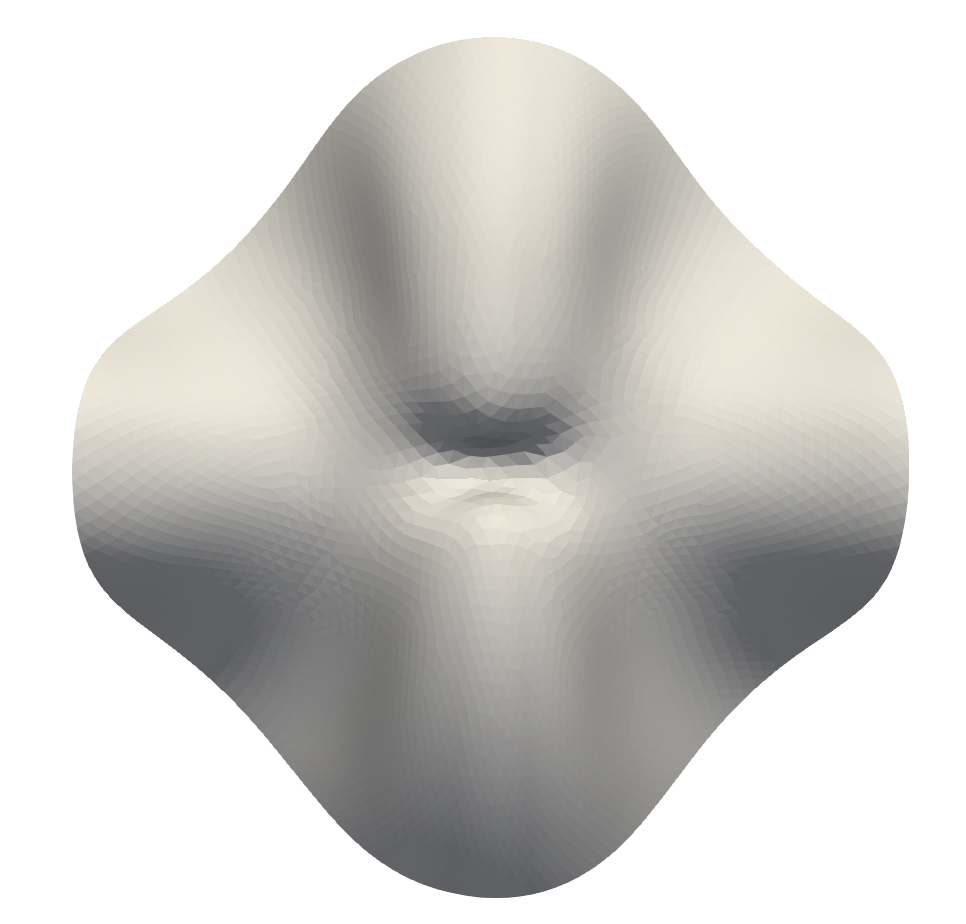}
\caption{Approximate lifted surface for degree 2 defect (left) and computed solution with the director field $\vm_2$ in Section \ref{sec:dir_defect} (right). Our derivation requires $x_1>0$, but the solution should be symmetric across the $x_2x_3$ plane, which is why we plot a reflected solution for $x_1<0$. We recover two bumps, consistent with the simulation but at the cost of a singularity at the origin.}
\label{fig:deg2-taylor}
\end{figure}

%--------------------------------------------------------------------------------------
\subsection{Approximate surface for degree 3/2 defect}\label{sec:metric-3half-defects}
%--------------------------------------------------------------------------------------
We now apply the above approach of composing defects, but for a defect of degree $3/2$. We intend to explain the intriguing ``bird beak" shape observed in our computations displayed in Figs.~\ref{fig:deg-3half-bird-beak} and  \ref{fig:comparison_1}. We first observe that the explicit expressions \eqref{eq:soln-degree-1half} and \eqref{eq:lifted-degree-1half} for a defect of degree $1/2$ do not quite conform with \eqref{eq:m_rot_sym_1} for $n=1/2$, except in the vicinity of the origin. Motivated by the recursion relation \eqref{eq:recursion}, we still write the degree $3/2$ director field as
\begin{equation}\label{eq:degree-3half}
\vm_{3/2}(\vx') = \vR_1(\vx')\vm_{1/2}(\vx'),
\end{equation}
with $\vm_{1/2}$ given in \eqref{eq:degree-1half}. We now construct an approximate map $\vy_{3/2}$ such that 
\begin{equation*}
\I[\vy_{3/2}(\vx')] \approx g_{3/2}(\vx') = (\lambda^2-\lambda^{-1})\vm_{3/2}(\vx')\otimes \vm_{3/2}(\vx') +\lambda^{-1}\Id_2,\end{equation*}
according to \eqref{eq:defect-comp} for $\lambda>1$. The deformation $\vy_{3/2}$ satisfies in turn \eqref{eq:lifted-surface-deg2}, namely
\begin{equation}\label{eq:y3half}
\vy_{3/2}(\vx')= \big(\vxl',\phi_{3/2}(\vv(\vxl')) \big)^T
\end{equation}
with $\phi_{3/2}$ related to $\phi_{1/2}$ via \eqref{eq:def-phi-n}. Since we are interested in an approximation for $x_1>0$ to capture the ``bird beak" structure, we deduce from \eqref{eq:lifted-degree-1half}
\[
\nabla \phi_{3/2}(\vx') = \frac{|\vx'|}{a} \nabla\phi_{1/2}(\vx')
\quad\Rightarrow\quad
\nabla \phi_{3/2}(\vx') = \sqrt{\lambda^3-1} \, \frac{|\vx'|}{a} \, \text{sign}(x_2) \,\ve_2.
\]
Unfortunately, this ideal relation is incompatible because $\curl\big( \frac{|\vxl'|}{a}\text{sign}(x_2)\ve_2\big)\neq0$ and we need to amend the construction of $\phi_{3/2}$ by approximation. To this end, we define for $x_1>0$ the following modification of $\phi_{3/2}$
\begin{equation}\label{eq:phi-3half-2}
\phi_{3/2}(\vx') = \sqrt{\lambda^3-1}\Big(1 - \frac{1}{a}\int_{0}^{|x_2|}\sqrt{s^2+x_1^2}\; ds\Big),
\end{equation}
whose gradient is
\begin{equation*}
\nabla\phi_{3/2}(\vx') = -\sqrt{\lambda^3-1}\frac{|\vx'|}{a}\text{sign}(x_2) \ve_2 -\sqrt{\lambda^3-1}\frac{1}{a}\left(\int_0^{|\vx_2|} \frac{x_1}{\sqrt{s^2+x_1^2}}\; ds\right)\ve_1.
\end{equation*}
Exploiting that the integrand in the second term is bounded by $1$ yields
\begin{equation*}
\nabla\phi_{3/2}(\vx) = -\sqrt{\lambda^3-1}\,\frac{|\vx|}{a}\,\text{sign}(x_2) \, \ve_2 -\frac{\sqrt{\lambda^3-1}}{a}\mathcal{O}(|x_2|).
\end{equation*}
To approximate the first fundamental form $\I[\vy_{3/2}]$, we recall \eqref{eq:func-comp} and compute
\begin{equation*}\label{eq:nabla-phi-nabla-v-1half}
\nabla\vv(\vx_{\lambda}')^T\nabla\phi(\vv(\vx_{\lambda}')) = \sqrt{\lambda^3-1}\vR_1(\vx_{\lambda}')\frac{|\vv(\vx_{\lambda}')|}{|\vx_{\lambda}'|}\text{sign}(\vv(\vx_{\lambda}')_2) \ve_2+ \mathcal{O}(|\vv(\vx_{\lambda}')_2|),
\end{equation*}
where $\vv(\vx_{\lambda}')_2=a \arctan(x_2/x_1)$ denotes the second component of $\vv(\vx_{\lambda}')$ written in \eqref{eq:def-v}. We thus deduce $\text{sign}(\vv(\vxl')_2) = \text{sign}(x_2)$ for $x_1>0$ and, employing that $\vm_{1/2}(\vx) = \text{sign}(x_2)\,\ve_2$ for $x_1>1$ along with \eqref{eq:degree-3half}, we arrive at
\begin{equation*}
\nabla\vv(\vx_{\lambda}')^T\nabla\phi(\vv(\vx_{\lambda}')) = \sqrt{\lambda^3-1} \, \vm_{3/2}(\vx')+ \mathcal{O}(|x_2|)+\mathcal{O}(|\vx' - \vx^*|^2),
\end{equation*}  
because $\arctan(x_2/x_1)=\mathcal{O}(|x_2|)$ for $x_1$ away from $0$. The expression \eqref{eq:y3half} for $\vy_{3/2}$ with $\phi_{3/2}$ defined in \eqref{eq:phi-3half-2} gives an approximate shape profile that satisfies
\begin{equation*}
\I[\vy_{3/2}(\vx')] = g_{3/2}(\vx') + \mathcal O(|x_2|) + \mathcal O(|\vx - \vx^*|^2).
\end{equation*}
The contour plot of the corresponding lifted surface $\phi_{3/2}\big(\vv(\vxl')\big)$ is displayed in Fig.~\ref{fig:deg-3half-bird-beak} (left) for $a = .75$. We note that the profile has a similar bird beak shape to the computational result reported in Fig.~\ref{fig:deg-3half-bird-beak} (right) and Fig.~\ref{fig:comparison_1}.
For the $x_1<0$, $\vm_{1/2}(\vx') = \vm_1(\vx')$, and one can apply the arguments in Section \ref{sec:metric-two-defects} to get the asymptotic profile for $x_1<0$.

\begin{figure}[htbp]
\includegraphics[width=.4\textwidth]{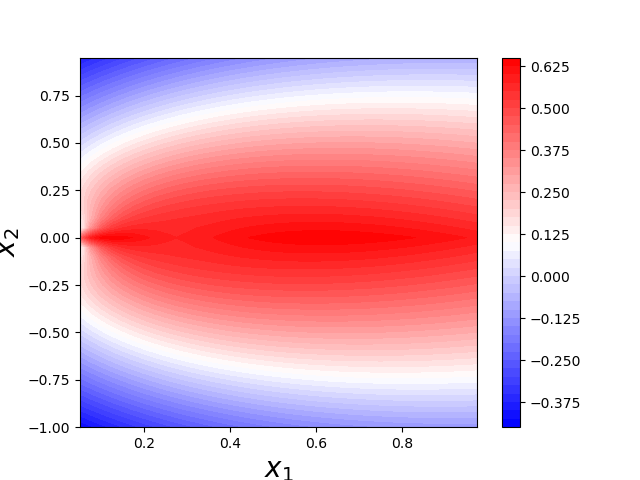}
\includegraphics[width=.33\textwidth]{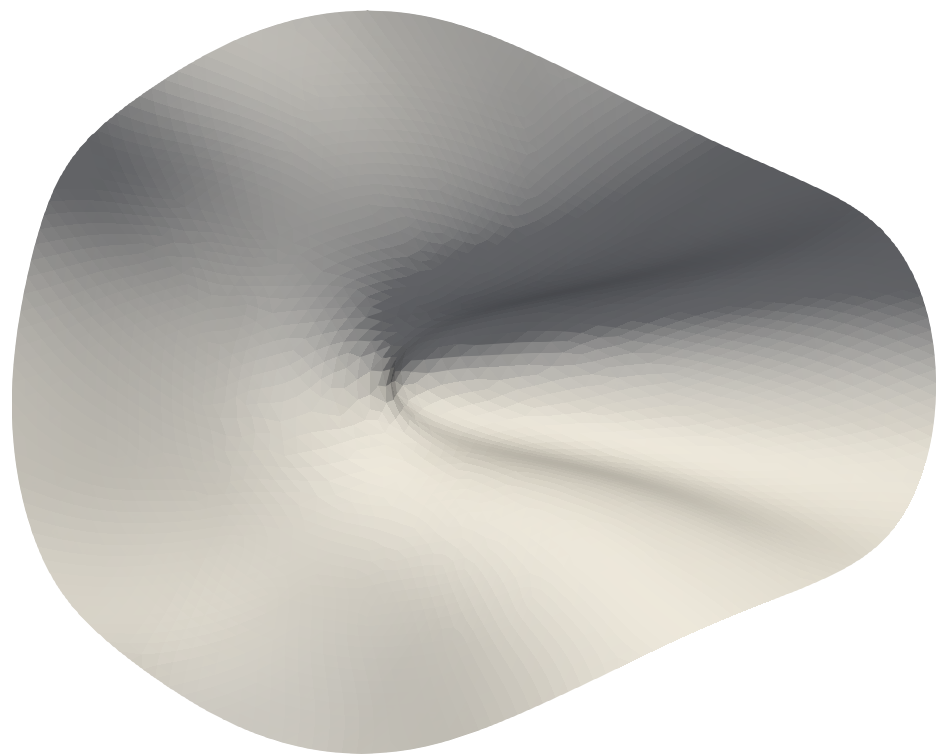}
\caption{Contour plot of approximate lifted surface for degree 3/2 defect for $a = .75$ and $x_1>0$ (left) and computational result for a degree 3/2 defect obtained in Section \ref{sec:dir_defect}. The profile matches the computed ``bird beak" shape. To see this, notice that contour lines pinch off as $x_1\to0$. As a result, the lifted surface gets steeper near the origin. This helps explains the ``bird beak" shape.}
\label{fig:deg-3half-bird-beak}
\end{figure}

The compositional method explains why we should expect the intriguing ``bird beak". We now provide a heuristic explanation. If $a$ is fixed but $x_1$ is small, we drop $x_1$ in the integrand of \eqref{eq:phi-3half-2}, and $\phi_{3/2}(\vx')$ behaves like $\tilde{\phi}_{3/2}(\vx') := \sqrt{\lambda^3-1}\big(1 - \frac{1}{2a}|x_2|^2\big)$. We see that level sets of this function are straight lines $|x_2| =$ constant that increase as $|x_2|$ decreases to $0$, very much like level sets of $\phi_{1/2}(\vx')$ in \eqref{eq:lifted-degree-1half} for $x_1>0$. On the other hand, the lifted surface $\phi_{3/2}(\vv(\vxl'))$ from \eqref{eq:y3half} behaves like 
\begin{equation}\label{eq:tilde-phi}
\tilde{\phi}_{3/2}(\vv(\vxl')) = \sqrt{\lambda^3-1}\Big(1 - \frac{a}{2} \, \arctan \Big(\frac{|x_2|}{x_1}\Big)^2\Big) ,
\end{equation}
whose level sets are radial lines $\frac{|x_2|}{x_1} = $ constant that increase as $\frac{|x_2|}{x_1}$ decreases to $0$. Therefore, the lifted surface $\tilde\phi_{3/2}(\vv(\vxl'))$ pinches off at the origin in the sense that it develops a discontinuity. In Section \ref{sec:higher-degree-idea} we advocated that a defect of degree $3/2$ could be viewed as a composition of degree $1/2$ and $1$ defects. The effect of the degree $1$ defect on \eqref{eq:soln-degree-1half} is to twist or compress the horizontal level sets of $\phi_{1/2}(\vv(\vxl'))$ into radial level sets of $\tilde\phi_{3/2}(\vv(\vxl'))$. This is due to the action of the vector-valued map $\vv$ and boils down to the replacement of $|x_2|$ in \eqref{eq:lifted-degree-1half} by $\frac{a}{2} \, \arctan \big( \frac{|x_2|}{x_1}\big)^2$ in \eqref{eq:tilde-phi}.

%%%%%%%%%%%%%%%%%%%%%%%%%%%%%%%%%%%%%%%%%%%%%%%%%%%%%%%%%%%%%%%%%%%%%%%%%%%%%%%%%%%%%%%%%
\section{Finite element method with regularization}\label{sec:method}
%%%%%%%%%%%%%%%%%%%%%%%%%%%%%%%%%%%%%%%%%%%%%%%%%%%%%%%%%%%%%%%%%%%%%%%%%%%%%%%%%%%%%%%%%
In this section, we introduce a finite element discretization to \eqref{eq:stretching-energy} and a nonlinear iterative scheme to solve the ensuing nonconvex discrete \lucas{minimization} problem.

%----------------------------------------------------------------------------------------
\subsection{Discretization}
%----------------------------------------------------------------------------------------
We consider a shape-regular family $\{ \Th \}$
of simplicial meshes of $\Omega$
parametrized by the mesh size $h = \max_{T \in \Th} h_T$,
where $h_T = \diam(T)$.
We denote by $\mathcal{N}_h$ the set of vertices of $\Th$, and by $\E_h$ the set of interior edges of $\Th$.

For any $T \in \Th$, we denote by $\mathcal{P}_1(T)$
the space of first-degree polynomials on $T$
and by $\mathcal{N}_h(T)$ the set of vertices of $T$.
We consider the space of $\Th$-piecewise affine and globally continuous vector-valued functions
\begin{equation}\label{eq:fem-space}
\V_h
:=
\left\{\vv_h \in C^0(\overline{\Omega};\mathbb{R}^3): \vv_h \vert_T \in[\mathcal{P}_1(T)]^3 \text{ for all } T \in\Th \right\},
\end{equation}
and note that $\V_h\subset H^1(\Omega;\R^3)$. We write $\V_h^{\ast}$ to designate the dual space of $\V_h$ and denote by $I_h$ the nodal Lagrange interpolation operator $I_h: C^0(\overline{\Omega};\R^3) \to\V_h$.

We now introduce the jump operators. To this end, let $\vn_e$ be a unit normal to $e\in\Eh$ (the orientation is chosen arbitrarily but is fixed once for all). Given a scalar piecewise polynomial function $v_h$ over $\Th$ and $e \in \E_h$, we set
\begin{equation} \label{def:jump-avrg}
\restriction{\jump{v_h}}{e} := v_h^{-}-v_h^+, 
\end{equation}
where $v_h^{\pm}(\vx'):=\lim_{s\rightarrow 0^+}v_h(\vx'\pm s\vn_e)$ for $\vx' \in e$; jumps of vector-valued and tensor-valued functions are computed component-wise. Note that for $\vv_h\in\V_h$, we have $\restriction{\jump{\vv_h}}{e}=0$ for any $e\in\Eh$, while $\restriction{\jump{\nabla\vv_h}}{e} \ne 0$ in general.

We propose the discrete energy $E_h:\V_h\to \mathbb{R}$ to be
\begin{equation}\label{eq:discrete-str}
E_h[\vy_h] = E_{str}[\vy_h] + R_h[\vy_h],
\end{equation}
where the first term is the stretching energy defined in \eqref{eq:stretching-energy} and \eqref{eq:stretching-energy-density}. $ R_h$ is the regularization term defined in \eqref{eq:regularization}, \lucas{namely}
\begin{equation}\label{eq:E-reg}
R_h[\vy_h] := \sum_{e\in\E_h}\int_{e} c_r h_e |\jump{\nabla\vy_h}|^2,
\end{equation}
\lucas{and allows for} a non-negative regularization \lucas{mesh function} $c_r:\E_h\to\mathbb{R}$
% , which can vanish on certain edges to accommodate foldings
\lucas{, which we set to zero on certain edges to accommodate folding in a pre-specified pattern. We expand on this later in Section \ref{sec:regularization} below.} 
%
%In \eqref{eq:E-reg}, $R_h$ allows for the introduction of a non-negative regularization parameter $c_r:\E_h\to\mathbb{R}$. \lucas{To prescribe the folding pattern $\Gamma$, we set $c_r|_e = 0$ on all edges $e\in\E_h$ such that $e\subset \Gamma$. }
Here, $h_e= \text{diam}(e)$ represents the length of edge $e$. 

This is a crucial computational feature of the discrete counterpart of \eqref{eq:stretching-energy}, which reads
\begin{equation}\label{eq:minimization-pb-discrete}
\vy^{\ast}_h\in\argmin_{\vy_h\in\V_h}E_h[\vy_h].
\end{equation} 

%------------------------------------------------------------------------------------------
\subsection{Regularization}\label{sec:regularization}
%------------------------------------------------------------------------------------------
This section is dedicated to a discussion of the regularization term $R_h$ in \eqref{eq:E-reg}. We add $R_h$ to the stretching energy $E_{str}$, given in \eqref{eq:stretching-energy} and \eqref{eq:stretching-energy-density}, to deal with the lack of weak lower semi-continuity of $E_{str}$. We refer to our accompanying paper \cite{bouck2022NA} for a detailed discussion of the lack of weak lower semi-continuity and other possible strategies to treat it. We also prove in \cite{bouck2022NA} convergence of discrete minimizers of \eqref{eq:minimization-pb-discrete}. 

We first point out that, in the context of discontinuous Galerkin (DG) methods (for instance \cite{cockburn2012discontinuous,bonito2021numerical}), a discrete $H^2$ semi-norm for functions in $\V_h$ is defined as
\begin{equation*}\label{eq:discrete-H2}
\|D^2_h\vy_h\|_{L^2(\Omega; \mathbb{R}^{3\times2\times2})}^2+\sum_{e\in\E_h}h_e^{-1}\|\jump{\nabla\vy_h}\|_{L^2(e;\mathbb{R}^{3\times 2})}^2+\sum_{e\in\E_h}h_e^{-3}\|\jump{\vy_h}\|_{L^2(e;\R^3)}^2,
\end{equation*}
where $D^2_h$ denotes a broken Hessian on every element. Since $\vy_h\in\V_h$ is piecewise affine and continuous, only the second term remains while other terms vanish. This motivates our choice of regularization $R_h[\vy_h]$ in \eqref{eq:E-reg}, which is proportional to this discrete $H^2$ semi-norm provided $c_r$ is uniformly positive and $\Th$ is quasi-uniform. Moreover, we can regard the ratio $h_e^{-1}|[\nabla\vy_h]|$ between jumps of constant $\nabla\vy_h$ in contiguous elements to $e\in\E_h$ to the length of $e$ as a discrete Hessian of $\vy_h$ associated to $e$. The following definition of \lucas{element-wise norm of the Hessian} reveals the significance of $R_h[\vy_h]$:
\[
\int_T\big|H_h[\vy_h] \big|^2 := h_T \int_{\partial T} \frac{|[\nabla\vy_h]|^2}{h_T^2}
\quad\Rightarrow\quad
R_h[\vy_h] \approx \sum_{T\in\Th} c_r h_T^2 \int_T \big|H_h[\vy_h] \big|^2.
\]

We emphasize that the full reduced model of LCNs derived in \cite{bouck:2023thesis} and reported in Section \ref{sec:bending-energy} reads $E_{str} + t^2 E_{bend}$, where the principal term of the bending energy density $W_{bend}$ in the case $s=0$ in \eqref{eq:bouck-bending-energy-simple} is given by
\begin{equation}\label{eq:bending}
W_{bend} [\vy] := C \int_\Omega \big| \tr (g^{-\frac12} \II[\vy] g^{-\frac12}) \big|^2.
\end{equation}
We recall that $ \II[\vy] = (\partial_{ij}\vy \cdot \vnu)_{ij=1}^2\in\R^{2\times2}$ stands for the second fundamental form of the surface $\vy(\Omega)$. 
%It turns out that \eqref{eq:bending} coincides with $W_{bend}$ in \cite{ozenda2020blend} provided $\det\I[\vy]=1$. 
Moreover, \eqref{eq:bending} is a component of the bending energy for {\it prestrained plates} \cite{efrati2009,lewicka2016},
%%
%\[
%W_{bend} [\vy] = \int_\Omega C_1 \big| g^{-\frac12} \II[\vy] g^{-\frac12} \big|^2 + C_2\big| \tr (g^{-\frac12} \II[\vy] g^{-\frac12}) \big|^2.
%\]
%%
and it is important to realize that this energy \eqref{eq:bending} can be equivalently written as \cite{bonito2022ldg,bonito2021numerical}
%
%\[
%W_{bend} [\vy] = \int_\Omega c_1 \big| g^{-\frac12} D^2\vy g^{-\frac12} \big|^2 + c_2\big| \tr (g^{-\frac12} D^2\vy g^{-\frac12}) \big|^2 + c_3(g),
%\]
\[
W_{bend} [\vy] = \int_\Omega c_1\big| \tr (g^{-\frac12} D^2\vy g^{-\frac12}) \big|^2 + c_2(g),
\]
at the expense of a term $c_2(g)$ that only depends on the given data $g$ but not on $\vy$, whence neglecting $c_2(g)$ does not change the minimizers; we refer to \cite{bartels2015numerical,bartels2017bilayer} for similar computations. This reveals that the effect of bending energy $\int_\Omega W_{bend} [\vy]$ with the density defined in \eqref{eq:bending} is similar to the simplified energy $\int_\Omega |D^2\vy|^2$ advocated in \cite[Eq. (1.11)]{plucinsky2018actuation}.

  This discussion provides physical justification for the structure of the regularization $R_h[\vy_h]$ in \eqref{eq:E-reg}: for $c_r>0$ constant and $\Th$ quasi-uniform, we realize that
\[
R_h[\vy_h] \approx c_r h^2 \int_\Omega \big| H_h[\vy_h] \big|^2
\]
mimics $t^2 \int_\Omega |D^2\vy|^2$ with $h$ being interpreted as a discrete thickness parameter. Moreover, allowing $c_r$ to vanish over a polygonal $\Gamma$ made of edges of $\E_h$ mimics discretely a material amenable to folding across $\Gamma$ \cite{bartels2022modeling,bartels2022error}. We prove the following statement about convergence of minimizers of \eqref{eq:minimization-pb-discrete} in \cite{bouck2022NA}.

\begin{thm}[convergence of discrete minimizers]\label{T:convergence}
Let $\Omega\backslash\Gamma = \cup_{i=1}^I\Omega_i$ be a decomposition of $\Omega$ into disjoint subdomains $\Omega_i$ due to the creases $\Gamma$, which are matched by the mesh $\Th$. Let the regularization parameter $c_r$ vanish on all edges contained in $\Gamma$. Let the target metric $g$ in \eqref{eq:target-metric} admit an isometric immersion $\vy\in W^{1,\infty}(\Omega;\mathbb{R}^3)$, i.e. $\I[\vy] = g$ a.e. in $\Omega$, that satisfies $\vy|_{\Omega_i}\in H^2(\Omega_i;\mathbb{R}^3)\cap C^1(\overline{\Omega}_i;\mathbb{R}^3)$ for all $i=1,\ldots,I$. Then there is a constant $\Lambda>0$ such that discrete minimizers $\vy_h^*$ of \eqref{eq:minimization-pb-discrete} satisfy the energy scaling
\begin{equation}\label{eq:energy-scaling}
E_h[\vy_h^*] \le \Lambda h^2.
\end{equation}
In addition, if a minimizer $\vy_h^*$ satisfies \eqref{eq:energy-scaling} and $\overline{\vy}_h^* := |\Omega|^{-1} \int_\omega\vy_h^*$ is its meanvalue, then there exists a subsequence (not relabeled) such that $\vy_h^* -  \overline{\vy}_h^* \to \vy^*$ converges strongly in $H^1(\Omega;\R^3)$ to a function $\vy^*\in W^{1,\infty}(\Omega;\mathbb{R}^3)$, that satisfies $\vy|_{\Omega_i}\in H^2(\Omega_i;\mathbb{R}^3)$ for all $i=1,\ldots,I$ and $\I[\vy^*]=g$ a.e. in $\Omega$, or equivalently
\begin{equation}\label{eq:energy-limit}
  E_{str,\Gamma}[\vy^*]=\int_{\Omega\backslash\Gamma} W_{str}(\vx',\nabla \vy^*) d\vx' =  0.
\end{equation}

\end{thm}
 
We regard \eqref{eq:E-reg} as a mechanism for equilibria selection. In the absence of regularization, i.e. $c_r=0$ in \eqref{eq:E-reg}, minimizers $\vy_h^*\in\V_h$ of \eqref{eq:minimization-pb-discrete} can exhibit extra bumps and wrinkling which have negligible influence on the stretching energy. This is a manifestation of lack of convexity of $W_{str}$, and thus of uniqueness, and leads to the formation of micro-structure \cite{bartels2015numerical}. In our model the liquid crystal director is frozen, so we expect mechanical wrinkling. We point to \cite[Example 2.8]{bouck2022NA} for an explicit construction showing the possibility of wrinkling. This topic is well studied in  nonlinear elasticity, both theory and computation, but it is not the focus of this paper. We resort to \eqref{eq:E-reg} to suppress numerical oscillations in Section~\ref{sec:eff-reg} as well as to allow for folding in the development of compatible origami-structures in Section \ref{sec:origami} and incompatible origami-structures in Section \ref{sec:incompatible-s}. The latter lead to weak limits $y^*\in H^1(\Omega;\R^3)$ with $E_{str}[\vy^*]>0$, so they are not minimizers of the stretching energy $E_{str}$.  

%------------------------------------------------------------------------------------------
\subsection{Iterative solver}\label{sec:newton-scheme}
%------------------------------------------------------------------------------------------
We design a nonlinear discrete gradient flow to find a solution to \eqref{eq:minimization-pb-discrete} in this subsection. Due to the stretching energy being non-quadratic and non-convex, we end up with a nonlinear non-convex discrete problem to solve.

%------------------------------------------------------------------------------------------
\subsubsection{Nonlinear gradient flow.}
%------------------------------------------------------------------------------------------
Implicit gradient flows are robust methods to find stationary points of energy functionals $E$ regardless of their convexity, and have the advantage of built-in energy stability; they belong to the class of energy descent methods. Consider the auxiliary evolution equation $\partial_tA[\vy] + \delta E[\vy]=0$, where $A$ is a symmetric elliptic operator, $\delta E[\vy]$ stands for the first variation of $E$, and $t$ is a pseudo-time. The backward Euler discretization reads: given $\vy^n$ solve for $\vy^{n+1}$
\begin{equation*} 
\frac{1}{\tau}\big(A\vy^{n+1}- A\vy^n \big) + \delta E[\vy^{n+1}]=0,
\end{equation*}
where $\tau$ is the time-step discretization parameter. The weak formulation of this semi-discrete equation is equivalent to minimizing the augmented functional
\begin{equation}\label{eq:penalization}
L^n[\vy] := \frac{1}{2\tau}\|\vy-\vy^n\|_A^2 + E[\vy]
\end{equation}
where $\|\cdot\|_A$ is the norm associated with the operator $A$, i.e. $\|\vy\|_A^2 := \langle A\vy,\vy\rangle$. This can be reinterpreted as finding a minimizer of $E$ constrained to be closed to $\vy^n$; so the first term in \eqref{eq:penalization} penalizes the deviation of $\vy$ from $\vy^n$ in the $A$-norm.

Since the stretching energy $E_{str}$ from \eqref{eq:stretching-energy} and \eqref{eq:stretching-energy-density} is formulated in $H^1(\Omega;\R^3)$, we choose $A=I-\Delta$ and the corresponding norm to be the $H^1(\Omega;\R^3)$-norm. This choice has the property of making $L^n_h$ convex in $H^1(\Omega;\R^3)$ provided $\tau$ is sufficiently small and $\det\I[\vy^n_h]$ is bounded away from 0. With this in mind, we devise a discrete counterpart of \eqref{eq:penalization} to find stationary points of $E_h$ in \eqref{eq:discrete-str} and, under some additional assumptions on the current iterate $\vy_h^n\in\V_h$ to be discussed below, solve \eqref{eq:E-reg}. We thus minimize
\begin{align}\label{eq:gf-energy}
L^n_h[\vy_h] :=\frac{1}{2\tau}\|\vy_h-\vy^n_h\|^2_{H^1(\Omega;\R^3)}+E_h[\vy_h],
\end{align}
whose Euler-Lagrange equation results from computing the first order variation of $L^n_h[\vy_h]$ in the direction $\vv_h$
\begin{equation}\label{eq:gf-onestep}
  \delta L^n_h[\vy_h](\vv_h)=\frac{1}{\tau}(\vy_h,\vv_h)_{H^1(\Omega;\R^3)}+\delta E_h[\vy_h](\vv_h)-F^n_h(\vv_h)=0
  \quad\forall \, \vv_h\in\V_h,
\end{equation}
where $F^n_h\in\V_h^{\ast}$ is defined as
\begin{align*}
F^n_h(\vv_h):=\frac{1}{\tau}(\vy_h^n,\vv_h)_{H^1(\Omega;\R^3)},
\end{align*}
and
\begin{align*}
  \delta E_h[\vy_h](\vv_h)
  =\delta E_{str}[\vy_h](\vv_h)+\delta R_h[\vy_h](\vv_h).
\end{align*}
 Moreover, given a tolerance $\tol_1>0$, we stop the nonlinear gradient flow when
\begin{equation*}
\frac{1}{\tau} \big|E_h[\vy_h^{N}]-E_h[\vy_h^{N-1}]\big|~<~ \tol_1
\end{equation*}
is satisfied for some $N>0$. The function $\vy_h^{N}\in\V_h$ is the desired output. We next discuss how we solve the nonlinear equation \eqref{eq:gf-onestep}.

We solve each step $n$ of \eqref{eq:gf-onestep} by a {\it Newton-type sub-iteration}: letting $\vy_h^{n,0}:=\vy_h^{n}$ and assuming $\vy_h^{n,k}\in\V_h$, we solve for the increment $\delta\vy_h^{n,k}\in\V_h$
\begin{equation}\label{eq:Newton-step}
  \delta^2 L^n_h[\vy^{n,k}_h](\delta\vy_h^{n,k},\vv_h)=-\delta L^n_h[\vy^{n,k}_h](\vv_h)
  \quad\forall\,\vv_h\in\V_h,
\end{equation}
In Theorem \ref{thm:Newton-conv-one} we summarize properties of this Newton iteration, and refer to \cite[Appendix A]{bouck2022reduced} for complete proofs and discussion.
\begin{thm}[quadratic estimate]\label{thm:Newton-conv-one}
For any $n\ge0$ and $k\ge0$, suppose $\vy^{n,k}_h$ satisfies
\begin{equation}\label{eq:cond-yh}
0<c_1 \le \lambda_1[\vy_h^{n,k}]  \leq  \lambda_2[\vy_h^{n,k}] \le c_2 \quad\forall\, T\in\Th,
\end{equation}
where $\lambda_1[\vy_h^{n,k}]$ and $\lambda_2[\vy_h^{n,k}]$ are eigenvalues of $\I[\vy_h^{n,k}]$. If $\tau$ is smaller than a constant depending only on $c_1,c_2,\lambda,\vm$, then \eqref{eq:Newton-step} is well-posed and $\delta\vy^{n,k}_h$ is the unique solution. Moreover, if $\vy^{n,k+1}_h:=\vy^{n,k}_h+\delta\vy^{n,k}_h$, then
\begin{equation}\label{eq:quadratic-conv}
\|\vy_h^{n,k+1}-\vy_h^{n,\ast}\|_{H^1(\Omega;\mathbb{R}^3)}\le\frac{C}{2h}\|\vy_h^{n,k}-\vy_h^{n,\ast}\|_{H^1(\Omega;\mathbb{R}^3)}^2,
\end{equation}
where $C$ is a constant that depends on $c_1,c_2,\lambda$ and $\vy_h^{n,\ast}$ is a local minimizer of $L^n_h$.
\end{thm}

Theorem \ref{thm:Newton-conv-one} yields quadratic convergence under a smallness assumption on $\tau$. To see this, let $\vy_h^{n,0}$ satisfy \eqref{eq:cond-yh} and 
\begin{equation}\label{eq:sub-initial}
\|\vy_h^{n,0}-\vy_h^{n,\ast}\|_{H^1(\Omega;\mathbb{R}^3)} \le Ch.
\end{equation}
Then for $k\ge0$ an induction argument combined with \eqref{eq:quadratic-conv} yields
\begin{equation}\label{eq:error-monotone}
\|\vy_h^{n,k+1}-\vy_h^{n,\ast}\|_{H^1(\Omega;\mathbb{R}^3)}\le \frac12 \|\vy_h^{n,k}-\vy_h^{n,\ast}\|_{H^1(\Omega;\mathbb{R}^3)} < Ch.
\end{equation}
This implies that the sub-iterations $\vy^{n,k}_h$ remain within an $H^1$-ball of radius $Ch$ centered at $\vy^{n,\ast}_h$ and converge to $\vy^{n,\ast}_h$; in view of \eqref{eq:quadratic-conv} this convergence is quadratic.

It remains to check whether the initialization condition \eqref{eq:sub-initial} is realistic. Assume that $E_h[\vy^0_h]\le A_0$ for a constant $A_0>0$ and recall that $\vy_h^{n,0}=\vy_h^n$ to deduce
\begin{equation*} 
\frac{1}{2\tau}\|\vy_h^{n,\ast}-\vy_h^n\|^2_{H^1(\Omega;\mathbb{R}^3)}\le L^n_h[\vy^{n,\ast}_h]\le L^n_h[\vy^n_h]=E_h[\vy^n_h]\le E_h[\vy^0_h]\le A_0.
\end{equation*}
Consequently, if $\tau\le\frac{C^2h^2}{2A_0}$ then \eqref{eq:sub-initial} is valid. However, quantitative numerical experiments in our accompanying paper \cite{bouck2022NA} reveal that the largest admissible value of $\tau$ is independent of $h$. This is strikingly better than our theoretical prediction.

%------------------------------------------------------------------------------------
\section{Computing Shapes: Defects and Origami Structures}\label{sec:simulations}
%------------------------------------------------------------------------------------
We have implemented Algorithm~\ref{algo:GF_euler} below using the multiphysics
finite element software Netgen/NGSolve \cite{schoberl2017netgen},
whereas the visualization relies on ParaView~\cite{Ahrens2005}.
\RestyleAlgo{boxruled}
\begin{algorithm}[htbp]
    \SetAlgoLined
	Given a pseudo time-step $\tau>0$ and target tolerances $\tol_1$ and $\tol_2$\;
	Choose initial guess $\vy_h^0\in\V_h$\;
	\While{$\tau^{-1}\big|E_h[\vy_h^{n+1}]-E_h[\vy_h^{n}]\big|>\tol_1$}{
		\textbf{Set} $\vy^{n,0}_h=\vy_h^n$, $k=0$\;
		\While{$\big|\delta L^{n}_h[\vy_h^{n,k}](\delta\vy_h^{n,k})\big|^{1/2}> \tol_2$}{
		\textbf{Solve} \eqref{eq:Newton-step} for $\delta\vy_h^{n,k}$\;
		\textbf{Update} $\vy^{n,k+1}_h=\vy^{n,k}_h+\delta\vy_h^{n,k}$, $k=k+1$\;
		}
		\textbf{Update} $\vy^{n+1}_h:=\vy_h^{n,k}$, where $k$ is the index of last sub-iterate.
	}
	\caption{(nonlinear gradient flow scheme)} \label{algo:GF_euler}
\end{algorithm}

We recall that global minimizers to \eqref{eq:stretching-energy} are characterized by the pointwise metric relation \eqref{eq:metric-cons}, namely $\I[\vy]=g$. Therefore, if $\vy^{\infty}_h$ denotes the output of Algorithm~\ref{algo:GF_euler}, we quantify the deviation of $\I[\vy^{\infty}_h]$ from the target metric $g$ by
\begin{equation}\label{eq:metric-deviation}
e_h^1[\vy^{\infty}_h] :=  \big\| \I[\vy^{\infty}_h] \, - \, g \big\|_{L^1(\Omega;\mathbb{R}^{2\times2})},
\end{equation}
and employ it as an error indicator between the approximate solution $\vy^{\infty}_h$ and an exact global minimizer to \eqref{eq:stretching-energy}.

We report quantitative properties of Algorithm~\ref{algo:GF_euler} in \cite{bouck2022NA}, namely
\begin{itemize} 
\item The metric deviation $e_h^1[\vy^{\infty}_h]$ converges as $\mathcal{O}(h)$. 
\item The energy $E_h[\vy^{\infty}_h]$ converges as $\mathcal{O}(h^2)$ when $\vm$ is smooth.
\item The energy $E_h[\vy^{\infty}_h]$ converges sub-quadratically when $\vm$ has a point defect.
\end{itemize}
The focus of this section, instead, is on the ability of Algorithm~\ref{algo:GF_euler} to capture quite appealing and practical physical phenomena related to shape formation. In Section \ref{sec:numerical-benchmark}, we validate our discrete reduced model \eqref{eq:minimization-pb-discrete} against configurations from the literature such as existing laboratory experiments or modeling results. In Section \ref{sec:incompatible-s}, we further explore incompatible nonisometric origami and present novel simulations for LCNs, which offer valuable insights for laboratory experiments.

%------------------------------------------------------------------------------------

\subsection{Benchmark shapes: defects and compatible nonisometric origami}\label{sec:numerical-benchmark}
We validate our reduced mathematical model and computational method against known equilibrium configurations originating from defects and creases.

%------------------------------------------------------------------------------------
\subsubsection{Rotationally symmetric director fields and defects}\label{sec:dir_defect}
%------------------------------------------------------------------------------------
Let $\Omega~\subset~\mathbb{R}^2$ be the unit disc and, motivated by \cite{modes2011gaussian,white2015programmable,mcconney2013topography,chung2017finite}, let the blueprinted director field $\vm\in\mathbb{S}^1$ be a rotation of \eqref{eq:m_rot_sym_1} by an angle $\alpha$ with degree $n$ 
\begin{equation}\label{eq:m_rot_sym_2}
\vm(r,\theta) = \big(\cos(n(\theta+\alpha)),\sin(n(\theta+\alpha))\big);
\end{equation}
$\vm$ is discontinuous at the origin. We run Algorithm~\ref{algo:GF_euler} with several values of $n$ and $\alpha$ and display the output in Fig.~\ref{fig:comparison_1}. We use physical and numerical parameters
\[
s=0.1, ~ s_0=1; \quad \textrm{tol}_1=10^{-6}, ~ \textrm{tol}_2=10^{-10},
~ \tau=0.1, ~ h=1/32, ~ c_r=1.
\]
We initialize Algorithm~\ref{algo:GF_euler} with $\vy_h^0=I_h\vy^0$, where
\begin{equation}\label{eq:initialization}
  \vy^0(\vx)~=~ \big(\vx,0.05(1-|\vx|^2) \big)
  \quad \vx\in\Omega
\end{equation}
is a small perturbation of a flat disc (i.e. $\vy(\vx)=(\vx,0)$), and $I_h:C^0(\Omega;\mathbb{R}^3)\to \mathbb{V}_h$ is the Lagrange interpolation operator. Metric deviations $e_h^1[\vy^{\infty}_h]$ and energies $E_h[\vy^{\infty}_h]$ are at the scale of $10^{-2}$ and $10^{-3}$ respectively, indicating that Algorithm~\ref{algo:GF_euler} produces quite accurate approximations to global minimizers of \eqref{eq:stretching-energy}.
\begin{figure}[htbp]
\includegraphics[width=13.cm]{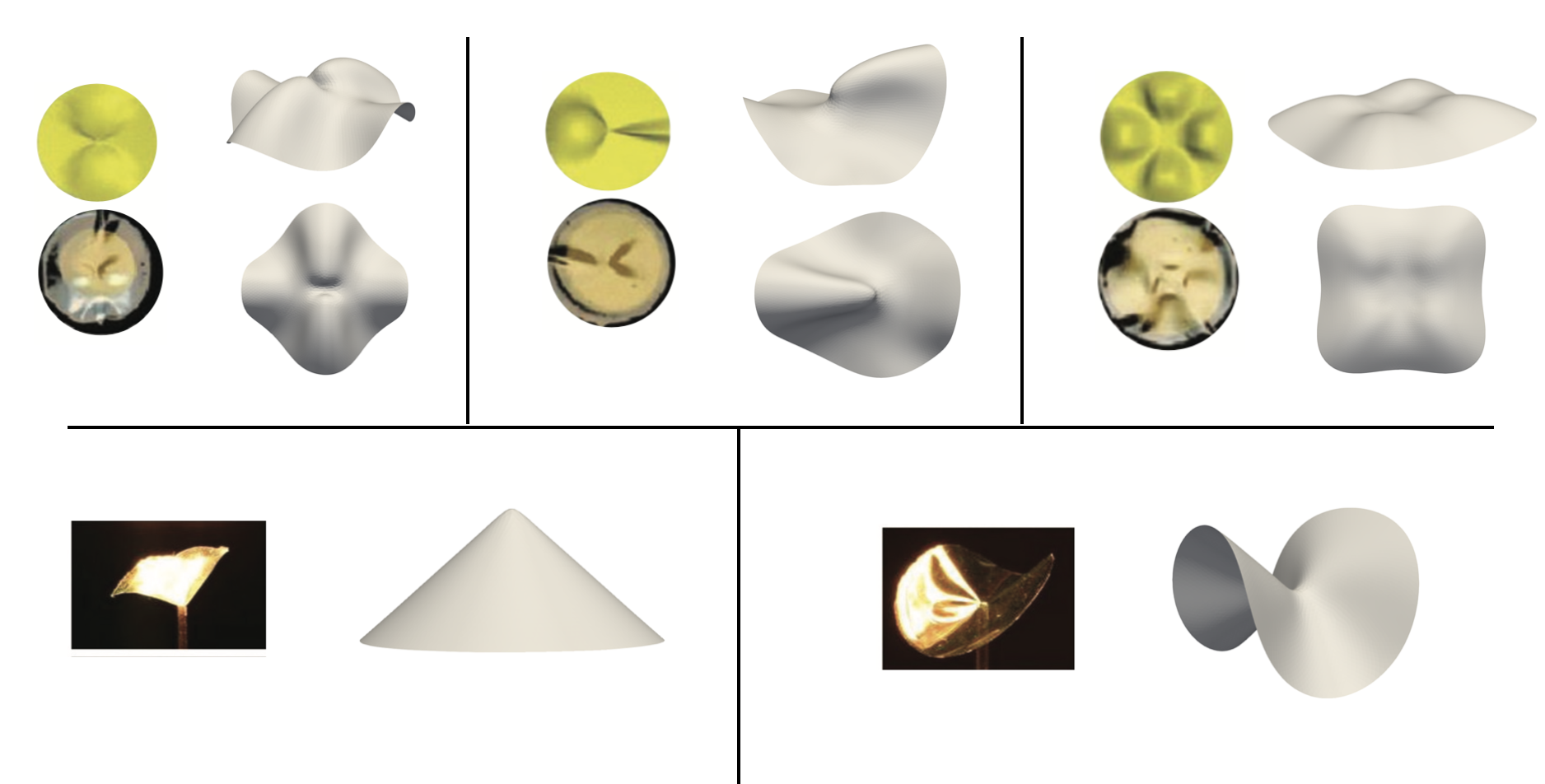}
\caption{\emph{Director fields with point defects of degree $n$}. First row displays $n=2,3/2,-1$ and $\alpha=0$ (from left to right). Each panel shows experimental and expected configurations from \cite{mcconney2013topography} as well as two views of the computed solution. Second row depicts experimental pictures from \cite{de2012engineering} and our simulations of the cone structure $n=1, \alpha=\frac{\pi}{2}$ (left) and anti-cone structure $n=1, \alpha=0$ (right). The numerical model reproduces experimental observations well.}
\label{fig:comparison_1}
\end{figure}

Fig.~\ref{fig:comparison_1} illustrates the ability of Algorithm~\ref{algo:GF_euler}, namely the reduced discrete model \eqref{eq:minimization-pb-discrete}, to capture physical phenomena. Comparing the computed shapes with experimental and expected configurations in \cite{modes2011gaussian,white2015programmable}, we find striking similarities.

%

%----------------------------------------------------------------------------------
\subsubsection{High order defects and role of regularization}\label{sec:eff-reg}
%----------------------------------------------------------------------------------
%
In Fig.~\ref{fig:deg_pm5}, we present computed shapes corresponding to director fields \eqref{eq:m_rot_sym_2} with $n=\pm5$ and $\alpha=0$; otherwise the setting is the same as in Section \ref{sec:dir_defect}. It is worth noting that conducting laboratory experiments for high-order defects poses practical challenges \cite{bauman2023private}. Conversely, our computational method and model offer a convenient and efficient means of predicting these shapes. Our results align with the observed trend found in \cite{mcconney2013topography}: higher-order defects degrees correlate with larger number of oscillations within the disc. However, a quantitative study of wave-lengths seems to be lacking.

Furthermore, we investigate the computational impact of the regularization parameter $c_r$ in Fig.~\ref{fig:deg_pm5} (left and middle) by comparing $c_r=1$ and $c_r=0.2$ for degree $n=5$. We observe that for $c_r$ small, self-similar wrinkling patterns emerge around the origin, where the point defect is located. Conversely, a larger value of $c_r$ eliminates these wrinkles, thereby providing further support for the interpretation in Section \ref{sec:regularization} of regularization as a pseudo bending energy term. The nature of these wrinkles, whether numerical oscillations or physical microstructures, is an open question worth investigating.

\begin{figure}[htbp]
\includegraphics[width=1.\textwidth]{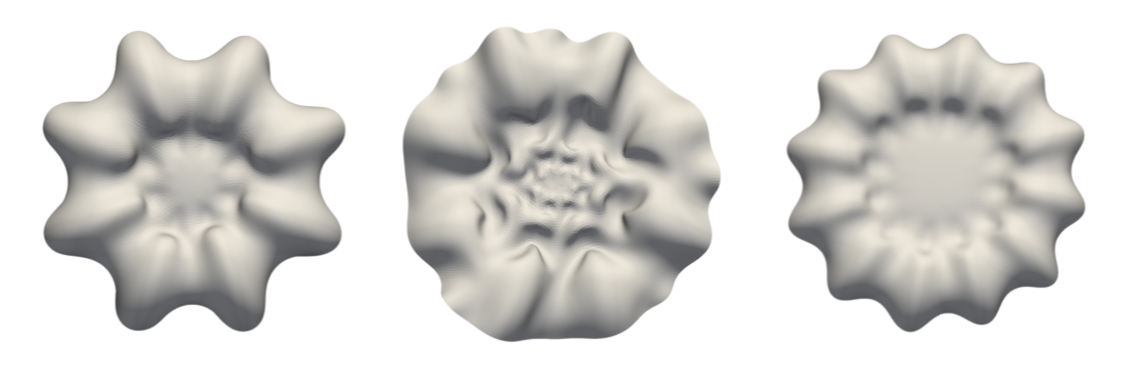}
\vskip-0.5cm
\caption{\emph{Director fields with point defects of degree $\pm5$}. Left: $n=5$ and $c_r=1$. Middle: $n=5$ and $c_r=0.2$ (notice oscillations close to the origin). Right: $n=-5$ and $c_r=1$. Left and middle computations illustrate the role of regularization.}
\label{fig:deg_pm5}
\end{figure}

%------------------------------------------------------------------------------------
\subsubsection{Compatible nonisometric origami}\label{sec:origami}
%------------------------------------------------------------------------------------
Origami are ancient structures  made of folding thin sheets. They have recently attracted growing interest in the area of materials science for its practical value in the design of medical devices, deployable space structures and robotics \cite{zhao2022twisting}. In this section we explore whether our discrete model \eqref{eq:minimization-pb-discrete} with regularization is able to capture such structures.

Motivated by \cite{plucinsky2016programming,plucinsky2018actuation,plucinsky2018patterning}, we embark on a computational study of nonisometric origami formed by LCNs. Unlike isometric origami shapes, that arise from pure folding and bending mechanisms, nonisometric origami result from satisfying the metric constraint \eqref{eq:metric-cons}, i.e. $\I[\vy]=g$, or equivalently minimizing the stretching energy $E_{str}$ of \eqref{eq:stretching-energy}; nonisometric origami thus exhibit stretching and shearing.
In each origami experiment below, we split the domain $\Omega$ into several subdomains by ``folding lines'' or ``creases'' $\gamma_i$ for $i=1,\ldots,N_{\gamma}$, and denote the set of them by $\Gamma$. We consider fitted meshes to $\Gamma$, i.e., $\Gamma\subset \bigcup_{e\in \mathcal{E}_h}e$; we also provide computations of a pyramidal structure in \cite{bouck2022NA} where the mesh is not fitted to $\Gamma$, confirming that the computed origami structures are not a result of mesh effect.

\medskip
{\it Compatibility}:
We assume the blueprinted director field $\vm$ to be constant in each subdomain. We say the set-up of nonisometric origami is compatible if
\begin{itemize}
\item $\vm$ satisfies the compatibility condition proposed in \cite[formula (6.3)]{plucinsky2018actuation}, namely 
\[
|\vm_{\gamma_i}^+\cdot\vt_{\gamma_i}|=|\vm_{\gamma_i}^-\cdot\vt_{\gamma_i}|,
\] 
for any $i=1,\ldots,N_{\gamma}$,
where $\vt_{\gamma_i}$ represents a unit tangent vector to $\gamma_i$ and $\vm_{\gamma_i}^\pm$ denote $\vm$ restricted to the two subdomains that share $\gamma_i$;
\item The actuation parameters $s,s_0$, and thus the parameter $\lambda$ defined in \eqref{eq:lambda}, are continuous across $\gamma_i$ for $i=1,\ldots,N_{\gamma}$.  
\end{itemize}
The compatibility condition means that the tangential component of the line field $\vm\otimes\vm$ and parameter $\lambda$ are continuous across $\gamma_i$. Therefore, since any equilibrium configuration satisfies the metric constraint \eqref{eq:metric-cons} with metric $g$ defined in \eqref{eq:target-metric}, such configuration sustains compatible stretching on both sides of a folding line $\gamma_i$.

%Moreover, we exploit the fact that the regularization parameter $c_r$ may depend on the edge $e$ to incorporate the creases $\Gamma$ in the discrete energy $E_h$ in \eqref{eq:discrete-str}. 
\lucas{Moreover, we tune the regularization parameter $c_r:\E_h\to \mathbb{R}$ to incorporate the creases $\Gamma$ in the discrete energy $E_h$ in \eqref{eq:discrete-str}. }
We thus take, unless specified otherwise, regularization parameter $c_r=0$ along the folding lines $\Gamma$ and $c_r=100$ in the rest of domain, i.e., we rewrite \eqref{eq:E-reg} as
\begin{equation}\label{eq:reg+folding}
R_h[\vy_h] = c_{r}h\sum_{e\in\E_h\setminus\Gamma}\int_{e}|\jump{\nabla\vy_h}|^2.
\end{equation}
In fact, the zero regularization (no jumps of gradient) models a weakened (or damaged) material on creases \cite{bartels2022modeling}, and mathematically this allows for the formation of kinks. On the other hand, the large regularization in the subdomains serves as a mechanism to prevent bending. Consequently, equilibrium configurations of $E_h$ prefer flat surfaces and folds to meet the target metric \eqref{eq:metric-cons}. 

Furthermore, since the exact solutions of compatible origami can be represented as piecewise affine functions over the mesh $\Th$, as long as $\Th$ aligns with the creases, our discretization approach employing piecewise linear finite elements \eqref{eq:fem-space} reproduces such solutions exactly. In fact, discrete energies $E_h$ of computed solutions in Sections \ref{sec:table} and \ref{sec:cube} achieve values close to zero, reaching levels between $10^{-8}$ to $10^{-6}$, even with rather coarse meshes. Any remaining error solely stems from the optimization process outlined in Algorithm \ref{algo:GF_euler}.     

%--------------------------------------------------------------------------------
\subsubsection{Folding table and non-uniqueness.}\label{sec:table}
%--------------------------------------------------------------------------------
The set-up of blueprinted director field $\vm$, creases $\Gamma$ and subdomains of $\Omega=[0,1]\times[0,2]$ are displayed in Fig.\ref{fig:pyramid_comb} (left). We choose parameters
\[
s=0.1, ~ s_0=1; ~ h=1/64, ~ \tau=0.5, ~ \tol_1=10^{-6}, ~ \tol_2=10^{-10}.
\]
\vskip-0.3cm
\begin{figure}[htbp]
\includegraphics[width=12.cm]{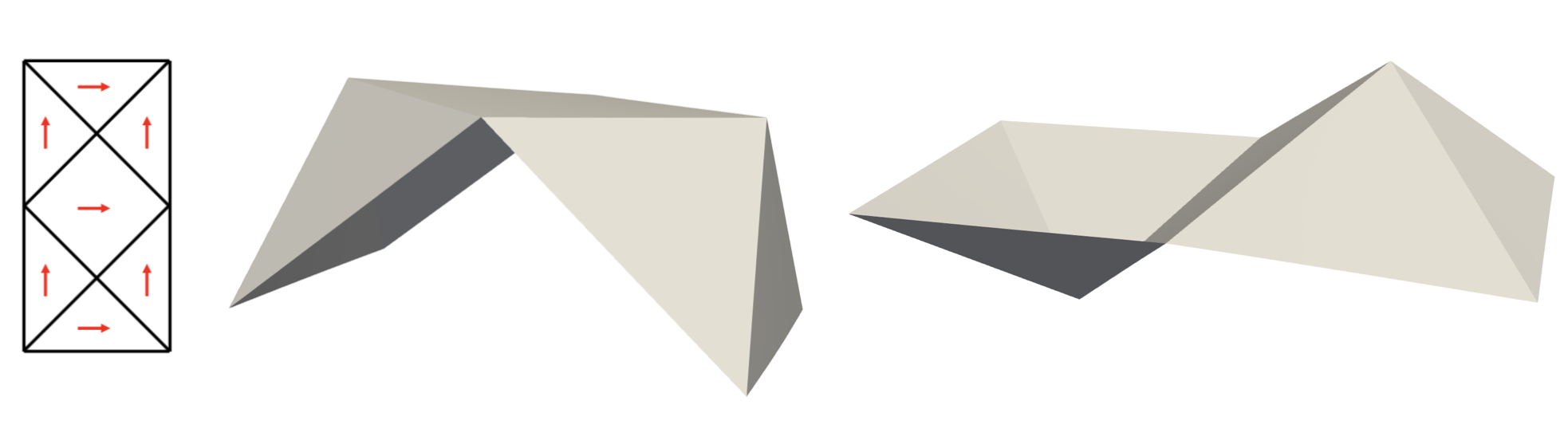}
\vskip-0.3cm
\caption{{\it Folding table}: Setting for origami in Subsection \ref{sec:table} (left). Final configurations for \textit{Case 1} (middle) and \textit{Case 2} (right) llustrate non-uniqueness of minimizers.}
\label{fig:pyramid_comb}
\end{figure}

\noindent
\textit{Case 1.} We use the initialization $\vy_h^0=I_h\vy^0$ with
\begin{equation}\label{eq:initialization_pyramid_comb}
\vy^0(x_1,x_2)~=~\big(x_1,x_2,0.8 \, x_1(1-x_1)x_2(2-x_2)\big)
\end{equation} 
and obtain $E_h[\vy_h^\infty]=8.27 \times 10^{-6}$ and $e_h^1[\vy^{\infty}_h]=4.5\times10^{-3}$.
  
\medskip\noindent
\textit{Case 2.} We use the initialization $\vy_h^0=I_h\vy^0$ with
\begin{equation}\label{eq:initialization_pyramid_1}
\vy^0(x_1,x_2)~=~\big(x_1,x_2,0.8 \, x_1(1-x_1)x_2(1-x_2)\big),
\end{equation}
and obtain $E_h[\vy_h^\infty]=5.97 \times 10^{-6}$ and $e_h^1[\vy^{\infty}_h]=3.7\times10^{-3}$.

Results for both cases are presented in Fig.\ref{fig:pyramid_comb}. We get final configurations consistent with the predicted and experimental shapes in \cite[Figure 5.2]{plucinsky2017deformations}. The two distinct final states correspond to different initial configurations. However, final energies $E_h[\vy_h^\infty]$ and metric deviations $e_h^1[\vy^{\infty}_h]$ are rather small, thereby indicating that both configurations are global minimizers with zero energy. Consequently, this provides an example of non-unique global minimizers due to the non-convex nature of the discrete model \eqref{eq:minimization-pb-discrete}. We refer to \cite{bouck2022NA} for further analytical and computational discussions. 

We emphasize that the gradient flow within Algorithm \ref{algo:GF_euler} may be viewed as relaxation dynamics: different initial states may evolve (or relax) into distinct minimizers. This process is natural in optimization, differential geometry and physics. However, our computational method and model are robust with respect to initialization in the following sense: minor adjustments to the amplitude of the perturbed initial states, such as modifying slightly the coefficient $0.8$ in the third component of \eqref{eq:initialization_pyramid_comb} or \eqref{eq:initialization_pyramid_1}, do not alter the final configuration.

Although a flat surface would be the most natural and straightforward initial state to implement, previous research on gradient flow for surface deformation problems \cite{bonito2022ldg} demonstrates that a planar initialization always results in a planar final configuration. To overcome this inherent limitation of general optimization approaches, we opt to perturb the flat plane $(x_1,x_2,0)$ slightly, as exemplified by \eqref{eq:initialization_pyramid_comb} and \eqref{eq:initialization_pyramid_1}, in the remaining part of this section. This enables us to generate \lucas{various} non-trivial configurations while still maintaining a natural physical meaning and a reasonable level of generality for initializations.
%

%--------------------------------------------------------------------------------------------
\subsubsection{Folding cube.}\label{sec:cube}
%--------------------------------------------------------------------------------------------
We now consider the design in \cite{plucinsky2016programming} whose folded shape is an origami cube. The set-up is given in Fig.\ref{fig:origami_cube} (left), where the domain $\Omega$ is a rhombus with vertices $(0,1),(0,2),(\sqrt{3},0),(\sqrt{3},1)$. We take a graded mesh such that $h=1/128$ near the creases and $h=1/32$ everywhere else. We choose the parameters
\[
s=-1/3, ~ s_0=1;  ~ \tau=0.1, ~ \tol_1=10^{-8}, ~ \tol_2=10^{-10},
\]
and use the initialization $\vy_h^0=I_h\vy^0$ with 
\begin{equation}\label{eq:initialization_cube}
\vy^0(x_1,x_2)~=~\Big(x_1,x_2,0.8x_1(x_1-\sqrt{3})\big(x_2+\frac{\sqrt{3}}{3}x_1-1\big)\big(x_2+\frac{\sqrt{3}}{3}x_1-2\big)\Big).
\end{equation} 
The evolution of our nonlinear gradient flow, regarded as relaxation dynamics, is displayed in Fig.~\ref{fig:origami_cube} (right). The final cube equilibrium configuration satisfies
\[
E_h[\vy_h^\infty]=7.34 \times 10^{-8}, \qquad e_h^1[\vy^{\infty}_h]=3.6\times10^{-4},
\]
and reveals that the cube is a discrete minimizer of \eqref{eq:minimization-pb-discrete}.
This also demonstrates the success of Algorithm~\ref{algo:GF_euler} in dealing with very large deformations accurately, as required to reach the cube configuration starting from an almost flat one.

\begin{figure}[htbp]  
\includegraphics[width=4.cm]{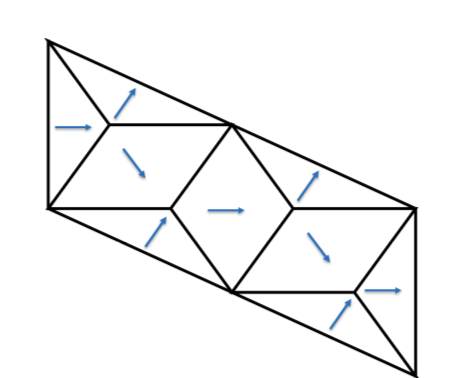}
\includegraphics[width=8.5cm]{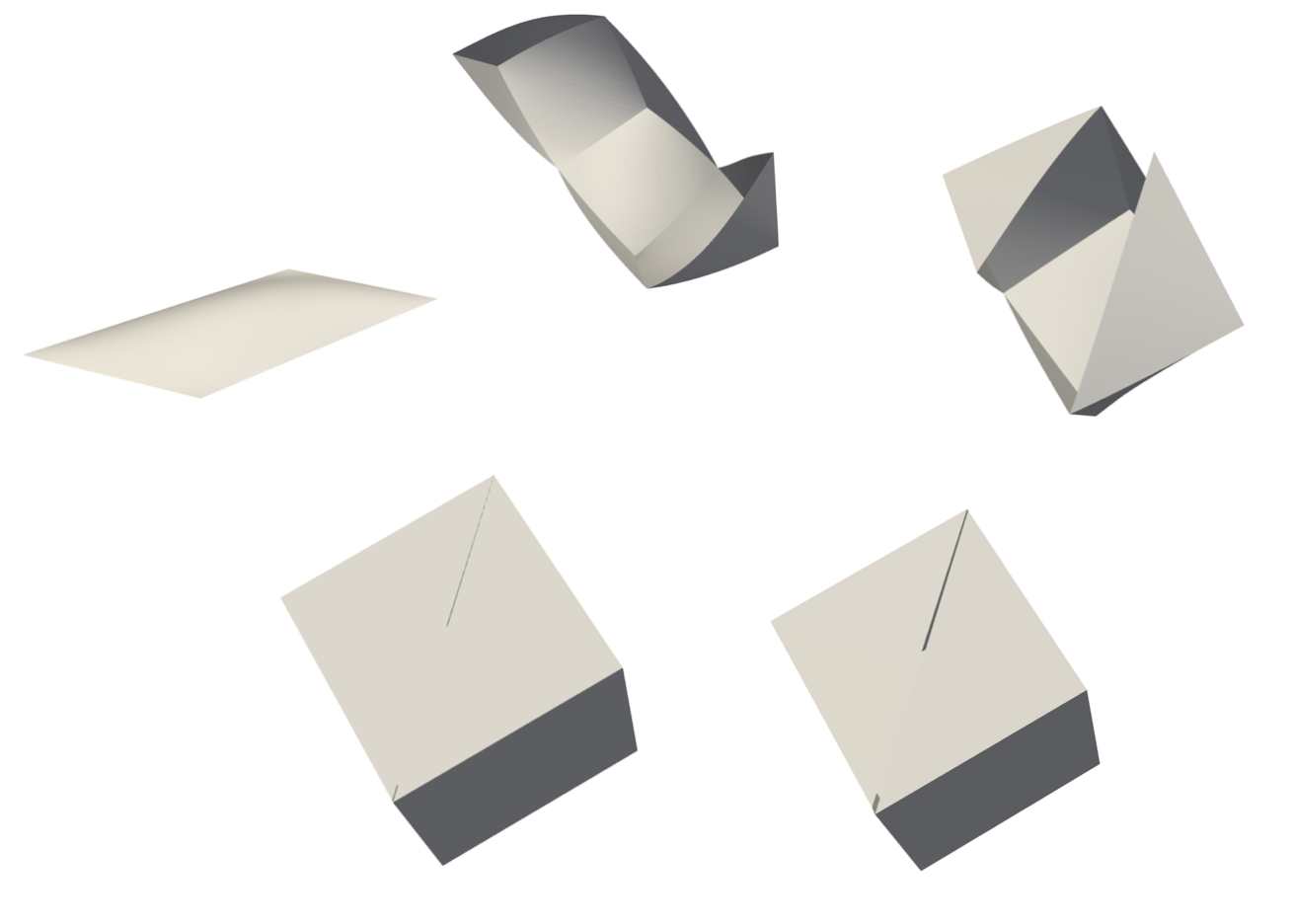}
\caption{{\it Folding cube}:  Rhombus $\Omega$, creases $\Gamma$ and director field $\vm_{\perp}$ (left). Gradient flow iterates $\vy_h^0,\vy_h^{110},\vy_h^{310},\vy_h^{1010}$ and final configuration $\vy_h^{1287}$ displayed clockwise (right).}
\label{fig:origami_cube}
\end{figure}
%

%-------------------------------------------------------------------------------------
\subsection{New shapes: Incompatible nonisometric origami}\label{sec:incompatible-s}
%-------------------------------------------------------------------------------------
In this section, we allow the physical quantities $\vm,s,s_0$ to violate the compatibility condition in Section \ref{sec:origami}, namely to be discontinuous across creases $\Gamma$. 
This entails a discontinuity of $g$ across $\Gamma$ and requires the material to sustain incompatible stretching on both sides of $\Gamma$. The presence of this discontinuity introduces inherent challenges, rendering manual construction of solutions through direct analysis of the metric $g$ exceedingly difficult, if not impossible. Consequently, a computational investigation assumes greater significance in assessing this problem.

Exploring incompatible origami seems to be a novel endeavor in the field of LCNs. Our computational investigation may provide valuable insights for future studies of such structures, including modeling, analysis, and laboratory experiments, along with their potential applications. As far as we know, discontinuous metrics are presently being studied within the framework of fully nonlinear non-Euclidean elasticity by Padilla et al. \cite{padilla2023preparation}. At the same time, motivated by our simulations below, laboratory experiments are being conducted on such structures by Bauman et al. \cite{bauman2023private}.
 
%-------------------------------------------------------------------------------------
\subsubsection{Lifted square incompatible origami.}\label{sec:lifted-sq}
%-------------------------------------------------------------------------------------
Let $\Omega:=[0,1]^2$ be the unit square and the creases and subdomains be as depicted in Fig.~\ref{fig:set-up-lifted-sq} (upper left). The latter are concentric squares with vertices connected by folding lines. We take $s=s_0=1$ in the inner square (ideally no deformation) and $s=0.1,s_0=1$ in the annulus between the two squares so that $\lambda<1$ in this region. This implies shrinking along the direction of the director field $\vm$, hence parallel to the sides, and stretching in the orthogonal direction $\vm_\perp$.
In the inner region, the preferred length of a side, as determined by $g=\Id_2$, is $1$. However, in the outer region, $g$ prescribes a preferred length of approximately $0.82$, because each side shrinks along $\vm$ by a ratio $\lambda=(\frac{s+1}{s_0+1})^{1/3}\approx0.82$. This discrepancy is visually illustrated in Fig.~\ref{fig:set-up-lifted-sq} (upper right). It is important to emphasize that the model being considered does not account for material fracture. Consequently, such a mismatch in preferred lengths is expected to result in the buckling of the inner square.

\begin{figure}[h!]
%\begin{tikzpicture}[scale=1.5]
%    \draw (1,-1) -- (1,1) -- (-1,1) -- (-1,-1) -- (1,-1);
%    \draw (.5,-.5) -- (.5,.5) -- (-.5,.5) -- (-.5,-.5) -- (.5,-.5);
%    \draw (-1,-1) -- (-.5,-.5);
%    \draw (-1,1) -- (-.5,.5);
%     \draw (1,1) -- (.5,.5);
%     \draw (1,-1) -- (.5,-.5);
%    \draw [-to](.75,-.25) -- (.75,.25);
%    \draw [-to](.25,.75) -- (-.25,.75);
%    \draw [-to](-.75,.25) -- (-.75,-.25);
%    \draw [-to](-.25,-.75) -- (.25,-.75);
%    \node at (0,0) {$\lambda=1$};
%\end{tikzpicture}
\includegraphics[width=6.5cm]{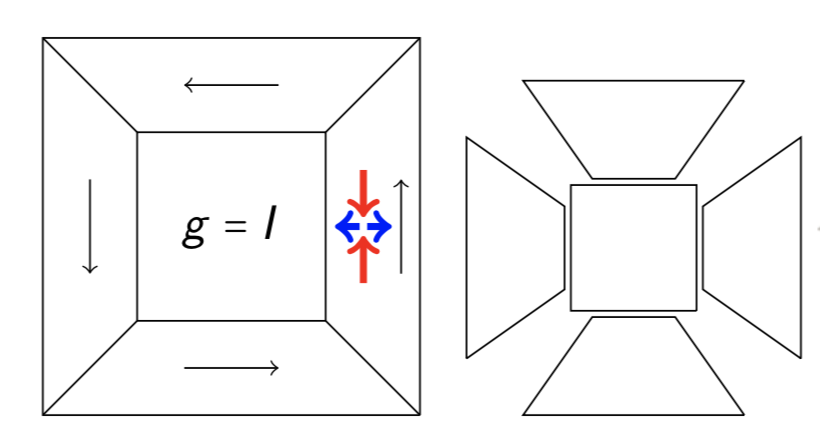}
\includegraphics[width=9.cm]{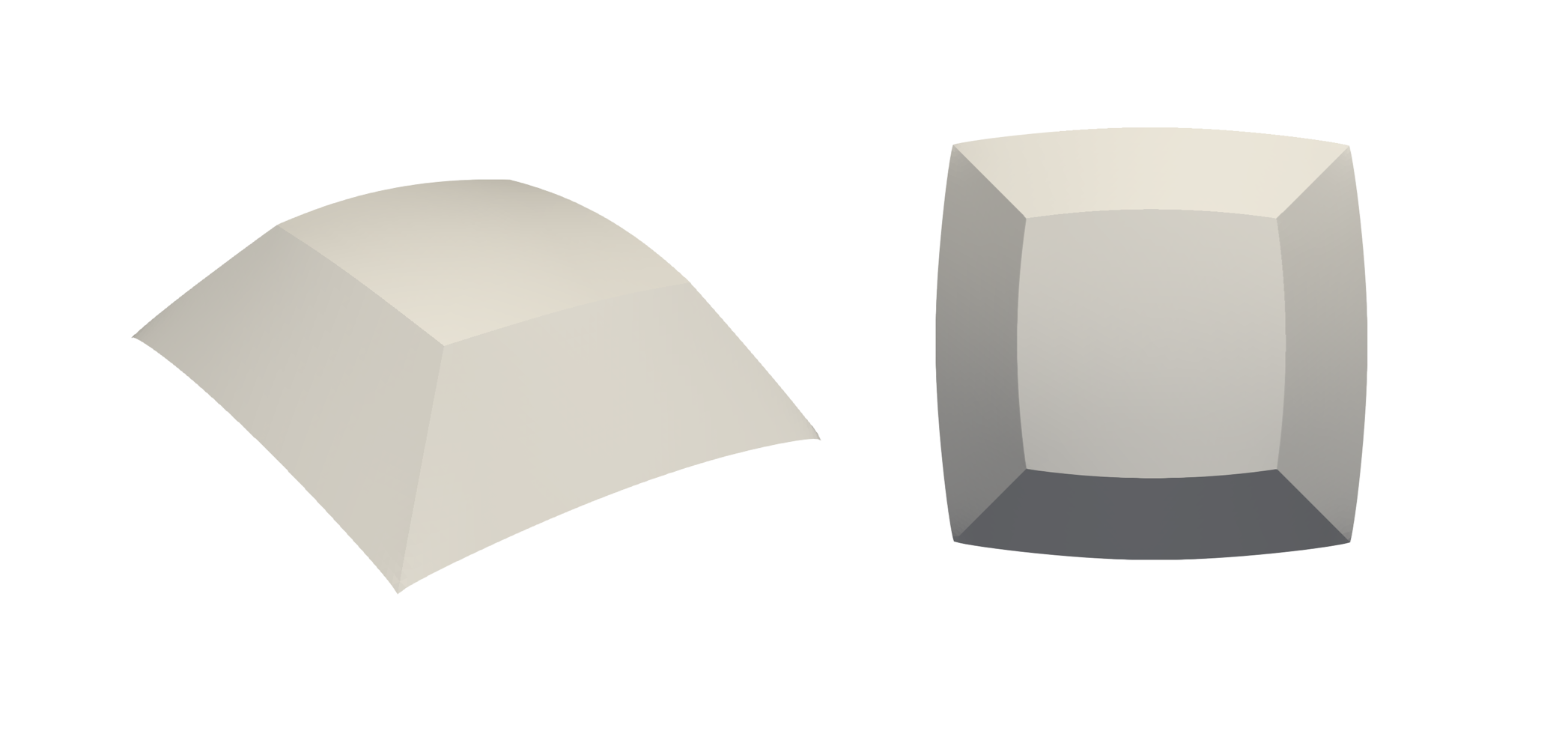}
\vskip-0.7cm
\caption{{\it Lifted square incompatible origami}: Upper left: Lines indicate creases and black arrows indicate the blueprinted director field $\vm$ in regions where $\lambda<1$ (red for shrinking direction and blue for stretching). The inner square has $\lambda=1$ (no internal deformation), i.e. $g=\Id_2$. 
Upper right: A schematic illustration of incompatible stretching caused by discontinuous $g$.
Bottom: Two views of equilibrium configuration show that buckling takes place to accommodate the lack of data compatibility.}
\label{fig:set-up-lifted-sq}
\end{figure}

We use the initialization \eqref{eq:initialization_pyramid_1} and choose the parameters to be
\[
h=1/64, ~ \tau=0.1, ~ \tol_1=10^{-6}, ~ \tol_2=10^{-10}.
\]
We consider the modified regularization \eqref{eq:reg+folding} with $c_r=0$ \lucas{along $\Gamma$, the folding line,} and $c_r=100$ in the rest of the domain. We plot two views of the final configuration in Fig.~\ref{fig:set-up-lifted-sq} (bottom). Despite the inner region has no internally-induced deformation, it shrinks and lifts up out of plane to accommodate  the change of the outer region.

To explore the asymptotic behavior of $\vy_h^\infty$ as $h\to0$ we run a series of experiments reported in Fig.~\ref{fig:conv_h}. We see that the energy $E_h[\vy_h^\infty]$ is $\calO(10^{-2})$ and decreases but not quadratically; in fact, it seems that it stabilizes to a positive value. Recall that the quadratic scaling $E_h[\vy_h]\le \Lambda h^2$ of \eqref{eq:energy-scaling} in Theorem \ref{T:convergence} (convergence of discrete minimizers) leads to strong convergence in $H^1(\Omega;\R^3)$ to a piecewise $H^2$-limit $\vy^*$ such that $E_{str,\Gamma}[\vy^*]=0$ and $\I[\vy^*]=g$. This suggests that a discontinuous $g$ may not admit an isometric immersion, at least not with the regularity stated in Theorem \ref{T:convergence}. This in turn contrasts with the compatible origami shapes in Sections \ref{sec:table} and \ref{sec:cube} for which the final energies are $\calO(10^{-6})$ and $\calO(10^{-8})$, thus making it \shuo{plausible} that \lucas{the asymptotic limit $\vy^*$ satisfies} $E_{str,\Gamma}[\vy^*]=0$ and $\I[\vy^*]=g$ in view of Corollary \ref{cor:minimizers} (immersions of $g$ are minimizers with vanishing energy). Moreover, such piecewise affine exact solutions obviously satisfy the piecewise $H^2$ regularity. When comparing the scenarios of compatible and incompatible origami, it becomes highly likely that the non-existence of isometric immersion for the latter can be attributed to the mismatch of expected side lengths on each side of the crease.
\begin{figure}[htbp]
\begin{minipage}{.4\textwidth}
\begin{tabular}{|c|c|c|}
\hline
h & $ E_h[\vy^{\infty}_h] $ & $e_h^1[\vy^{\infty}_h]$  \\
\hline\hline
 1/32  & 0.0277   & 0.177   \\ \hline
 1/64  & 0.0127   & 0.103   \\ \hline
 1/128 & 0.00751  & 0.0754  \\ \hline
 1/256 & 0.00525  & 0.0626  \\ \hline
\end{tabular}
\end{minipage}
\begin{minipage}{.5\textwidth}
\begin{center}
 \includegraphics[width=7.5cm,height=2.0cm]{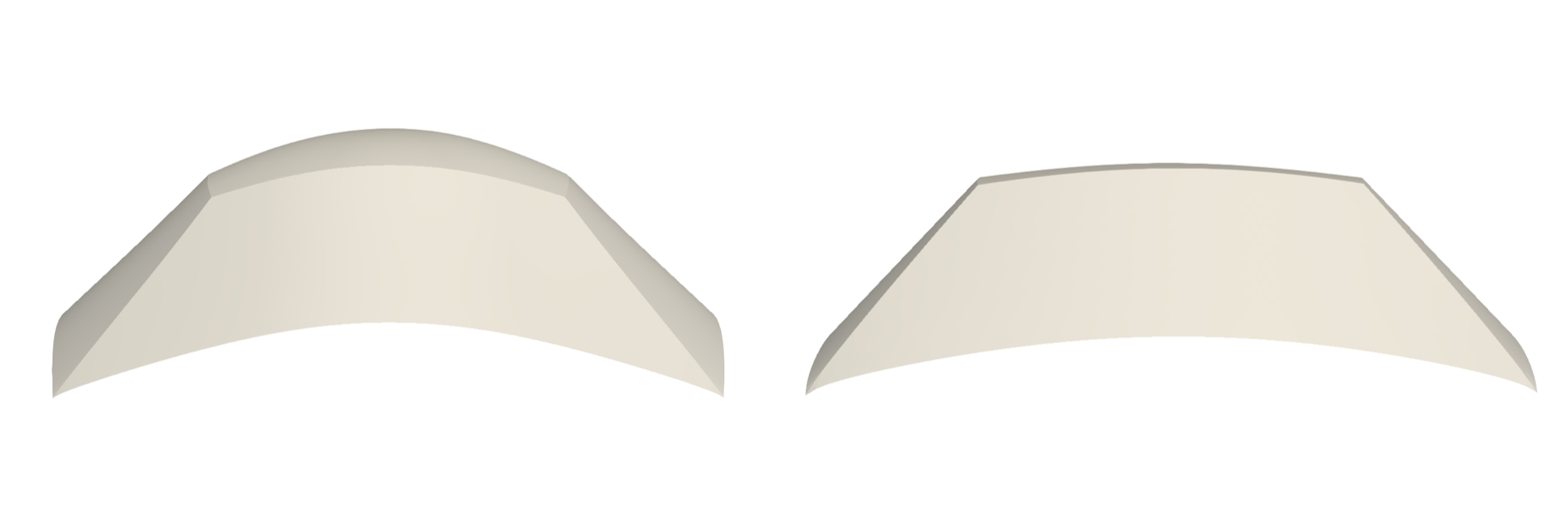}   
\end{center}
\end{minipage}
\hfill
\vspace{0.3cm}
\caption{{\it Lifted square incompatible origami}: The table shows a monotone decrease of discrete energy $E_h[\vy^{\infty}_h]$ and metric defect $e_h^1[\vy^{\infty}_h]$ in terms of $h$, which stabilizes to a positive value. Side views of the deformations for $h=1/128$ (middle) and $h=1/64$ (right). The buckling is more pronounced for smaller $h$.} \label{fig:conv_h}
\end{figure}

Fig.~\ref{fig:conv_h} also reveals that buckling becomes more pronounced as the meshsize $h$ decreases. One possible explanation for this phenomenon is that the effect of regularization reduces as $h$ becomes smaller. We can interpret $R_h[\vy_h]$ as a numerical mechanism that selects equilibrium configurations in the limit as $h$ approaches zero. 

To validate the impact of regularization, we conduct additional experiments with the same parameters, except that the regularization parameter becomes $c_r=1$ or $c_r=0$ away from creases $\Gamma$; see Fig.~\ref{fig:set-up-lifted-sq-2} for results. In contrast to the case $c_r=100$ in $\Omega\setminus\Gamma$, depicted in Fig.~\ref{fig:set-up-lifted-sq}, it is evident that a smaller regularization parameter leads to more pronounced buckling. Furthermore, the use of $R_h$ with $c_r=1$ eliminates the observed wrinkles seen in the case of $c_r=0$. This reinforces the impact of regularization, as discussed in Section \ref{sec:eff-reg}.

\begin{figure}[h!]
\includegraphics[width=10.cm]{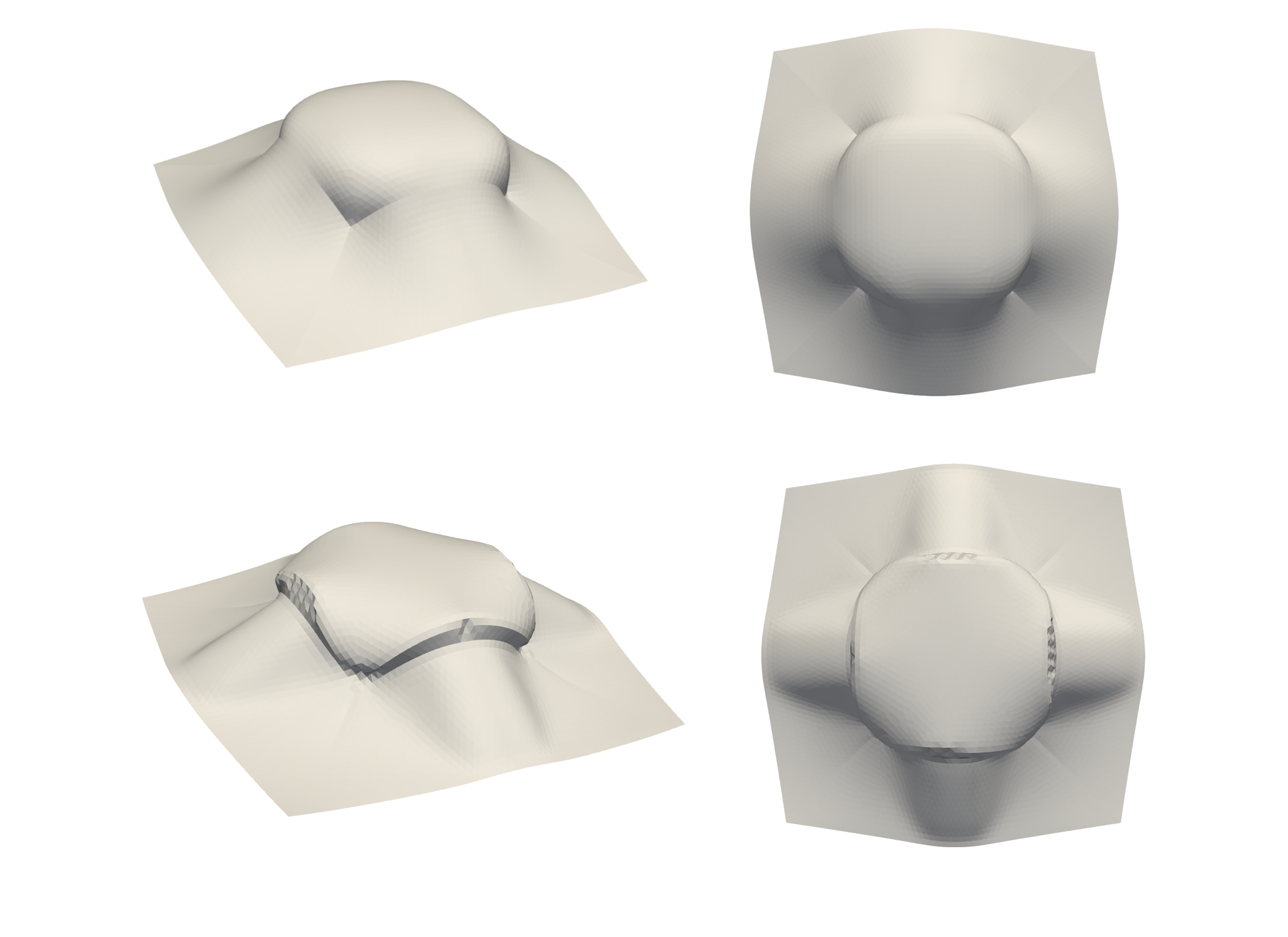}
\vskip-0.5cm  
\caption{{\it Incompatible square origami}: Two views of the final configurations with regularization parameter $c_r=1$ (top) and $c_r=0$ (bottom) away from the creases. Wrinkling occurs for $c_r=0$.}
\label{fig:set-up-lifted-sq-2}
\end{figure}

We now discuss the behavior of $e^1_h[\vy_h^{\infty}]$ and $E_h[\vy_h^{\infty}]$ for various $c_r$. We obtain  
\[
c_r=100, ~ 1, ~0 \Rightarrow e^1_h[\vy_h^{\infty}]\approx0.1,~ 0.035,~ 0.01;~ E_h[\vy_h^{\infty}]\approx0.012,~ 0.0026,~ 8.9\times10^{-4}. 
\]
We note that as $c_r$ increases, the corresponding $e^1_h[\vy_h^\infty]$ and $E_h[\vy_h^\infty]$ increase as well. The minimizing configuration cannot keep the stretching energy low without \lucas{increasing the bending energy}. We may conclude that the regularization \eqref{eq:E-reg} serves as a mechanism to select minimizers of stretching energy for {compatible} origami, while a competition between stretching and bending energies determines the final shape for {incompatible} origami. This process as $h\to0$ produces a candidate minimizing sequence for $E_{str}$ where the length scale of oscillations \lucas{is determined by a trade-off between minimizing $E_{str}$ and minimizing the bending energy.}% would be determined by the relative strength of the bending energy through the constant $c_r$. 
%

%-------------------------------------------------------------------------------------
\subsubsection{Lifted square origami without creases.}\label{sec:lifted-sq-no-creases}
%-------------------------------------------------------------------------------------

We proceed to investigate again \shuo{the lifted square example, considering the blueprinted director field with $\lambda=1$ in the inner square as} depicted in Fig.~\ref{fig:set-up-lifted-sq} (upper left), but with an empty folding set $\Gamma = \emptyset$. We set $s=s_0=1$ within the inner square and $s=-0.3,s_0=1$ within the square annulus, ensuring that $\lambda<1$ in this exterior region. This computation is motivated by the fact that creasing the material may be impractical in laboratory experiments. Additionally, we plan to investigate the influence of the size of the inner square relative to the outer domain on the resulting deformation.

In Fig.~\ref{fig:lifted-square-exp-compare} (top row), we display the output of Algorithm~\ref{algo:GF_euler} with 
%$c_r=0.2$ and $h=1/64$. 
\shuo{
\[
\textrm{tol}_1=10^{-6}, ~ \textrm{tol}_2=10^{-10}, ~ h=1/64, ~ c_r=0.2.
\]
}
The aspect ratios between side lengths for inner and outer squares are $0.3$, $0.5$, and $0.7$ from left to right on top of Fig.~\ref{fig:lifted-square-exp-compare}, \shuo{whereas the corresponding parameters $\tau$ are $\tau=0.3,0.4,0.5$.} The initialization is chosen to be the same as in Section \ref{sec:lifted-sq}.

In Fig.~\ref{fig:lifted-square-exp-compare}, we also present a comparison between our computational results and ongoing laboratory experiments \cite{bauman2023private}. In the latter, the outer square side length is 16mm, while the inner square exhibits varying side lengths of 4.8mm, 8mm, and 11.2mm from left to right, as depicted in Fig.~\ref{fig:lifted-square-exp-compare} (bottom). The sample used in the experiments has a thickness of 0.45mm. We observe a remarkable agreement between our computational predictions and the experimental outcomes. Both exhibit concave regions, or wrinkling, near the edges of the inner square region. It is worth mentioning that this wrinkling structure occurs at a much larger length scale than those shown in Fig.~\ref{fig:set-up-lifted-sq-2}. This phenomenon can be again attributed to the mismatch of expected length changes across the boundary of the inner square.

\begin{figure}[htbp]
\begin{minipage}{.3\textwidth}
\begin{center}
 \includegraphics[width=\textwidth]{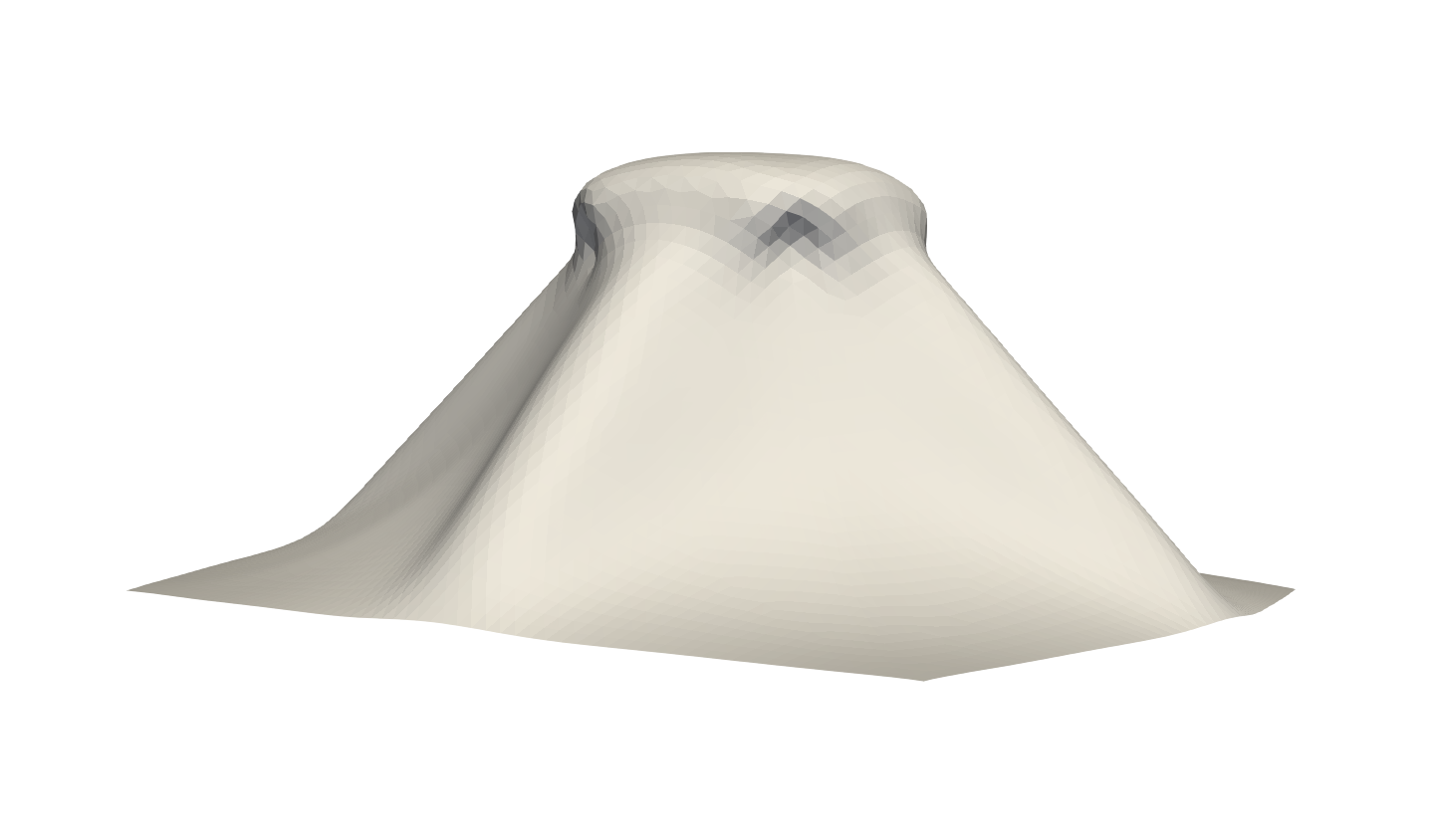}  \\
 \includegraphics[width=\textwidth]{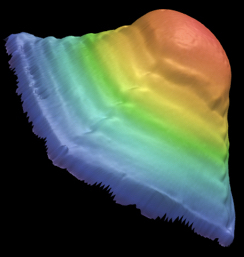}   
\end{center}
\end{minipage}
\begin{minipage}{.3\textwidth}
\begin{center}
 \includegraphics[width=\textwidth]{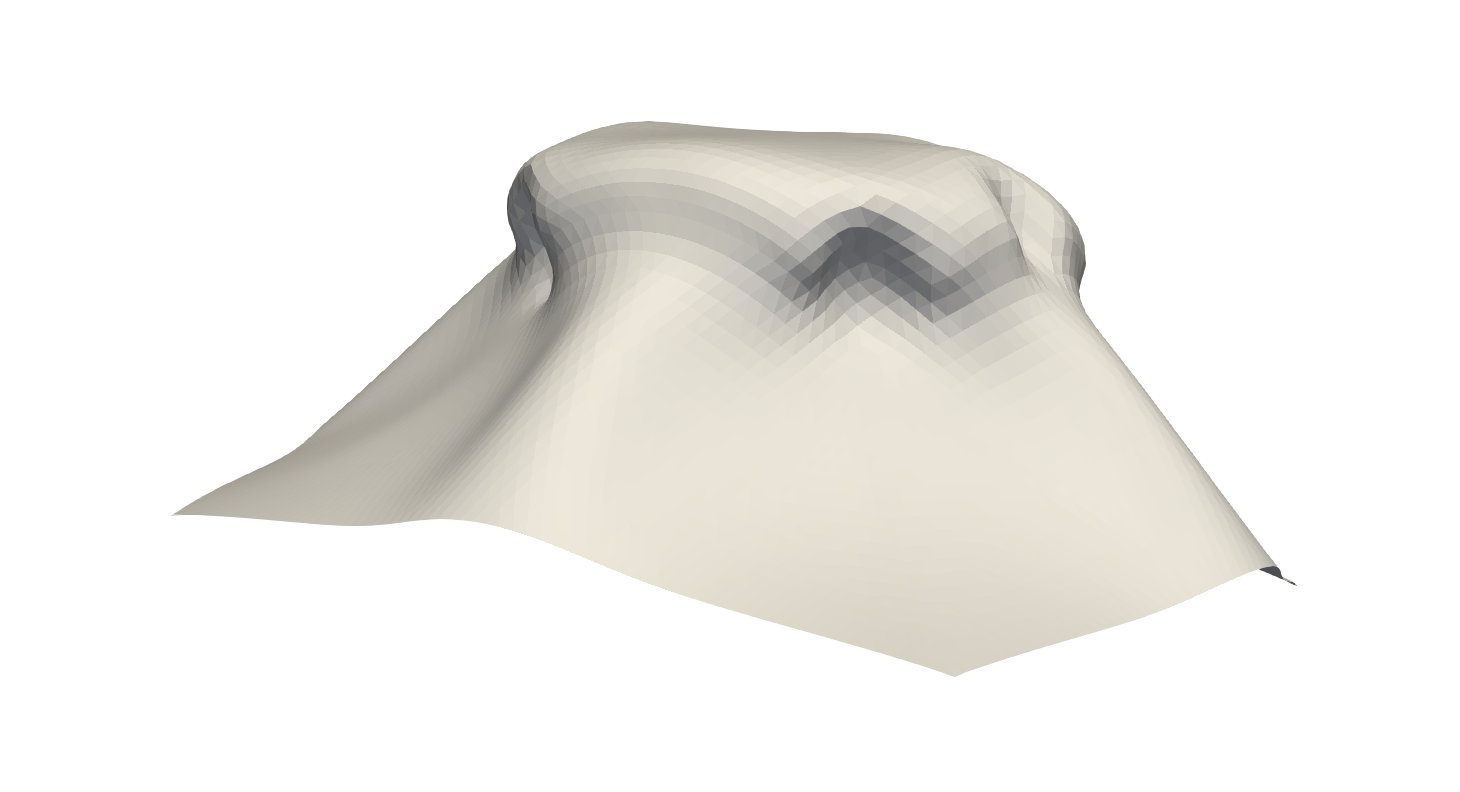}  \\
 \includegraphics[width=\textwidth]{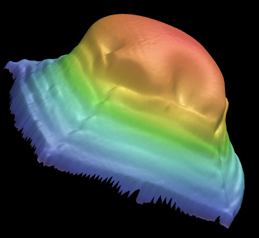}   
\end{center}
\end{minipage}
\begin{minipage}{.3\textwidth}
\begin{center}
 \includegraphics[width=\textwidth]{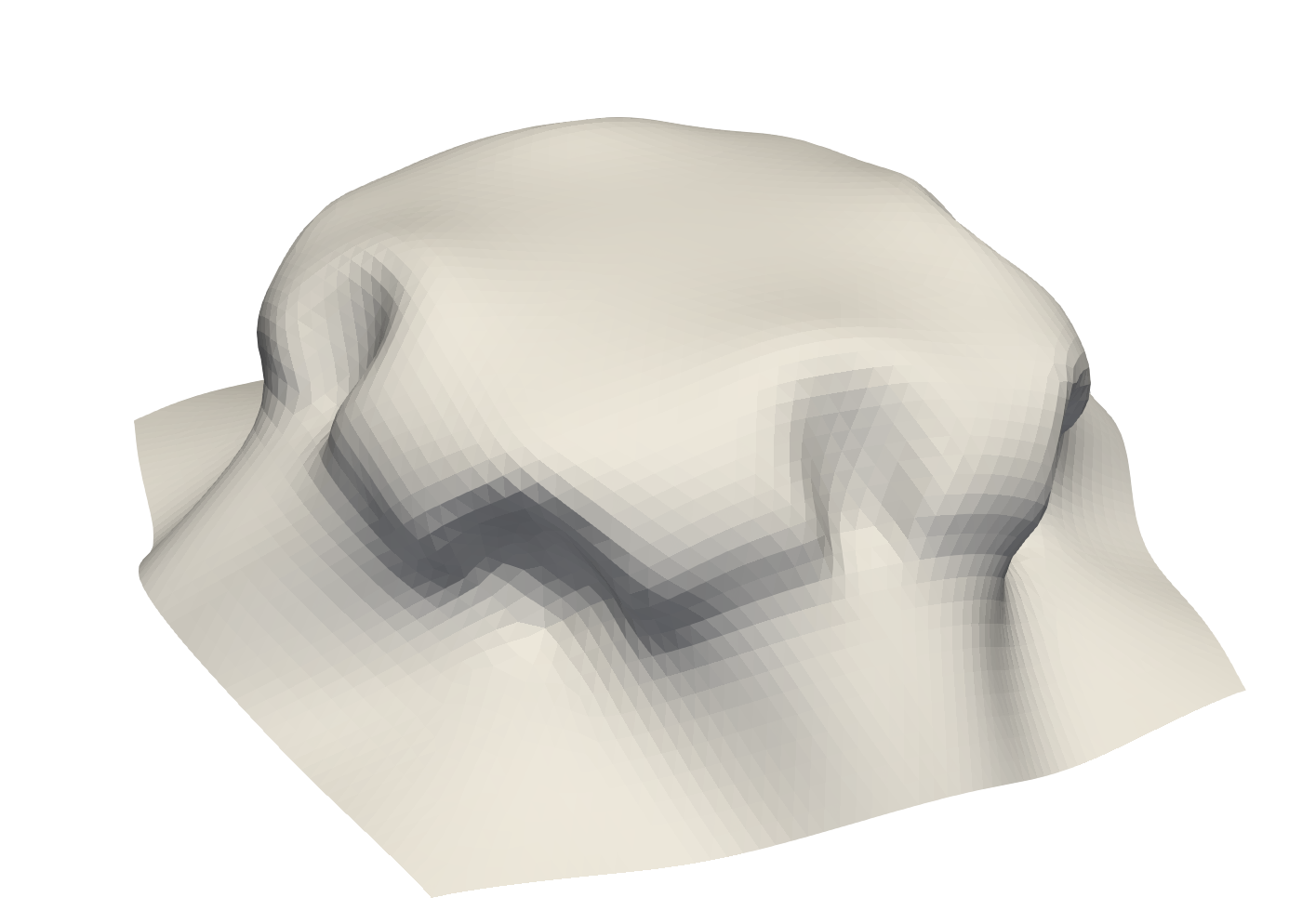}  \\
 \includegraphics[width=\textwidth]{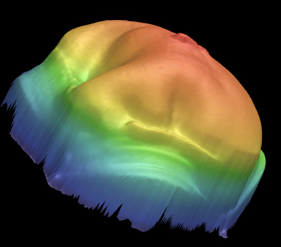}   
\end{center} 
\end{minipage}
\hfill
\vspace{0.3cm}
\caption{{\it Lifted incompatible square origami}: Computed solutions with $c_r=0.2$ (top) and laboratory configurations by Bauman et al. \cite{bauman2023private} (bottom) for lifted squares without creases. The size of the inner square relative to the outer domain varies from left to right. The images in the bottom row are 3-D scans of experimental samples heated to 150\textdegree C generated with an optical profilometer.} \label{fig:lifted-square-exp-compare}
\end{figure}

The practical interest in lifted squares stems from their potential to exert larger lifting forces  than cones \cite{ware2015voxelated}. Preliminary lab experiments conducted by Baumann and White \cite{bauman2023private} \lucas{suggest that the configuration on the left in Fig.~\ref{fig:lifted-square-exp-compare} outperforms
the others.} This seems to be related to the relative size of wrinkles which is more pronounced for the middle and right configurations in Fig.~\ref{fig:lifted-square-exp-compare}.

We explore this observation next, upon considering the current choice of $s$ and $s_0$ which leads to $\lambda\approx0.7$.
The interface length from the outer region undergoes a fixed contraction dictated by $\lambda$ whereas the inner length remains fixed. Letting $\ell$ denote the side length of the inner square, this implies that the total length mismatch across the interface is given by $4\ell(1-\lambda)$. Consequently, if the inner side length $\ell$ is larger, the length of the mismatch that needs to be accommodated also increases and so does the size of wrinkles. This provides a plausible explanation for the observed phenomenon.

We further compare the outcomes presented in Figs.~\ref{fig:lifted-square-exp-compare} and \ref{fig:set-up-lifted-sq-2}, where the folding set is defined as $\Gamma = \emptyset$ and $\Gamma \neq \emptyset$, respectively. Both scenarios involve a disparity in preferred arc lengths along the inner square boundary, yet the resulting configurations adapt to this mismatch in different ways. Notably, in the case where $\Gamma \neq \emptyset$, the inner square bulges upward as the primary behavior. Conversely, in the case where $\Gamma = \emptyset$, the material begins to form large wrinkles, or concave regions, near the boundary of the inner square. This contrasting behavior can be attributed to the penalty imposed on $\jump{\nabla\vy_h}$ along the inner square when $\Gamma = \emptyset$. The presence of a significant jump in $\jump{\nabla\vy_h}$ across the metric discontinuity becomes energetically costly, and so does a pronounced buckling behavior. Consequently, the resulting configuration  in Fig.~\ref{fig:lifted-square-exp-compare} mitigates this energy increase by introducing wrinkling across the interface to accommodate the length mismatch. It is worth noting that wrinkling is observed in analogous situations, such as the uniform compression of thin sheets \cite{conti2005self}.

We also include in Fig.~\ref{fig:lifted-square-plot-data} plots of the metric defect $|\I(\vy^{\infty}_h)-g|(\vx)$ and the $l_1$ norm of the stress tensor on $\Omega$. The former is an elementwise constant function, while the latter is the sum of the absolute values of each component. The stress tensor is obtained by calculating variational derivatives of the energy density function, but we omit the detailed calculations here. The plots in Fig.~\ref{fig:lifted-square-plot-data} correspond to the solution of the lifted square with larger inner square, specifically the third image in the top row of Fig.~\ref{fig:lifted-square-exp-compare}. It is evident that both the metric defect and the stress are concentrated along the interfaces, which are the locations where the metric exhibits a discontinuity. This phenomenon can be plausibly attributed to the inconsistent stretching experienced on both sides of the interfaces in incompatible origami configurations.

\begin{figure}[htbp]
\includegraphics[width=.7\textwidth]{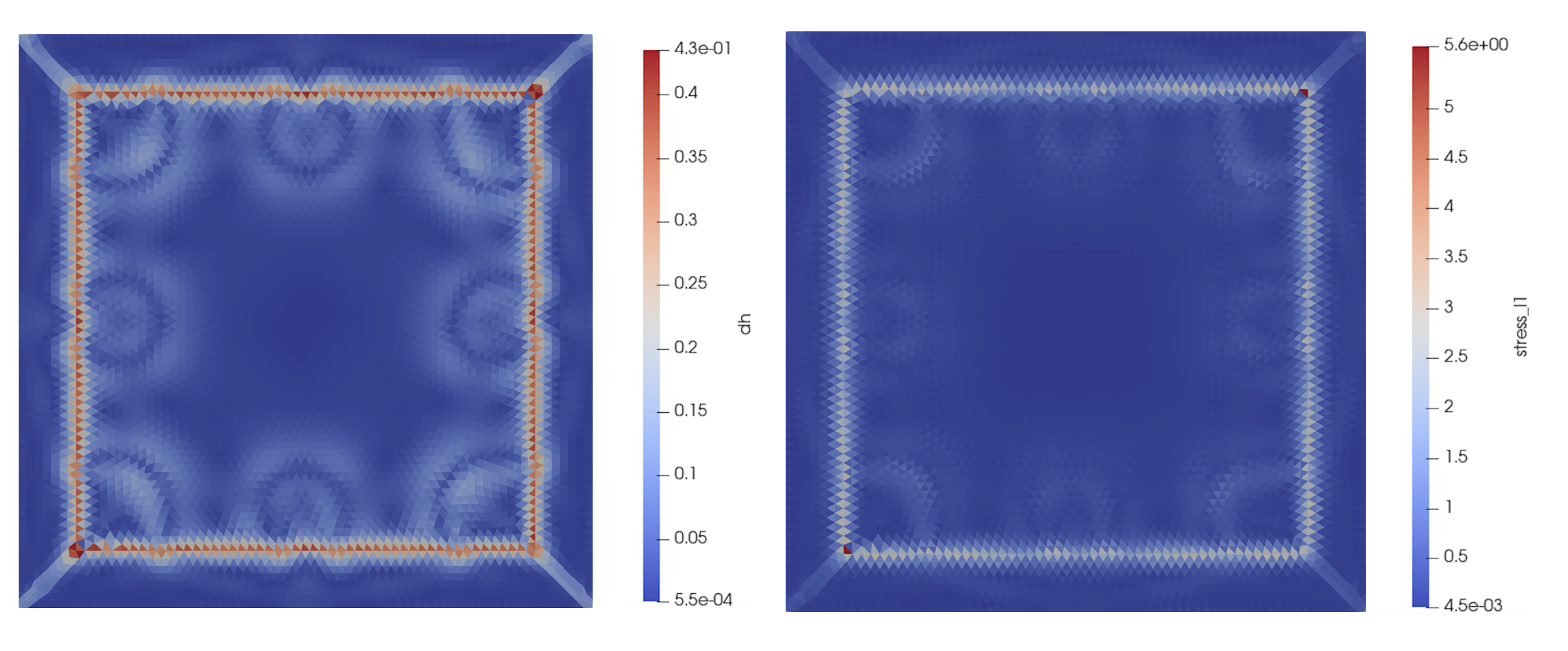}
\caption{{\it Metric defect and stress.} Left: plot of metric deviations on $\Omega$. Right: plot of the $l_1$ norm of the stress tensor on $\Omega$. In the visualization, the color dark blue represents values close to zero, while the color red indicates larger positive values\shuo{; see legend.}} 
\label{fig:lifted-square-plot-data}
\end{figure}
%

%------------------------------------------------------------------------------------
\subsubsection{Generalizations: lifted M-origami}\label{sec:lifted-M}
%------------------------------------------------------------------------------------
We explain now how to exploit the idea in Subsection \ref{sec:lifted-sq} as a building block to design lifted configurations of any polygonal shapes. In this case, our strategy entails penalizing bending by setting a high value of $c_r=100$, while also permitting the presence of creases in the material.
%Here, we adopt the strategy of having a large influence of the bending energy $c_r=100$, and allow for creases in the material.
In fact, for any polygonal subdomain $P\subset\Omega:=[0,1]^2$ with $\textrm{dist}(P,\partial \Omega)>0$, we can always construct a dilation $P'$ of $P$ so that it is ``concentric'' with $P$ and $P\subset P'$ with $\textrm{dist}(P',\partial \Omega)>0$. Then we further connect corresponding vertices of $P$ and $P'$ with folding lines, and also let all the sides of $\partial P$ and $\partial P'$ be creases. We finally take $\vm$ parallel to the sides of $\partial P$ and $\lambda<1$   in $P'\setminus P$, while $\lambda=1$ in $P$ and $\Omega\setminus P'$. The discontinuity of $\lambda$ across creases implies again $E_{str}[\vy]>0$ for all $\vy\in H^1(\Omega;\R^3)$.

We apply this procedure to an M-shaped subdomain, as shown in \lucas{Fig.~\ref{fig:lifted-M} (left)}. We choose all the parameters, \lucas{unless otherwise specified}, to be the same as in Subsection \ref{sec:lifted-sq}. In particular $s=-0.5$ inside the M-annulus region while $s=1$ in the rest of domain. We use a graded mesh of size $h=1/256$ near the creases and otherwise $h=1/32$. The initialization is also taken as \eqref{eq:initialization_pyramid_1}.

We display the computed solution in Fig.~\ref{fig:lifted-M}, which is the desired lifted M-shape. We stress that the background and solid M are not completely flat due to the same buckling effect already discussed in Section \ref{sec:lifted-sq}. However, this effect is not so pronounced because the shrinking layer is thin relative to the rest of the M and background. We emphasize that the current procedure is different from the construction of lifted surfaces in Section \ref{sec:lifted-surface}. The latter requires $|\nabla\phi|=\sqrt{\lambda^3-1}$ a.e. in $\Omega$, which makes it harder to implement; recall the discussion after \eqref{eq:metric-lifted-2}.  

\begin{figure}[htbp]
\begin{minipage}{.35\textwidth}

\begin{tikzpicture}[scale=4.5]

	%M
    \draw (.837,.28) -- (.755,.77) -- (.655,.77) -- (.5,.46) -- (.345,.77) -- (.245,.77)--(.163,.28)--(.263,.28)--(.3225,.635) -- (.5,.28) -- (.6775,.635) -- (.737,.28) -- (.837,.28);
    % outer edge of origami layer
    \draw (.9,.2)-- (.8,.8) -- (.6,.8)-- (.5,.6)--(.4, .8)--(.2, .8)--(.1,.2)--(.3,.2) -- (.35,.5)--(.5,.2)-- (.65,0.5)--(.7,.2)--(.9, .2);
%    

%	\draw (0,0) -- (1,0) -- (1,1) -- (0,1) -- (0,0);
    % creases in origami layer
    \draw (.837,.28) --(.9,.2);
    \draw (.755,.77)--(.8,.8);
    \draw (.655,.77) -- (.6,.8);
    \draw (.5,.46)--(.5,.6);
    \draw  (.345,.77) -- (.4, .8);
    \draw (.245,.77) -- (.2, .8);
    \draw (.163,.28) -- (.1,.2);
    \draw (.263,.28) -- (.3,.2);
    \draw (.3225,.635) -- (.35,.5);
    \draw (.5,.28) -- (.5,.2);
    \draw (.6775,.635)  -- (.65,0.5);
    \draw (.737,.28) -- (.7,.2);
%    \draw (.7,-.4) -- (.8,-.5);
    
    % labels outside and inside as lambda = 1
%    \node at (.69,.69) {\tiny Inactive};
%    \node at (.9,.9) {\tiny Inactive};
%    \node at (1,.75) {$\lambda=1$};
%    

%     -- (.755,.77) -- (.655,.77) -- (.5,.46) -- (.345,.77) -- (.245,.77)--(.163,.28)--(.263,.28)--(.3225,.635) -- (.5,.28) -- (.6775,.635) -- (.737,.28) -- (.837,.28);

    \draw [-to] (0.83266666666,0.44) -- (0.80533333333,0.60333333333);  
     \draw [-to] (.755,.785) -- (.655,.785);
     
     \draw [-to] (.585,.7) -- (.515,.56);
     \draw [to-] (.42,.7) -- (.48,.56);
      \draw [-to] (.345,.785) -- (.245,.785);
      
      \draw [-to] (0.195,0.61)--(0.17,0.45); 
      \draw [-to] (.163,.23)--(.263,.23);
      
      \draw [-to] (.29,0.33)--(.3125,0.45);
%      \draw [-to] (.263,.28)--(.3225,.635);
%      \draw [-to] (.3225,.635) -- (.5,.28);
	\draw [to-] (.455,0.33)--(.395,0.45);
      \draw [-to] (.545,0.33)--(.605,0.45);
%      \draw [-to] (.5,.28) -- (.6775,.635);
%      \draw[-to]  (.6775,.635) -- (.737,.28);
	\draw [to-] (0.71,0.33)--(0.6875,0.45);
      \draw[-to]  (.737,.23) -- (.837,.23);

%\node at (-.03,-.03) {(0,0)};
%\node at (1.03,1.03) {(1,1)};
\end{tikzpicture}
\end{minipage}\hfill
\begin{minipage}{.6\textwidth}
  \includegraphics[width=\textwidth]{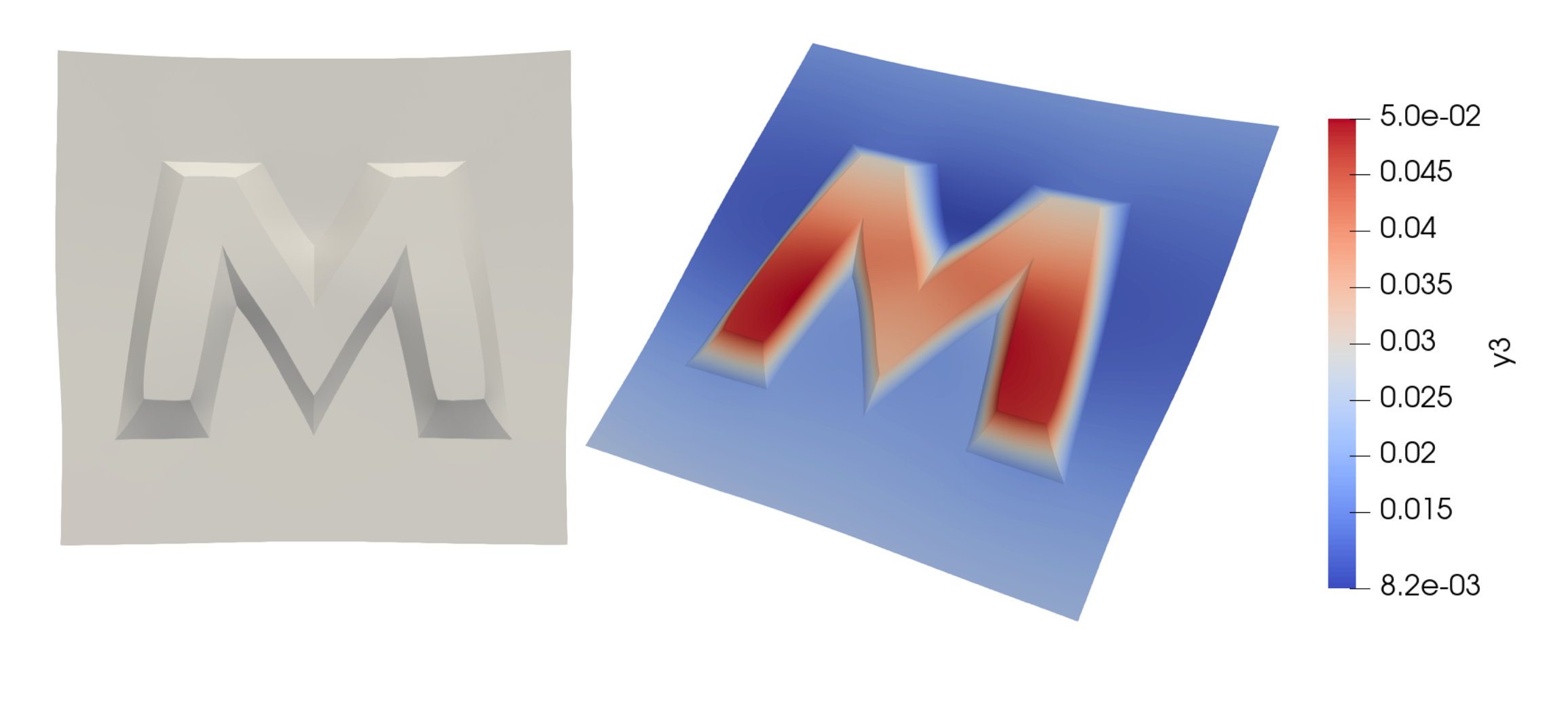}
  \end{minipage}
\vskip-0.7cm  
\caption{{\it Lifted M-origami}: (Left) Lines indicate creases and arrows indicate the blueprinted director field $\vm$ in regions where $\lambda<1$, whereas $\lambda=1$ within and outside the M. (Center and Right) Two views of final equilibrium configuration. Color on the right picture represents the value of $y_3$ and shows that the solid M and background are not completely flat.}
\label{fig:lifted-M}
\end{figure}
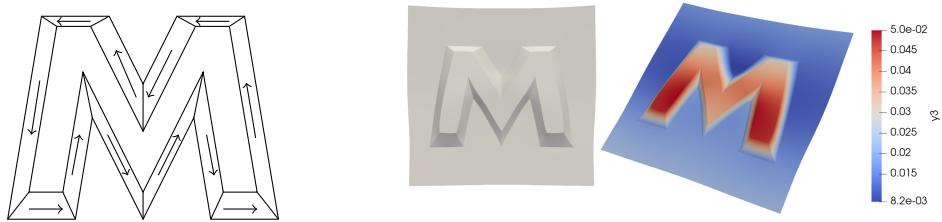

%-----------------------------------------------------------------------------------
\subsubsection{Incompatible origami with cuts}\label{sec:lifted-cuts}
%-----------------------------------------------------------------------------------

In this section, we examine the same configuration of the director field $\vm$ as depicted in Fig.~\ref{fig:set-up-lifted-sq} (upper left). 
%However, inspired by kirigami \cite{sussman2015algorithmic}, we introduce cuts to the boundaries of the inner squares, where the metric experiences discontinuity.
\lucas{We now introduce cuts to the boundaries of the inner squares, where the metric experiences discontinuity. This additional set of numerical experiments is inspired by kirigami \cite{sussman2015algorithmic}, which provide new methods for shape-programming beyond origami.}
 
The domains of the experiments considered in this section are illustrated in Fig.~\ref{fig:cuts-domain}. 
%We note that cuts have been explored quite extensively.
Similar to previous investigations, we utilize an empty folding set $\Gamma = \emptyset$. The \shuo{numerical} parameters chosen for this experiment \shuo{are: 
\[
 h=1/32, ~ \tau=0.15, ~ \tol_1=10^{-6}, ~ \tol_2=10^{-10}, c_r=0.2.
\]
Moreover, $s=s_0=1$ within the inner square, and $s=-0.3,s_0=1$ within the annulus between the two squares.} The ratio between the side length of inner and outer square is $1/2$. The initialization is chosen to be the same as in Section \ref{sec:lifted-sq}.

\begin{figure}[htbp]
  \includegraphics[width=.3\textwidth]{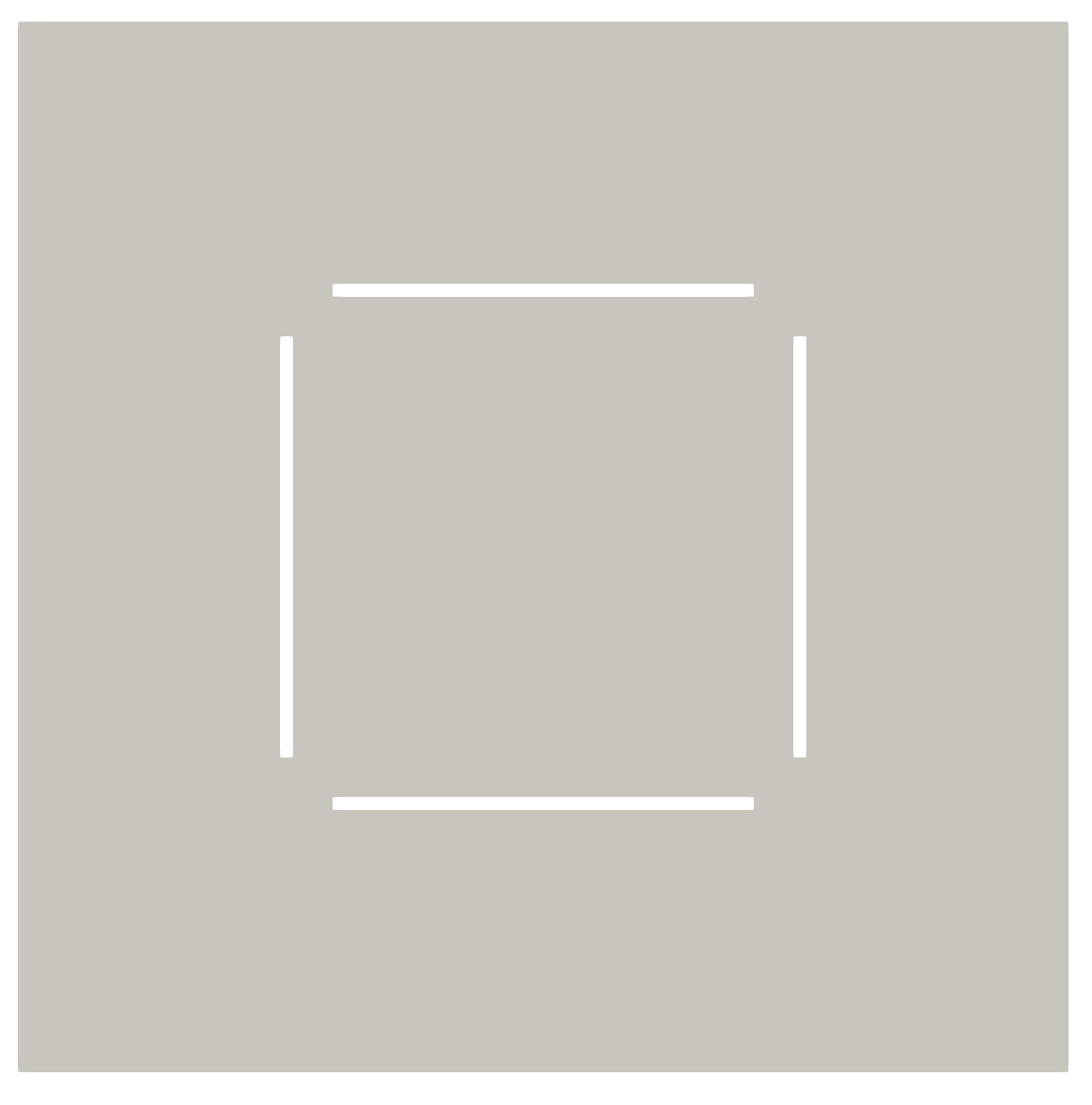} %.35
  \includegraphics[width=.315\textwidth]{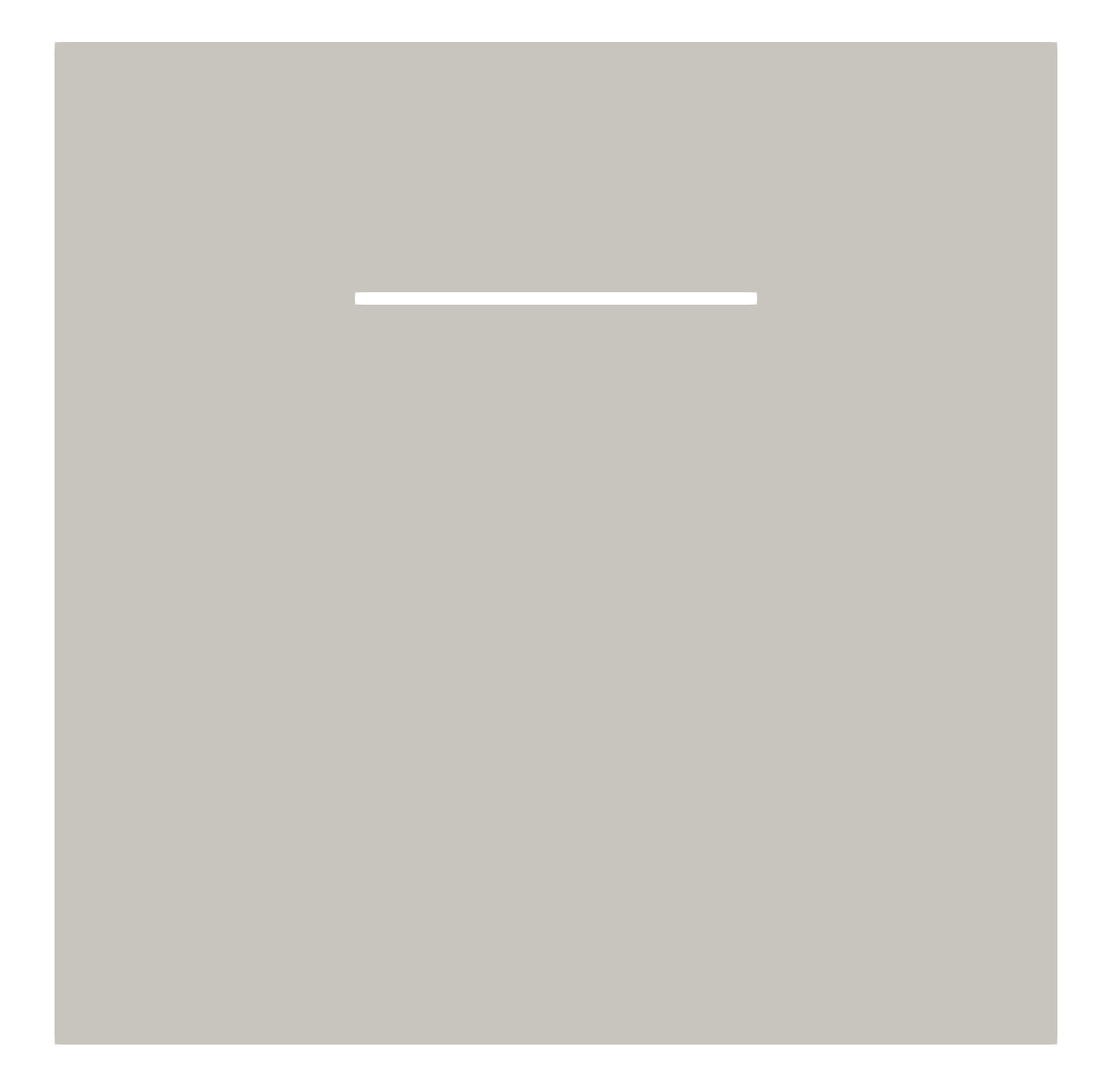} %.365
\caption{{\it Incompatible origami with cuts}: The domain is the same as in Fig.~\ref{fig:set-up-lifted-sq} (left), but with cuts on the inner square boundary. Domains with two different cuts are displayed in this image.}
\label{fig:cuts-domain}
\end{figure}

We display in Fig.\ref{fig:cuts-solutions} the computed solutions for two scenarios, one with cuts added to all four sides of the inner square and another with a single cut on one side. We observe that the presence of cuts releases the wrinkles across the interfaces in the inner region. The inner square now experiences fewer constraints and lifts up more prominently to accommodate the incompatibility, compared to the experiments in Section \ref{sec:lifted-sq-no-creases}. Additionally, the compression in the outer region across the interfaces, characterized by the contraction ratio $\lambda$, gives rise to waves near the interfaces in the outer region. Interestingly, the case with a single cut appears to exhibit larger wave amplitudes compared to the case with four cuts. 
%This difference may arise due to the longer length of mismatch on the sides of the inner square in the single-cut scenario, coupled with its reduced freedom to accommodate the mismatch.
\shuo{The plausible explanations for this phenomenon are as follows: 1. In the single-cut scenario, the mismatch on the sides of the inner square is greater compared to the four-cuts case; 2. The material with a single cut has reduced flexibility (i.e., fewer cuts) to accommodate the mismatch.}

The final energies obtained for four cuts, one cut, and no cut are, respectively
\[
E_h[\vy_h^{\infty}] \, = \, 0.0173, \,\, 0.0269, \,\, 0.0303;
\]
the metric defect $e^1_h[\vy_h^{\infty}]$ exhibits a similar trend, indicating that configurations with larger mismatch lengths face greater difficulty in satisfying the metric $g$. The decrease in energy and metric defect as the number of cuts increases confirms the relaxation effect achieved by introducing cuts to the domain. This observation further supports our earlier conclusion of Section \ref{sec:lifted-sq} regarding the non-existence of isometric immersions with the desired regularity for incompatible origami.

\begin{figure}[htbp]
  \includegraphics[width=.35\textwidth]{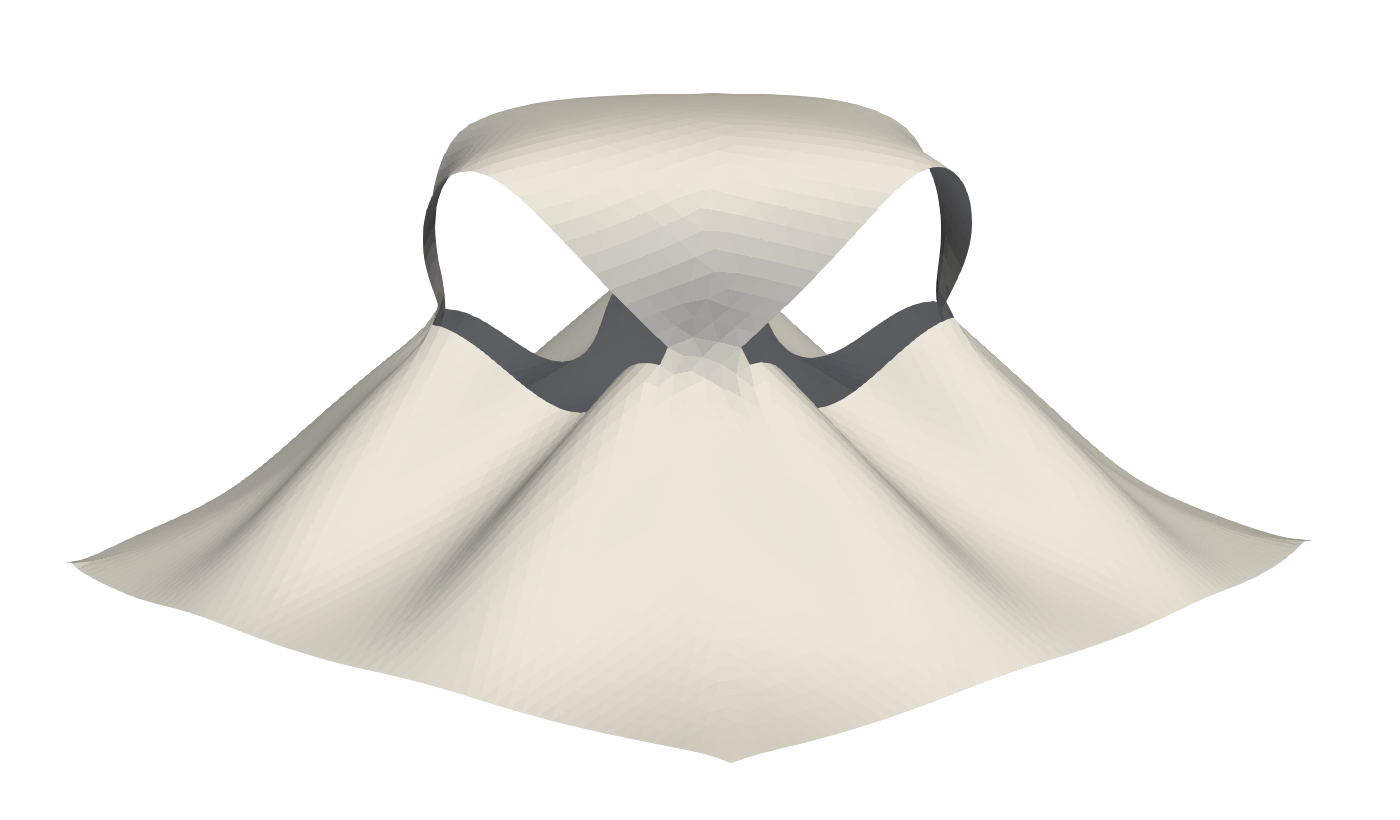}
  \includegraphics[width=.37\textwidth]{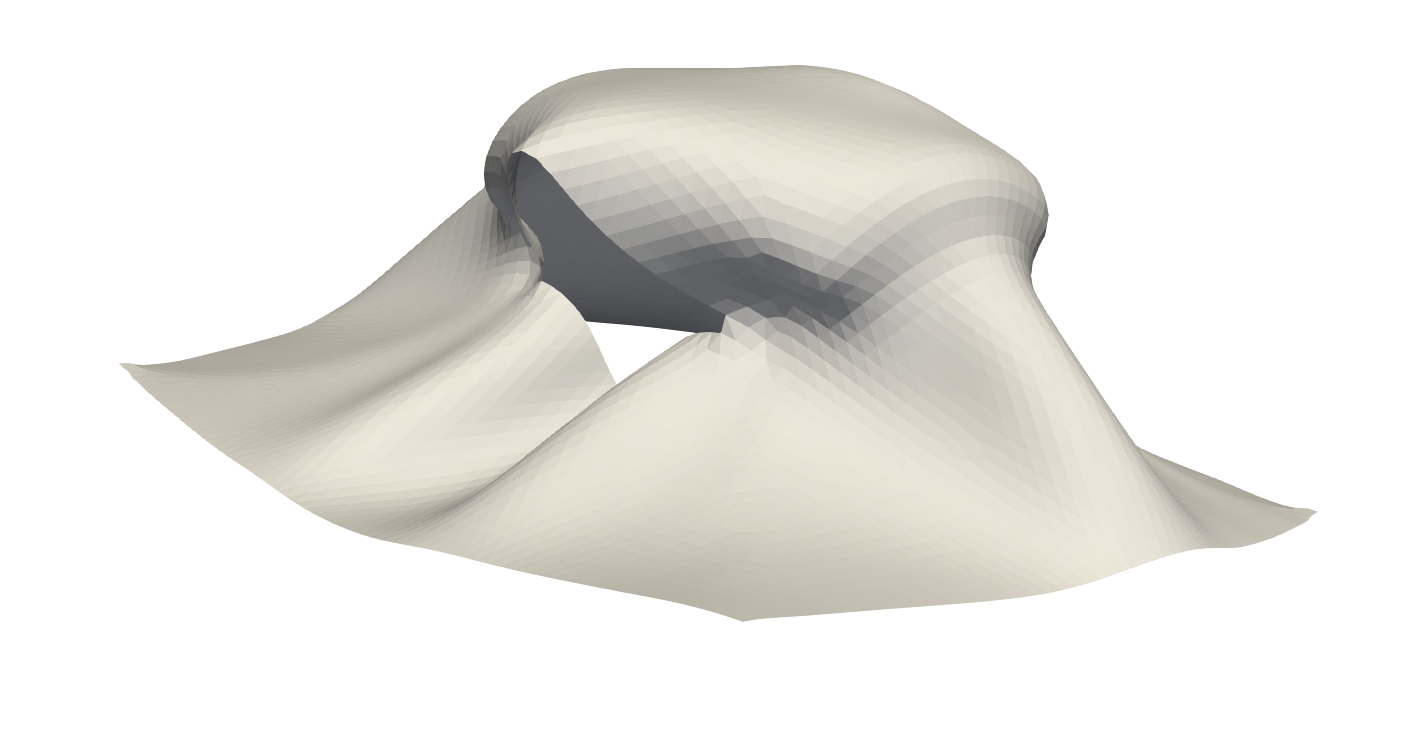}
   \includegraphics[width=.25\textwidth]{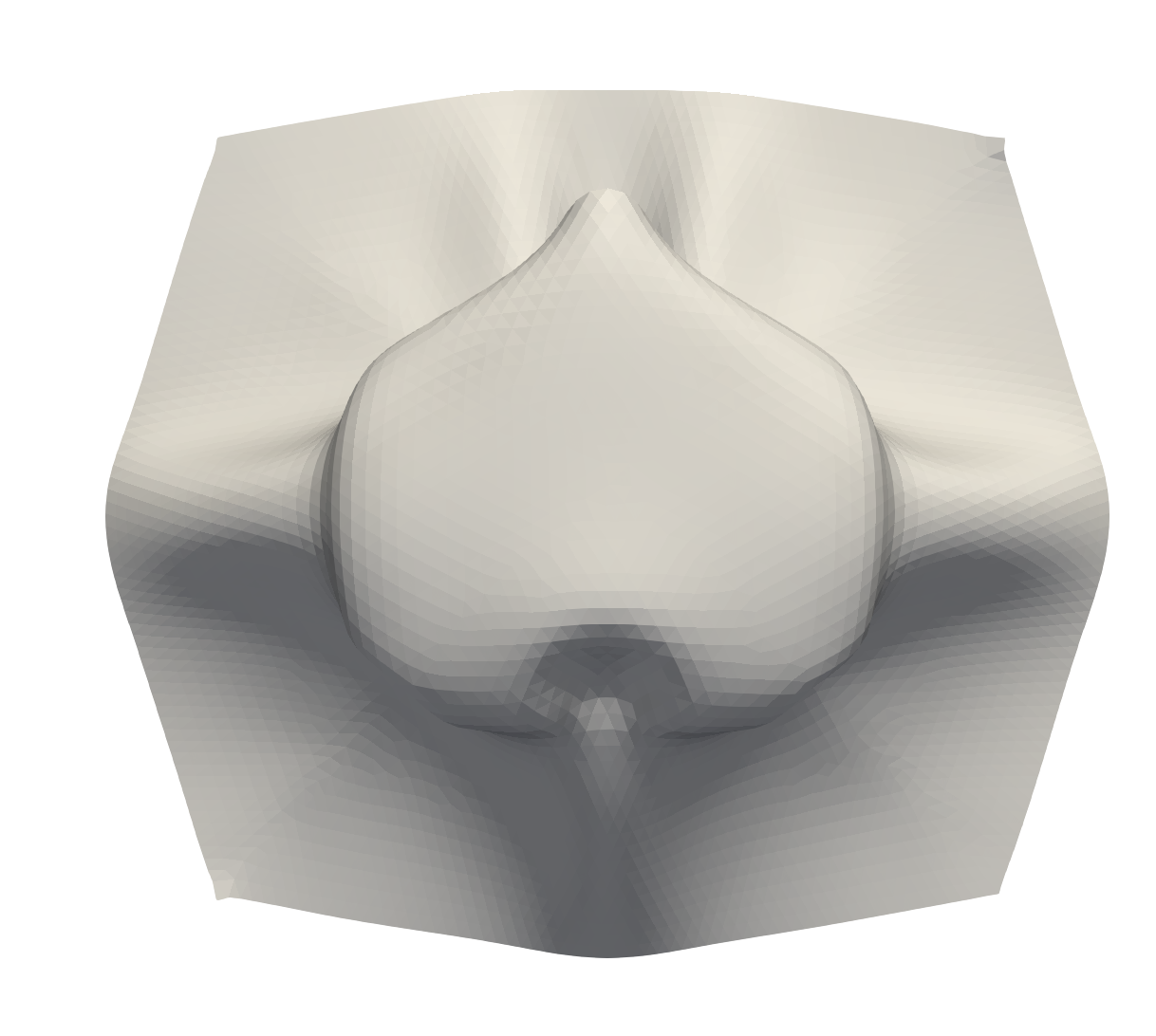}
\caption{{\it Lifted squares with cuts.} Left: computed solution with cuts \lucas{on all $4$ sides of the inner square.} Middle and Right: computed solution with cuts on one side of the inner square, with different angles of viewing. Compared to Figs.~\ref{fig:set-up-lifted-sq-2} and \ref{fig:lifted-square-exp-compare}, cuts relax wrinkle formation by allowing incompatible length changes.}
\label{fig:cuts-solutions}
\end{figure}

%-------------------------------------------------------------------------------------
\section{Conclusions}
%-------------------------------------------------------------------------------------
This paper deals with a membrane model of nematic liquid crystal polymer networks (LCNs). The two main threads in this paper are asymptotics and computation. Our main contributions as are follows.
\begin{itemize}
\item
\shuo{{\it Asymptotic modeling and profiles}.} We use formal asymptotics, in the spirit of \cite{ozenda2020blend}, to derive a reduced model of LCNs  involving the nonconvex stretching energy \eqref{eq:stretching-energy} without inextensibility constraint. We prove crucial properties of the reduced model such as the characterization of zero energy states in terms of deformations that satisfy a metric constraint and an energy gap relative to \cite{ozenda2020blend}. We utilize a novel formal asymptotic method to explore shapes of high-order degree defects.
\item
 \shuo{{\it Numerics: discretization}.} We propose a finite element method for the stretching energy consisting of continuous piecewise affine functions over a shape regular 2d mesh made of triangles. We augment the stretching energy with a regularization term involving a discrete $H^2$-norm. The latter acts as a selection mechanism to prevent oscillations of the solution. \shuo{The minimizers of the discrete energy converge under a realistic regularity assumption that permits creases, resulting in a quadratic scaling of the discrete energy. This result is stated in Theorem \ref{T:convergence}, which we prove in our companion paper \cite{bouck2022NA}.} 
%Minimizers of the discrete energy converge under a realistic regularity assumption that allows for creases and leads to a quadratic scaling of the discrete energy and report this in Theorem \ref{T:convergence}, which we prove in our companion paper \cite{bouck2022NA}. 
We employ an implicit nonlinear gradient flow to minimize the discrete energy, and a Newton method to solve each sub-problem. 
\item
  \shuo{{\it Numerics: computations}.} We present numerous computations highlighting many features of the discrete membrane model. Simulations included shapes coming from LCNs with point defects, nonisometric compatible and incompatible origami. We embark on a substantial assessment of the novel computations of incompatible origami, thoroughly exploring their intricacies, implications in modeling and analysis, predictive power and connection to ongoing lab experiments.
\end{itemize}
%
%In the companion paper \cite{bouck2022NA} we discuss additional properties of the model, such as the lack of weak lower semi-continuity. We also show in \cite{bouck2022NA} convergence of the finite element solution under a realistic regularity assumption that allows for creases and leads to a quadratic scaling of the discrete energy. We report this in Theorem \ref{T:convergence}.
%
%-------------------------------------------------------------------------------------
\section*{Acknowledgements}\label{sec:acknowledgements}
The authors thank Grant Bauman and Timothy J.\ White of the University of Colorado, Boulder, who generously shared their experimental data and engaged in insightful discussions. Additionally, we express our appreciation to Ian Tobasco from the University of Illinois, Chicago, for his fruitful conversations, bringing our attention to the reference \cite{padilla2023preparation}, and suggesting the visualization of pointwise metric error.

Lucas Bouck was supported by the NSF grant DGE-1840340. Ricardo H. Nochetto and Shuo Yang were partially supported by NSF grant DMS-1908267. \shuo{Shuo Yang was also partially supported by National Key R\&D Program of China (Grant No. 2021YFA0719200).}

\bibliographystyle{unsrt}
\bibliography{LCE_ref}

%%%%%%%%%%%%%%%%%%%%%%%%%%%%%%%%%%%%%%%%%%%%%%%%%%%%%%%%%%%%%%%%%%%%%%%%%%%%%%%%%%%%%%%%%%%%
\end{document}